\documentclass[a4paper, 11pt]{article}
\usepackage [a4paper,left=2.5cm,bottom=2.5cm,right=2.5cm,top=2.5cm]{geometry}
\usepackage[english]{babel}
\usepackage{amssymb}
\usepackage{amsmath,amsthm, dsfont}
\usepackage{subcaption}
\usepackage{tabularx,array}
\usepackage{graphicx, url}
\usepackage{enumitem}
\usepackage[normalem]{ulem}

\theoremstyle{plain}
\newtheorem{remark}{Remark}
\newtheorem{ass}{Assumption}
\newtheorem{theorem}{Theorem}[section]
\newtheorem{lemma}[theorem]{Lemma}

\newtheorem{proposition}[theorem]{Proposition}

\newcommand{\keywords}[1]{\par\addvspace\baselineskip
\noindent\enspace\ignorespaces#1}
\usepackage{color}
\usepackage[dvipsnames]{xcolor}
\newcommand{\modch}{\color{black}}
\newcommand{\modar}{\color{black}}

\newcommand{\modchi}{\color{blue}}

\newcommand{\one}{\mathds{1}}

\newcommand{\E}{\mathbb{E}}
\newcommand{\R}{\mathbb{R}}
\newcommand{\w}{\widehat}

\renewcommand{\P}{\mathbb{P}}

\newcommand{\pen}{\text{pen}}
\newcommand{\cM}{\mathcal{M}}
\newcommand{\abs}[1]{\lvert #1 \rvert}

\title{On the nonparametric inference 
of coefficients\\ of self-exciting jump-diffusion
}
\author{Chiara Amorino$^{(1)}$, Charlotte Dion-Blanc$^{(2)}$,  Arnaud Gloter$^{(3)}$, Sarah Lemler$^{(4)}$}

\date{}

\begin{document}
\maketitle

\footnote{
\begin{itemize}
    \item[$^{(1)}$] Unit\'e de Recherche en Math\'ematiques, Universit\'e du Luxembourg, \url{chiara.amorino@uni.lu}. \\
    Chiara Amorino gratefully acknowledges financial support of ERC Consolidator Grant 815703 “STAMFORD: Statistical Methods for High Dimensional Diffusions”.
    \item[$^{(2)}$] LPSM, Sorbonne Université 75005 Paris UMR CNRS 8001
    \url{charlotte.dion_blanc@sorbonne-universite.fr.}
    \item[$^{(3)}$] Laboratoire de Math\'ematiques et Mod\'elisation d'Evry, CNRS, Univ Evry, Universit\'e Paris-Saclay, 91037, Evry, France \url{arnaud.gloter@univ-evry.fr}
    \item[$^{(4)}$] Universit\'e Paris-Saclay, \'Ecole CentraleSup\'elec, MICS Laboratory, France,\\ \url{sarah.lemler@centralesupelec.fr}
\end{itemize}
}

\begin{abstract}
In this paper, we consider a one-dimensional diffusion process with jumps driven by a Hawkes process.
We are interested in the estimations of the volatility function and of the jump function from discrete high-frequency observations in a long time horizon which remained an open question until now. First, we propose to estimate the volatility coefficient. For that, we introduce a truncation function in our estimation procedure that allows us to take into account the jumps of the process and estimate the volatility function on a linear subspace of $L^2(A)$ where $A$ is a compact interval of $\R$. We obtain a bound for the empirical risk of the volatility estimator, ensuring its consistency, and then we study an adaptive estimator w.r.t. the regularity. Then, we define an estimator of a sum between the volatility and the jump coefficient modified with the conditional expectation of the intensity of the jumps. We also establish a bound for the empirical risk for the non-adaptive estimators of this sum, the convergence rate up to the regularity of the true function, and an oracle inequality for the final adaptive estimator.

Finally, we give a methodology to recover the jump function in some applications.
We conduct a simulation study to measure our estimators' accuracy in practice and discuss the possibility of recovering the jump function from our estimation procedure.

\keywords{Jump diffusion, Hawkes process, Volatility estimation, Nonparametric, Adaptation\\
AMS: 62G05, 60G55}
\end{abstract}

\section{Introduction}

The present work focuses on the jump-diffusion process introduced in \cite{DLL}.
It is defined as the solution of the following equation
\begin{eqnarray}\label{eq:model}
 dX_t =b(X_t)  dt + \sigma(X_t)  dW_{t} + a(X_{t^-}) \sum_{j =1}^M dN^{(j)} _t,
\end{eqnarray}
where $X_{t^-}$ denotes the process of left limits, $N=(N^{(1)}, \ldots, N^{(M)})$ is a $M$-dimensional Hawkes process with  intensity function $\lambda$ and $W$ is the standard Brownian motion independent of $N$. Some probabilistic results have been established for this model in \cite{DLL}, such as the ergodicity and the  $\beta-$mixing. A second work has then been conducted to estimate the drift function of the model using a model selection procedure and upper bounds on the risk of this adaptive estimator have been established in \cite{Dion Lemler} in the high frequency observations context. 

In this work, we are interested in estimating the volatility function $\sigma^2$ and the jump function $a$.
The jumps in this process make estimating these two functions difficult. We assume that discrete observations of a $X$ are available at high frequency and on a large time interval.  


\subsection{Motivation and state of the art}

Let us notice first that this model has practical relevance thinking of continuous phenomenon impacted by an exterior event, with auto-excitation structure.
For example, one can think of the interest rate model (see \cite{gomez}) in insurance; then, in neurosciences of the evolution of the membrane potential
impacted by the signals of the other neurons around it (see \cite{Dion Lemler}). Indeed, it is common to describe the \textit{spike train} of a neuron through a Hawkes process which models the auto-excitation of the phenomenon: for a specific type of neurons, when it spikes once, the probability that it will spike again increases. 
Finally, referring to \cite{bacryfinance} for a complete review on Hawkes process in finance, the reader can see the considered model as a generalization of the so-called mutually-exciting-jump diffusion proposed in \cite{aitsahalia} to study an asset price evolution.
This process generalizes Poisson jumps (or L\'evy jumps, which have independent increments) with auto-exciting jumps and is more tractable than jumps driven by L\'evy process.

Nonparametric estimation of coefficients of stochastic differential equations from the observation of a discrete path is a challenge studied a lot in literature. From a frequentist point of view in the high-frequency context, one can cite \cite{hoffmann, CGCR} and in bayesian, one recently in \cite{kweku}. 
Nevertheless, the purpose of this article falls more under the scope of statistics for stochastic processes with jumps. 
The literature for the diffusion with jumps from a pure centered L\'evy process is large. For example one can refer to \cite{jacod}, \cite{reiss} and \cite{Schmisser}.

The first goal of this work is to estimate the volatility coefficient $\sigma^2$. As it is well known, in the presence of jumps, the approximate quadratic variation based on the squared increments of $X$ no longer converges to the integrated volatility. As in \cite{mancinireno}, we base the approach on truncated quadratic variation to estimate the coefficient $\sigma^2$.
The structure of the jumps here is very different from the one induced by the pure-jump L\'evy-process. Indeed, the increments are not independent, and this implies the necessity to develop a proper methodology as the one presented hereafter. 

Secondly, we want to to find a way to approximate the coefficient $a$. It is important to note that, as presented in \cite{Schmisser}, in the classical jump-diffusion framework (where a L\'evy process is used instead of the Hawkes process for $M=1$), it is possible to obtain an estimator for the function $\sigma^2+a^2$ by considering the quadratic increments (without truncation) of the process.
This is no longer the case here due to the form of the intensity function of the Hawkes process. Indeed, we recover a more complicated function to be estimated, as explained in the following. 

\subsection{Main contribution}
The estimations of the coefficients in Model (\ref{eq:model}) are challenging in the sense that we have to take into account the jumps of the Hawkes process. Statistical inference for the volatility and for the jump function in a jump-diffusion model with jumps driven by a Hawkes process has never been studied before. As for the estimation of the drift in \cite{Dion Lemler}, we assume that the coupled process $(X,\lambda)$ is ergodic, stationary, and exponentially $\beta-$mixing. Besides, in this article we obtain that the projection on $X$ of the invariant measure of the process has a density which is lower and upper bounded on compact sets. This property is useful to lead studies of convergence rates for nonparametric estimators since it gives equivalence between empirical and continuous norms.
To estimate the volatility in a nonparametric way, as in \cite{Unbiased} we consider a truncation of the increments of the quadratic variation of $X$ that allows judging if a jump occurred or not in a time interval. We estimate $\sigma^2$ on a collection of subspaces of $L^2$ by minimizing a least-squares contrast over each model, and we establish for the obtained estimators a bound on the risk. We give the convergence rates of these estimators depending on the regularity of the true volatility function. Then, we propose a selection model procedure through a penalized criteria, we obtain non-asymptotic oracle-type inequality for the final estimator that guarantees its theoretical performance.

In the second part of this work, we are interested in the estimation of the jump function 
As it has been said before, it is not possible to recover directly the jump function $a$ from the quadratic increments of $X$, and what appears naturally is the sum of the volatility and of the product of the square of the jump function and the jump intensity. The jump intensity is hard to control properly, and it is unobserved. To overcome such a problem, we introduce the conditional expectation of the intensity given the observation of $X$, which leads us to estimate the sum of the volatility and of the product between $a^2$ and the conditional expectation of the jump intensity given $X$. We lead a penalized minimum contrast estimation procedure again, and we establish a non-asymptotic oracle inequality for the adaptive estimator. The achieved rates of convergence are similar to the ones obtained in the Lévy jump-diffusion context  in \cite{Schmisser}.
Both adaptive estimators are studied using Talagrand's concentration inequalities. 

We then discuss how we can recover $a$, as a quotient in which we plug the estimators of $\sigma^2$ and $g:=\sigma^2+a^2\times f$, where $f$ is the conditional expectation of the jump intensity that we do not know in practice. We propose to estimate $f$ using a Nadaraya-Watson estimator. We show that the risk of the estimator of $a$ cumulates the errors coming from the estimation of the three functions $\sigma^2$, $g$ and the conditional expectation of the jump intensity, which shows how hard it is to estimate $a$ correctly. 

Finally, we have conducted a simulation study to observe the behavior of our estimators in practice. We compare the empirical risks of our estimators to the risks of the oracle estimator to which we have access in a simulation study (they correspond to the estimator in the collection of models, which minimizes the empirical error). We show that we can recover rather well the volatility $\sigma^2$ and $g$ from our procedure, but it is harder to recover the jump function $a$. 

\subsection{Plan of the paper}

The model is described in Section \ref{sec:framework}, some assumptions on the model are discussed and we give  properties on the process $(X_t,\lambda_t)$. In Section \ref{sec:volatility} we present the adaptive estimation procedure for the volatility $\sigma^2$ and obtain the consistency and the convergence rate.
 Section \ref{S: volatility + jumps} is devoted to the estimation of $\sigma^2+a^2\times f$, where $f$ is the expectation of the jump intensity $\lambda$ given $X$. In this section, we return to the reason for estimating this function, we detail the estimation procedure and establish bounds for the risks of the non-adaptive estimator and of the adaptive estimator in the regularity. The estimation of the jump coefficient $a$ is discussed in Section \ref{sec:a}. In Section \ref{sec:numerical} we have conducted a simulation study and give a little conclusion and some perspective to this work in Section \ref{sec:discussion}. Finally, the proofs of the main results are detailed in Section \ref{sec:proofs} and the technical results are proved in Appendix \ref{sec:appendix}.

\section{Framework and Assumptions}
\label{sec:framework}

\subsection{The Hawkes process}
Let $(\Omega, \mathcal{F}, \mathbb{P})$ be a probability space. We define the Hawkes process for $t \ge 0$ through stochastic intensity representation. We introduce the $M$-dimensional point process 
$N_t := (N_t^{(1)},\ldots , N_t^{(M)} )$ and its intensity $\lambda$ is a vector of non-negative stochastic intensity functions given by a collection of baseline
intensities. It consists in positive constants $\zeta_j$, for $j \in \left \{ 1,\ldots , M \right \}$, and in $M \times M$ interaction functions $h_{i,j} : \mathbb{R}^+ \rightarrow \mathbb{R}^+$, which are measurable functions ($i, j \in \left \{ 1,\ldots , M \right \}$).
For $i \in \left \{ 1,\ldots , M \right \}$ we also introduce $n^{(i)}$, a discrete point measure on $\mathbb{R}^-$ satisfying
$$\int_{\mathbb{R}^-} h_{i,j}(t - s) n^{(i)}(ds) < \infty \qquad \mbox{for all }t \ge 0. $$ 
They can be interpreted as initial condition of the process. The linear Hawkes process with initial condition $n^{(i)}$ and with parameters $(\zeta_i, h_{i,j})_{1 \le i,j \le M}$ is a multivariate counting process $(N_t)_{t \ge 0}$. It is such that for all $i \neq j $, $\mathbb{P}$ - almost surely, $N^{(i)}$ and $N^{(j)}$ never jump simultaneously. Moreover, for any $i \in \left \{ 1,\ldots , M \right \}$, the compensator of $N^{(i)}$ is given by $\Lambda_t
^{(i)} := \int_0^t \lambda_s^{(i)} ds$, where $\lambda$ is the intensity process of the counting process $N$ and satisfies the following equation:
$$\lambda_t^{(i)} = \zeta_i + \sum_{j = 1}^M \int_0^{t^-} h_{i,j} (t - u) dN_u^{(j)} + \sum_{j = 1}^M \int_{- \infty}^0 h_{i,j} (t - u) dn_u^{(j)}.$$
We remark that $N_t^{(j)}$ is the cumulative number of events in the j-th component at time t while $dN^{(j)}_t$ represents the number of points in the time increment $[t, t + dt]$.
We define $\tilde{N}_t := N_t - \Lambda_t$ and $\bar{\mathcal{F}}_t := \sigma (N_s, 0 \le s \le t) $ the history of the counting process $N$ (see Daley and Vere - Jones \cite{Daley2007}). The intensity process $\lambda= (\lambda^{(1)}, \ldots, \lambda^{(M)})$ of the counting process $N$ is the $\bar{\mathcal{F}}_t$-predictable process that makes $\tilde{N}_t$ a $\bar{\mathcal{F}}_t$-local martingale. 

Requiring that the functions $h_{i,j}$ are locally integrable, it is possible to prove with standard arguments the existence of a process $(N_t^{(j)})_{t \ge 0}$ (see for example \cite{Delattre_Hoffmann}). We denote as $\zeta_j$ the exogenous intensity of the process and as $(T^{(j)}_k)_{k \ge 1}$ the non-decreasing jump times of the process $N^{(j)}$. 

We interpret the interaction functions $h_{i,j}$ (also called kernel function or transfer function) as the influence of the past activity of subject $i$ on the subject $j$, while the parameter $\zeta_j>0$ is the spontaneous rate and is used to take into account all the unobserved signals.
In the sequel we focus on the exponential kernel functions defined by
$$h_{i,j}: \mathbb{R}^+ \rightarrow \mathbb{R}^+, \quad h_{i,j} (t) = c_{ij} e^{- \alpha t}, \quad \alpha > 0, \quad c_{ij} > 0, \quad 1 \le i,j\le M.$$
With this choice of $h_{i,j}$ the conditional intensity process $(\lambda_t)$ is then Markovian. In this case we can introduce the auxiliary Markov process $Y = Y^{(i j)}$:
$$Y^{(i j)}_t = c_{i,j} \int_0^t e^{- \alpha(t- s)} d N_s^{(j)} + c_{i,j} \int_{- \infty}^0 e^{- \alpha(t- s)} d n_s^{(j)}, \quad 1 \le i,j \le M. $$
The intensity can be expressed in terms of sums of these Markovian processes that is, for all $1 \le i \le M$
$$\lambda_t^{(i)} = f_i \left(\sum_{j = 1}^M Y_{t^-}^{(i j)}\right), \qquad \mbox{with } f_i(x) = \zeta_i + x.$$
We remark that all the point processes $N^{(j)}$ behave as homogeneous Poisson processes with constant intensity $\zeta_j$, before the first occurrence. Then, as soon as the first occurrence appears for a particular $N^{(i)}$, it affects all the process increasing the conditional intensity through the interaction functions $h_{i,j}$.
\\ 

Let us emphasized that from the work \cite{DLL}, it is possible to not assume the positiveness of the coefficients $c_{i,j}$, taking then $f_i(x)=(\zeta_i + x)_+$. This is particularly important for the neuronal applications where the neurons can have \textit{excitatory or inhibitory} behavior.

\subsection{Model Assumptions}
In this work we consider the following jump-diffusion model.
We write the process as $M+1$ stochastic differential equations:
\begin{equation}
\begin{cases}
d\lambda^{(i)}_t = - \alpha (\lambda^{(i)}_t - \zeta_i) dt + \sum_{j = 1}^M c_{i,j} dN^{(j)}_t, \quad i = 1,\ldots , M \\
dX_t = b(X_t) dt + \sigma(X_t) dW_t + a(X_{t^-}) \sum_{j = 1}^M dN^{(j)}_t,
\end{cases}
\label{eq: model}
\end{equation}
with $\lambda^{(j)}_0$ and $X_0$ random variables independent of the others. In particular, $(\lambda^{(1)}_t,\ldots , \lambda^{(M)}_t, X_t)$ is a Markovian process for the general filtration
$$\mathcal{F}_t := \sigma(W_s, N_s^{(j)}, \quad j=1,\ldots , M, \quad 0 \le s \le t).$$
We remark that the process $N_t^{(j)}$ has jumps of size 1.
We aim at estimating, in a non-parametric way, the volatility $\sigma$ and the jump coefficient $a$ starting from a discrete observation of the process $X$.
The process $X$ is indeed observed at high frequency on the time interval $[0, T_n]$. For $0 = t_0 \le t_1 \le\ldots \le t_n = T_n$, the observations are denoted as $X_{t_i}$. We define $\Delta_{n,i} := t_{i + 1} - t_i$ and $\Delta_n := \sup_{i =0,\ldots , n} \Delta_{n,i}$. We are here assuming that $\Delta_n \rightarrow 0$ and $n \Delta_n \rightarrow \infty$, for $n \rightarrow \infty$. We suppose that there exists $c_1$, $c_2$ such that, $\forall i \in \left \{ 0,\ldots , n-1 \right \}$, $c_1 \Delta_{min} \le \Delta_{n,i} \le c_2 \Delta_n$. We remark we are considering a general case, where the discretization step is not necessarily uniform. However, in case where the discretization step is uniform we clearly have $\Delta_n = \Delta_{n,i}$ for any $i \in \left \{ 0, ... , n-1 \right \}$ (which implies that the condition here above is clearly respected with $c_1 = c_2 = 1$) and the time horizon becomes $T_n = n \Delta_n$.
 Furthermore, we require that
\begin{equation}
\log n = o(\sqrt{n \Delta_n}).
\label{eq: cond delta}
\end{equation}
The size parameter $M$ is fixed and finite all along, and asymptotic properties are obtained when
$T \rightarrow \infty$. \\
Requiring that the size of the discretization step is always the same, as we do asking that the maximal and minimal discretization steps differ only on a constant, is a pretty classical assumption in our framework. On the other side, the step conditions \eqref{eq: cond delta} is more technical. This condition is replaced with a stronger one to obtain Theorem \ref{th: estim both adaptive} (see Assumption \ref{ass: step} and the discussion below).

\begin{ass}[Assumptions on the coefficients of $X$]\label{ass: X}
\begin{enumerate}
    \item The coefficients $b$ and $\sigma$ are of class $\mathcal{C}^2$ and there exists a  positive constants $c$ such that, for all $x \in \mathbb{R}$, $|b'(x)| + |\sigma'(x)| + |a'(x)|  \le c$. 
    \item There exist positive constants $a_1$ and $\sigma_1$ such that $|a(x)| < a_1$ and $0 < \sigma^2(x) < \sigma^2_1$ for all $x\in \mathbb{R}$.
    \item There exist positive constants $c', q$ such that, for all $x \in \mathbb{R}$, $|b''(x)| + |\sigma''(x)| \le c'(1 + |x|^q )$.
    \item There exist $d \ge 0$ and $r > 0$ such that, for all $x$ satisfying $|x| > r$, we have $x b(x) \le - d x^2$.
\end{enumerate} 
\end{ass}

We remark that, as a consequence of Assumption 1, item 1, the coefficients $b$, $\sigma$ and $a$ are globally Lipschitz continuous.

 The first three assumptions ensure the existence of a unique solution $X$ (as proven in \cite{DLL} Proposition 2.3). The last assumption is introduced to study the longtime behavior of $X$ and to ensure its ergodicity (see \cite{DLL}). Note that the assumption on $a$ can be relaxed (see \cite{DLL}).

\begin{ass}[Assumptions on the kernels]\label{ass: kernel}
 \begin{enumerate}
    \item Let $H$ be a matrix such that \\$H_{i,j} := \int_0^{\infty} h_{i,j} (t) dt = {c_{ij}}/{\alpha}$, for $1 \le i,j\le M$. The matrix $H$ has a spectral radius smaller than 1.
    \item We suppose that $\sum_{j = 1}^M \zeta_j > 0$ and that the matrix $H$ is invertible. 
\end{enumerate} 
\end{ass}

The first point of the Assumption \ref{ass: kernel} here above implies that the process $(N_t)$ admits a version with stationary increments (see \cite{Bremaud}). In the sequel, we always will consider such an assumption satisfied. The process $(N_t)$ corresponds to the asymptotic limit and $(\lambda_t)$ is a stationary process. The second point of A2 is needed to ensure the positive Harris recurrence of the couple $(X_t, \lambda_t)$. A discussion about it can be found in Section 2.3 of \cite{Dion Lemler}.

\subsection{Ergodicity and moments}{\label{S: ergodicity}}
In the sequel, we repeatedly use the ergodic properties of the process $Z_t := (X_t, \lambda_t)$. From Theorem 3.6 in \cite{DLL} we know that, under Assumptions \ref{ass: X} and \ref{ass: kernel}, the process $(X_t, \lambda_t)_{t \ge 0}$ is positive Harris recurrent with unique invariant measure $\pi(dx)$. Moreover, in \cite{DLL}, the Foster-Lyapunov condition in the exponential frame implies that, for all $t \ge 0$, $\mathbb{E}[X_t
^4] < \infty$ (see Proposition 3.4). In the sequel we need $X$ to have arbitrarily big moments and, therefore, we propose a modified Lyapunov function. In particular, following the ideas in \cite{DLL}, we take $V : \mathbb{R} \times \mathbb{R}^{M \times M} \rightarrow \mathbb{R}_+$ such that
\begin{equation}
V(x,y):= |x|^m + e^{\sum_{i,j} m_{i j} | y^{(i j)}| }, 
\label{eq: potential}
\end{equation}
where $m \ge 2$ is a constant arbitrarily big and $m_{i j} := \frac{k_i}{\alpha}$, being $k \in \mathbb{R}^M_+$ a left eigenvector of $H$, which exists and has non-negative components under our Assumption \ref{ass: kernel} (see \cite{DLL} below Assumption 3.3). \\
We now introduce the generator of the process $\tilde{Z}_t := (X_t, Y_t)$, defined for sufficiently smooth test function $g$ by
\begin{equation}
A^{\tilde{Z}} g (x,y) = - \alpha \sum_{i,j=1}^M y^{(i j)} \partial_{y^{(i j)}} g (x,y) + \partial_x g(x,y) b(x) + \frac{1}{2} \sigma^2(x) \partial^2_x g(x,y)
\label{eq: def generator}
\end{equation}
$$ + \sum_{j = 1}^M f_j \left(\sum_{k = 1}^M y^{(j k)} \right) [g(x + a(x), y + \Delta_j) - g (x,y)],$$
with $(\Delta_j)^{(i l)} = c_{i, j} \one_{j = l}$, for all $1 \le i, l \le M$. Then, the following proposition holds true.
\begin{proposition}
Suppose that Assumptions 1 and 2 hold true. Let V be as in \eqref{eq: potential}. Then, there exist positive constants $d_1$ and $d_2$ such that the following Foster-Lyapunov type drift condition holds:
$$A^{\tilde{Z}} V \le d_1 - d_2 V.$$
Moreover, both $X$ and $\lambda$ have bounded moments of any order.
\label{prop: Lyapounov}
\end{proposition}
Proposition \ref{prop: Lyapounov} is proven in the Appendix. Let us now add the third assumption.

\begin{ass}\label{ass: statio}
$(X_0, \lambda_0)$ has probability $\pi$.
\end{ass}

Then, the process $(X_t, \lambda_t)_{t \ge 0}$ is in its stationary regime.

We recall that the process $Z$ is called $\beta$ - mixing if $\beta_Z (t) = o(1)$ for $t \rightarrow \infty$ and exponentially $\beta$ - mixing if there exists a constant $\gamma_1 > 0$ such that $\beta_Z (t) = O(e^{- \gamma_1 t})$ for $t \rightarrow \infty$, where $\beta_Z$ is the $\beta$ - mixing coefficient of the process $Z$ as defined for a Markov process $Z$ with transition semigroup $(P_t)_{t \in \mathbb{R}^+}$, by
\begin{equation}
\beta_Z (t) :=  \int_{\mathbb{R} \times \mathbb{R}^M} \left \| P_t (z, .) - \pi \right \| \pi (d z),
\label{eq: def beta}    
\end{equation}
where $\left \| \lambda \right \|$ stands for the total variation norm of a signed measure $\lambda$.

Moreover, it is 
$$\beta_X (t) :=  \int_{\mathbb{R} \times \mathbb{R}^M} \left \| P_t^1 (z, .) - \pi^X \right \| \pi (d z),$$
where $P_t^1 (z, .)$ is the projection on $X$ of $P_t (z, .)$ such that $P_t^1 (z, dx) := P_t (z, dx \times \mathbb{R}^M)$ and $\pi^X(dx):= \pi(dx \times \mathbb{R}^M)$ is the projection of $\pi$ on the coordinate $X$
(which exists, see Theorem 2.3 in \cite{Dion Lemler} and proof in \cite{DLL}). Then, according to Theorem 4.9 in \cite{DLL}, under A1-A3 the process $Z_t := (X_t, \lambda_t)$ is exponentially $\beta$-mixing and there exist some constant $K, \gamma > 0$ such that $$\beta_X (t) \le \beta_Z (t) \le K e^{- \gamma t}.$$

	Furthermore, from Proposition 3.7 in \cite{DLL}, we know that the measure $\pi^X(d x)$ admits a Lebesgue density $x\mapsto \pi^X(x)$ and it is lower bounded on each compact set of $\R$. In the following lemma we additionally prove that this density is also upper bounded on each compact set.
	\begin{lemma} 
		Assume that Assumptions 1 and 2 hold true. Then, for any compact set $K$ of $\mathbb{R}$, there exists a constant $C_K>0$ such that $\pi^X(x)\le C_K$ for all $x\in K$.	
		\label{lemma: boundpiX}
	\end{lemma}
Finally, we know that for each compact set $A \subset \R$ there exist two positive constants $\pi_0, \pi_1$ such that for $x\in A$ we have
\begin{equation}
    \label{eq:controlpi}
    0<\pi_0 \leq \pi_X(x) \leq \pi_1.
\end{equation}
Let us define the norm with respect to $\pi_X$:
$$\left \| t \right \|^2_{\pi^X} := \int_A t^2(x) \pi_X (dx).$$ 
According to Lemma \ref{lemma: boundpiX} it yields that for a deterministic function $t$,
$\pi_0 \|t\|^2 \leq \|t\|_{\pi^X} \leq \pi_1 \|t\|^2$. 

\section{Estimation procedure of the volatility function}{\label{S: volatility}}
\label{sec:volatility}
With the background introduced in the previous sections, we are now ready to estimate the volatility function to whom this section is dedicated. We remind the reader that the procedure is based on the observations 
$(X_{t_i})_{i=1, \ldots, n}.$

First of all, in Subsection \ref{S: non adaptive volatility}, we propose a non-adaptive estimator based on the squared increments of the process $X$. To do that, we decompose such increments in several terms, aimed to isolate the volatility function. Regarding the other terms, we can recognize a bias term (which we will show being small), the contribution of the Brownian part (which is centered), and the jumps' contribution. To make the latter small as well, we introduce a truncation function (see Lemma \ref{lemma: esp sauts} below). Thus, we can define a contrast function based on the truncated squared increments of $X$ and the associated estimator of the volatility. In Proposition \ref{prop: volatility}, which is the main result of this subsection, we prove a bound for the empirical risk of the volatility estimator we propose.

As the presented estimator depends on the model, in Subsection \ref{S: adaptive volatility} we introduce a fully data-driven procedure to select the best model automatically in the sense of the empirical risk. We choose the model such that it minimizes the sum between the contrast and a penalization function, as explained in \eqref{eq: m1}. In Theorem \ref{th: vol adaptive} we show that the estimator associated with the selected model realizes the best compromise between automatically the bias term and the penalty term.

\subsection{Non-adaptive estimator}{\label{S: non adaptive volatility}}

Let us consider the increments of the process $X$ as follows:
\begin{eqnarray}
X_{t_{i + 1}} - X_{t_i} &=& \int_{t_i}^{t_{i + 1}} b(X_s) ds + \int_{t_i}^{t_{i + 1}} \sigma(X_s) dW_s + \int_{t_i}^{t_{i + 1}} a(X_{s^-}) \sum_{j = 1}^M dN_s^{(j)} \nonumber\\
&=& \int_{t_i}^{t_{i + 1}} b(X_s) ds + Z_{t_i} + J_{t_i}
\label{eq: introduction notation}
\end{eqnarray}
where $Z,J$ are given in Equation \eqref{eq: ZJ}:
\begin{equation}\label{eq: ZJ}
Z_{t_i} := \int_{t_i}^{t_{i + 1}} \sigma(X_s) dW_s, \qquad  J_{t_i} := \int_{t_i}^{t_{i + 1}} a(X_{s^-}) \sum_{j = 1}^M dN_s^{(j)}.
\end{equation}
To estimate $\sigma^2$ for a diffusion process (without jumps), the idea is to consider the random
variables $T_{t_i} := \frac{1}{\Delta_n} (X_{t_{i + 1}} - X_{t_i})^2$. Following this idea, we decompose $T_{t_i}$, in order to isolate the contribution of the volatility computed in $X_{t_i}$. In particular, Equation \eqref{eq: introduction notation} yields
\begin{equation}\label{eq: Tipoursigma}
T_{t_i} = \frac{1}{\Delta_n} (X_{t_{i + 1}} - X_{t_i})^2=  \sigma^2(X_{t_i}) +  \bar{A}_{t_i} + B_{t_i} + E_{t_i},
\end{equation}
where $A,B,E$ are functions of $Z,J$:
\begin{eqnarray*} \bar{A}_{t_i} &:=&\frac{1}{\Delta_n} \left(\int_{t_i}^{t_{i + 1}} b(X_s) ds\right)^2 + \frac{2}{\Delta_n}(Z_{t_i} + J_{t_i})\int_{t_i}^{t_{i + 1}} (b(X_s) - b(X_{t_i})) ds
\\
&&+ \frac{1}{\Delta_n} \int_{t_i}^{t_{i + 1}} (\sigma^2(X_s) - \sigma^2(X_{t_i})) ds + 2 b(X_{t_i}) Z_{t_i},  \end{eqnarray*}
\begin{equation}
B_{t_i} :=  \frac{1}{\Delta_n} [Z_{t_i}^2 - \int_{t_i}^{t_{i + 1}} \sigma^2(X_s) ds];
\label{eq: def B}
\end{equation}
$$E_{t_i} := 2 b(X_{t_i}) J_{t_i} + \frac{2}{\Delta_n} Z_{t_i} J_{t_i} + \frac{1}{\Delta_n}J_{t_i}^2.$$
The term $\bar{A}_{t_i}$ is small, whereas $B_{t_i}$ is centered. In order to make $E_{t_i}$ small as well, we introduce the truncation function $\varphi_{\Delta_{n,i}^\beta}(X_{t_{i + 1}} - X_{t_i})$, for $\beta \in (0, \frac{1}{2})$. It is a version of the indicator function, such that $\varphi(\zeta) = 0$ for each $ \zeta$, with $|\zeta| \ge 2$ and $\varphi(\zeta) = 1$ for each $ \zeta $, with $ |\zeta| \le 1$. Also, we define $\varphi_{z}(.):=\varphi(. / z), z>0$. The idea is to use the size of the increment of the process  $ \Delta_i X := X_{t_{i+1}} - X_{t_i}$ in order to judge if a jump occurred or not in the interval $[t_i, t_{i + 1})$.
As it is hard for the increment of $X$ with continuous transition to overcome the threshold $\Delta_{n,i}^\beta$ for $\beta {\modchi <} \frac{1}{2}$, we can assert the presence of a jump in $[t_i, t_{i + 1})$ if $|X_{t_{i+1}} - X_{t_i}| > \Delta_{n,i}^\beta $.
Hence, we consider the random variables
$$T_{t_i} \varphi_{\Delta_{n,i}^\beta}(\Delta_i X) = \sigma^2(X_{t_i}) + \tilde{A}_{t_i} + B_{t_i} + E_{t_i}\varphi_{\Delta_{n,i}^\beta}(\Delta_i X),$$
with
$$\tilde{A}_{t_i} := \sigma^2(X_{t_i}) (\varphi_{\Delta_{n,i}^\beta}(\Delta_i X) - 1) +  \bar{A}_{t_i} \varphi_{\Delta_{n,i}^\beta}(\Delta_i X) + B_{t_i}(\varphi_{\Delta_{n,i}^\beta}(\Delta_i X) - 1).$$ 
\\
Now, the just-introduced $\tilde{A}_{t_i}$ is once again a small term, because so $ \bar{A}_{t_i}$ was and because the truncation function does not differ a lot from the indicator function, as better justified in Lemma \ref{lemma: proba varphi -1} below. \\
\\
In the sequel, the constant $c$ may change the value from line to line.
\begin{lemma}
Suppose that Assumptions 1,2,3 hold. Then, for $\beta \in (0, \frac{1}{2})$ and for any $k \ge 1$, 
$$\mathbb{E}\left[\left|\varphi_{\Delta_{n,i}^\beta}(\Delta_i X) - 1\right|^k\right]   \le c \Delta_{n,i}.$$
\label{lemma: proba varphi -1}
\end{lemma}
The proof of Lemma \ref{lemma: proba varphi -1} can be found in the Appendix. The same is for the proof of Lemma \ref{lemma: esp sauts} below, which illustrates the reason why we have introduced a truncation function. Indeed, without the presence of $\varphi$, the same Lemma would have held with just a $c \, \Delta_{n,i}$ on the right-hand side. Filtering the contribution of the jumps, we can gain an extra $\Delta_{n,i}^{\beta q}$ which, as we will see in Proposition \ref{lemma: size A B E}, will make the contribution of $E_{t_i}$ small.
\begin{lemma}
Suppose that Assumptions 1,2,3 hold. Then, for $q \ge 1$, for $\beta \in (0, \frac{1}{2})$ and for any $k \ge 1$
$$\mathbb{E}\left[|J_{t_i}|^q \varphi^k_{\Delta_{n,i}^\beta}(\Delta_i X)\right] \le c \Delta_{n,i}^{1 + \beta q}.$$
\label{lemma: esp sauts}
\end{lemma}
From Lemmas \ref{lemma: proba varphi -1} and \ref{lemma: esp sauts} here above, it is possible to prove the following proposition. Also, its proof can be found in the Appendix.
\begin{proposition}
Suppose that Assumptions 1,2,3 hold. Then, for $\beta\in(\frac{1}{4},\frac{1}{2})$,
\begin{enumerate}
\item $\forall \tilde{\varepsilon} > 0$, $\mathbb{E}[\tilde{A}_{t_i}^2] \le c \Delta_{n,i}^{1 - \tilde{\varepsilon}},  \qquad \mathbb{E}[\tilde{A}_{t_i}^4] \le c \Delta_{n,i}^{1 - \tilde{\varepsilon}};$
\item $\mathbb{E}[B_{t_i}|\mathcal{F}_{t_i}] = 0, \qquad \mathbb{E}[B^2_{t_i}|\mathcal{F}_{t_i}] \le c \sigma_1^4, \qquad \mathbb{E}[B_{t_i}^4] \le c; $
\item $\mathbb{E}[|E_{t_i}|\varphi_{\Delta_{n,i}^\beta}(\Delta_i X)] = c \Delta_{n,i}^{2 \beta}, \quad \mathbb{E}[E^2_{t_i}\varphi_{\Delta_{n,i}^\beta}(\Delta_i X)] \le c \Delta_{n,i}^{4 \beta - 1}, \quad \mathbb{E}[E_{t_i}^4\varphi_{\Delta_{n,i}^\beta}(\Delta_i X)] \le c\Delta_{n,i}^{8 \beta - 3}.$
\end{enumerate}
\label{lemma: size A B E}
\end{proposition}
In the Proposition here above, it is possible to see in detail in what terms the contribution of $\tilde{A}_{t_i}$ and of the truncation of $E_{t_i}$ are small. Moreover, an analysis of the centered Brownian term $B_{t_i}$ and its powers is proposed.

Based on these variables, we propose a nonparametric estimation procedure for the function $ \sigma^2(\cdot)$ on a {\modar compact} interval A of $\mathbb{R}$. We consider $\mathcal{S}_m$ a linear subspace of $L^2(A)$ such that $\mathcal{S}_m = \text{span} (\varphi_1,\ldots , \varphi_{D_m})$
of dimension $D_m$, where $(\varphi_i)_i$ is an orthonormal basis of $L^2(A)$. We denote
$\tilde{S}_n := \cup_{m \in \cM_n} \mathcal{S}_m$, where $\cM_n$ is a set of indexes for the model collection. The contrast function is defined by
\begin{equation}\label{eq:gamma}
\gamma_{n, M}(t):= \frac{1}{n} \sum_{i = 0}^{n - 1} (t(X_{t_i})- T_{t_i}\varphi_{\Delta_{n,i}^\beta}(\Delta_i X))^2{\modar\one_A(X_{t_i}) }
\end{equation}
with the $T_{t_i}$ given in Equation \eqref{eq: Tipoursigma}.
The associated mean squares contrast estimator is
\begin{equation}
\w{\sigma}^2_m := \arg\min_{t \in \mathcal{S}_m} \gamma_{n, M}(t).
\label{eq: def estimator sigma} 
\end{equation}
We observe that as $\w{\sigma}^2_m$ achieves the minimum, it represents the projection of our estimator on the space $\mathcal{S}_m$. {\modar The indicator function in \eqref{eq:gamma} suppresses the contribution
	of the data falling outside the compact set $A$, on which we estimate the unknown function $\sigma^2$.  However, this indicator function is introduced for convenience and could be removed without affecting the value of the argmin in \eqref{eq: def estimator sigma}, as all elements of $\mathcal{S}_m$ are compactly supported on $A$.}
The approximation spaces $\mathcal{S}_m$ have to satisfy the following properties.

\begin{ass}[Assumptions on the subspaces]\
\label{ass: subspace}

\begin{enumerate}
    \item There exists $\phi_1$ such that, for any $t \in \mathcal{S}_m$, $\left \| t\right \|_\infty^2 \le \phi_1 D_m \left \| t\right \|^2 $.
    \item Nesting condition: $(\mathcal{S}_m)_{m \in \cM_n}$ is a collection of models such that there exists a space denoted by  $\tilde{\mathcal{S}}_n$, belonging to the collection, with $\mathcal{S}_m \subset \tilde{\mathcal{S}}_n$ for all $m \in \cM_n$. We denote by $N_n$ the dimension of $ \tilde{\mathcal{S}}_n$. It implies that, $\forall m \in \cM_n$, $D_m \le N_n$.
    \item  For any positive $d$ there exists $\tilde{\varepsilon} > 0$ such that, for any $\varepsilon < \tilde{\varepsilon}$, $\sum_{m \in \cM_n}  e^{-d D_m^{1 - \varepsilon}} \le \Sigma (d)$, where $\Sigma(d)$ denotes a finite constant depending only on $d$.
\end{enumerate} 
\end{ass}
Several possible collections of spaces are available: we can consider the collection of dyadic regular piecewise polynomial spaces [DP], the trigonometric spaces [T] or the dyadic wavelet-generated spaces [W] (see Sections 2.2 and 2.3 of \cite{CGCR} for details). As we will see, in the numerical part we will consider the trigonometric spaces while all the results gathered in the theory hold true for no matter which space. However, depending on the collection we consider, we need to require a different condition on the dimension of the space, as stated in the assumption below.
\begin{ass}[Assumptions on the dimension]\
\begin{enumerate}
    \item There exists a constant $c > 0$ such that $N_n \le c \frac{\sqrt{n \Delta_n}}{\log n}$ and $\frac{N_n^3}{n} \le 1$ for the collection [T].
    \item  There exists a constant $c > 0$ such that $N_n \le c \frac{n \Delta_n}{\log^2 n}$ for collections [DP] and [W].
\end{enumerate}
\end{ass}

We now introduce the empirical norm 
$$\left \| t \right \|^2_n :=  \frac{1}{n} \sum_{i = 0}^{n - 1} t^2(X_{t_i})\one_A(X_{t_i}).$$

The main result of this section consists in a bound for $ \mathbb{E}[\left \| \w{\sigma}^2_m - \sigma^2 \right \|_n^2 ]$, which is gathered in the following proposition. Its proof can be found in Section \ref{S: proof volatility}.
\begin{proposition}
Suppose that Assumptions 1,2,3,4,5 hold and that $\beta \in (\frac{1}{4}, \frac{1}{2})$. If $\Delta_n \rightarrow 0$, $\log n = o(\sqrt{n \Delta_n})$ for $n \rightarrow\infty$, then the estimator $\w{\sigma}^2_m$ of $\sigma^2$ on A given by Equation \eqref{eq: def estimator sigma} satisfies 
\begin{equation}\label{eq: propsigma}
\E\left[\left \| \w{\sigma}^2_m - \sigma^2 \right \|_n^2 \right] \le  3 \inf_{t \in \mathcal{S}_m} \left \| t - \sigma^2 \right \|_{\pi^X}^2 +  \frac{C_1 \sigma_1^4 D_m}{n} +  C_2 \Delta_n^{4 \beta - 1} + \frac{C_3 \Delta_{n}^{0 \land  (4 \beta -\frac{3}{2})}}{n^2},
\end{equation}
with $C_1$, $C_2$ and $C_3$ positive constants.
\label{prop: volatility}
\end{proposition}
%
%

This result measures the accuracy of our estimator $\w\sigma_m^2$ for the empirical norm. The right-hand side of the Equation \eqref{eq: propsigma} is decomposed into different types of error. The first term corresponds to the bias term, which decreases with the dimension $D_m$ of the space of approximation $\mathcal{S}_m$. The second term corresponds to the variance term, i.e., the estimation error, and contrary to the bias, it increases with $D_m$. The third term comes from the discretization error and the controls obtained in Proposition \ref{lemma: size A B E}, taking into account the jumps. Then, the fourth term arise evaluating the norm $\| \w{\sigma}^2_m - \sigma^2\|_n^2 $ when $\|.\|_n$ and $\|.\|_{\pi^X}$ are not equivalent.
This inequality ensures that our estimator $\w\sigma^2_m$ does almost as well as the best approximation of the true function by a function of $S_m$.

Finally, it should be noted that the variance term is the same as for a diffusion without jumps.  Nevertheless, the remaining terms are larger because of the presence of the jumps.

\paragraph{Rates of convergence}
Let us remind that 
$\| t - \sigma^2 \|_{\pi^X}^2 \leq \pi_1\| t - \sigma_{|A}^2\|^2$ according to Lemma \ref{lemma: boundpiX}. Thus let us consider $t=\sigma^2_m$ the projection of $\sigma^2$ on $\mathcal{S}_m$ for the $L^2$-norm which realizes the minimum.
We assume now that the function of interest $\sigma_{|A}^2$ is in a Besov space $\mathcal{B}^{\alpha}_{2, \infty}$ with regularity $\alpha >0$ 
(see \emph{e.g.} \cite{devore} for a proper definition). Then it comes that
$\|\sigma^2_m-\sigma_{|A}^2\|^2 \leq C(\alpha) D_m^{-2\alpha}$. 
Finally, choosing $D_{m_{\text{opt}}}= n^{1/(2\alpha+1)}$, and $\Delta_n= n^{-\gamma}$ with $0<\gamma<1$, we obtain the rates of convergences given in Table \ref{tab:ratessig}. This table allows to compare the actual rates with the one obtained when $a \equiv 1$ and the process is a simple diffusion process. 

Since $\beta \in (1/4, 1/2)$, the best choice for $\beta$ is to choose it close to $1/4$. In this case, the estimator reaches the classical nonparametric convergence rate for high-frequency observations
($n\Delta_n^{(1+2\alpha)/2\alpha}= O(1)$). Otherwise, most of the time the remainder term will be predominant in the risk.
This result is analogous to the jump-diffusion case studied in \cite{Schmisser}.

\begin{table}
    \centering
    \begin{tabular}{c|c|c}
    & Hawkes-diffusion & Diffusion \\
    \hline
$0< \gamma \leq \frac{2\alpha}{2(2\alpha+1)}\leq 1/2$ & $\Delta_n^{2\beta-1/2}$ & $\Delta_n$\\
       $\frac{2\alpha}{2(2\alpha+1)}\leq \gamma \leq \left(\frac{2\alpha}{(2\alpha+1)(4\beta-1)}\right)\land 1 $  &  $\Delta_n^{2\beta-1/2}$ & $n^{-\alpha/(2\alpha+1)}$\\
        $ \left(\frac{2\alpha}{(2\alpha+1)(4\beta-1)}\right)\land 1\leq \gamma < 1$ &  $n^{-\alpha/(2\alpha+1)} $ & $n^{-\alpha/(2\alpha+1)} $
    \end{tabular}
    \caption{Rates of convergence for $\hat\sigma^2_{m}$}
    \label{tab:ratessig}
\end{table}

We observe that, after having replaced the optimal choice for $D_m$, which corresponds to $D_{m_{\text{opt}}}= n^{1/(2\alpha+1)}$, the conditions gathered in Assumption 5 become $\gamma \le \frac{2 \alpha -1}{1 + 2 \alpha}$ for collection [T] and $\gamma \le \frac{2 \alpha}{1 + 2 \alpha}$ for collections [DP] and [W]. \\
It is important also to remark that in order to have negligible reminder terms, we want $\gamma$ to be such that $\gamma \ge \frac{2\alpha}{(2\alpha+1)(4\beta-1)}$. As $\beta \in (\frac{1}{4}, \frac{1}{2})$, the best possible case is to have $\beta$ in the neighbourhood of $1/2$, so that the right hand side of the inequality here above becomes $\frac{2 \alpha}{1 + 2 \alpha}$. It means that, considering the collections [DP] and [W] are respected at the same time up to have a discretization step $\Delta_n = (\frac{1}{n})^{\frac{2 \alpha}{1 + 2 \alpha}}$ while considering the collection [T] the two conditions can not be respected at the same time and the only possibility is to take $\alpha$ big. Indeed, if the function of interest is in a Besov space $\mathcal{B}^{\infty}_{2, \infty}$, then the condition in Assumption 5 is respected and the reminder terms are negligible for $\gamma = 1$, which means $\Delta_n = 1/n$. \\
This implies that the term $\Delta_n^{4 \beta-1}$ always dominates the others. As it is easy to see comparing it with the third point of Proposition \ref{lemma: size A B E}, it comes from the contribution of the jumps. However, the bound obtained in Proposition \ref{lemma: size A B E} on the filtered jumps is optimal. Hence, in order to have negligible reminder terms, the idea is to try to lighten the condition on $N_n$ gathered in Assumption 5, rather than trying to improve the bound on the contribution of the jumps.

\subsection{Adaption procedure}{\label{S: adaptive volatility}}
We want to define a criterion in order to select automatically the best dimension $D_m$ (and so the best model) in the sense of the empirical risk. This procedure should be adaptive, meaning independent of $\sigma^2$ and dependent only on the observations. The final chosen model minimizes the following criterion:

\begin{equation}\label{eq: m1}
\w{m}_\sigma := \arg\min_{m \in \cM_n} \{\gamma_{n, M} (\w{\sigma}^2_m) + \pen_\sigma(m)\},
\end{equation}
with $\pen_\sigma(\cdot)$ the increasing function on $D_m$ given by
\begin{equation}\label{eq:pensig}
\pen_\sigma(m) := \kappa_1 \frac{D_m}{n}
\end{equation}
where $\kappa_1>0$ is a constant which has to be calibrated. 
Next theorem is proven in Section \ref{S: proof volatility}.
\begin{theorem}
Suppose that Assumptions 1,2,3,4 hold and that $\beta \in (\frac{1}{4}, \frac{1}{2})$. If $\Delta_n \rightarrow 0$ and $\log n = o(\sqrt{n \Delta_n})$ for $n \rightarrow\infty$, then the estimator $\w{\sigma}^2_{\w{m}_\sigma }$ of $\sigma^2$ on A given by equations \eqref{eq: def estimator sigma} and \eqref{eq: m1} satisfies 
$$\mathbb{E}\left[\left \| \w{\sigma}^2_{\w{m}_\sigma} - \sigma^2 \right \|_n^2\right] \le C_1 \inf_{m \in \cM_n} \left \{
 \inf_{t \in \mathcal{S}_m}\|t-\sigma^2\|^2_{\pi^X} 
+ \text{\rm pen}_\sigma(m) 
\right \} + C_2 \Delta_n^{4 \beta - 1} + \frac{C_3 \Delta_n^{ 4 \beta -\frac{3}{2}}}{n^2 } + \frac{C_4}{n }$$
where $C_1>1$ is a numerical constant and $C_2, C_3, C_4$ are positive constants depending on $\Delta_n, \sigma_1$ in particular.
\label{th: vol adaptive}
\end{theorem}
%
%
This inequality ensures that the final estimator $\w\sigma_{\w{m}_\sigma}^2$ realizes the best compromise between the bias term and the penalty term, which is of the same order as the variance term.
Indeed, it achieves the rates given in Table \ref{tab:ratessig} automatically, without the knowledge of the regularity of the function $\sigma^2$.

The convergence of this adaptive estimator is studied in Section \ref{sec:numerical}.

\section{Estimation procedure for both coefficients}{\label{S: volatility + jumps}}
\label{ref:both}
In addition to the volatility estimation, our goal is to propose a procedure to recover in a non-parametric way the jump coefficient $a$. The idea is to study the sum between the volatility and the jump coefficient and to recover consequently a way to approach the estimation of $a$ (see Section \ref{S: estim a} below). However, what turns out naturally is the volatility plus the product between the jump coefficient and the jump intensity, which leads to some difficulties as we will see in the sequel. To overcome such difficulties, we must bring ourselves to consider the conditional expectation of the intensity of the jumps with respect to $X_{t_i}$. 
In this way, we analyze the squared increments of the process $X$ differently to highlight the role of the conditional expectation. In the following, we use for the decomposition of the squared increments, the same notation as before: we denote the small bias term as $A_{t_i}$, the Brownian contribution as $B_{t_i}$ and the jump contribution as $E_{t_i}$, even if the forms of such terms are no longer the same as in Section \ref{S: volatility}. In particular, $A_{t_i}$ and $E_{t_i}$ are no longer the same as before, and their new definition can be found below, while the Brownian contribution $B_{t_i}$ remains exactly the same. To these, as previously anticipated, a term $C_{t_i}$ deriving from the conditional expectation of the intensity is added. \\

Besides, as in the previous section, we show that $A_{t_i}$ is small and $B_{t_i}$ is centered. Moreover, in this case, we also need the jump part to be centered. Therefore, we consider the compensated measure $d\tilde{N}_t$ instead of $dN_t$, relocating the difference in the drift. 

Let us rewrite the process of interest as:
\begin{equation}
\begin{cases}
d\lambda_t^{(j)} = - \alpha (\lambda_t^{(j)} - \zeta_t) dt +  \sum_{i = 1}^M c_{j,i} dN^{(i)}_t \\
dX_t = (b(X_t)+ a(X_{t^-}) \sum_{i = 1}^M \lambda_t^{(i)}) dt + \sigma(X_t) dW_t + a(X_{t^-}) \sum_{i = 1}^M  d\tilde{N}^{(i)}_t.
\end{cases}
\label{eq: model jump centered}
\end{equation}
We set now
\begin{equation}\label{eq: newJ}
J_{t_i} := \int_{t_i}^{t_{i + 1}} a(X_{s^-}) \sum_{i = 1}^M  d\tilde{N}^{(i)}_s.
\end{equation}
The increments of the process $X$ are such that
\begin{equation}\label{eq: introduction notation jump centered}
X_{t_{i + 1}} - X_{t_i} 
= \int_{t_i}^{t_{i + 1}} \left(b(X_s)+  a(X_{s^-}) \sum_{j = 1}^M \lambda^{(j)}_s\right) ds + Z_{t_i} + J_{t_i}
\end{equation}
where $J$ is given in Equation \eqref{eq: newJ} and $Z$ has not changed and is given in Equation \eqref{eq: ZJ}.
%
Let us define this time:
\begin{eqnarray}\label{eq: Ati}    
A_{t_i} &:=&\frac{1}{\Delta_n} \left(\int_{t_i}^{t_{i + 1}} (b(X_s)+  a(X_{s^-}) \sum_{j = 1}^M \lambda^{(j)}_s) ds\right)^2 + \frac{1}{\Delta_n} \int_{t_i}^{t_{i + 1}} (\sigma^2(X_s) - \sigma^2(X_{t_i})) ds \nonumber\\
&& + \frac{2}{\Delta_n}(Z_{t_i} + J_{t_i})\left(\int_{t_i}^{t_{i + 1}} (b(X_s) - b(X_{t_i})) +  (a(X_{s^-}) \sum_{j = 1}^M \lambda^{(j)}_s - a(X_{t_i}) \sum_{j = 1}^M \lambda^{(j)}_{t_i}) ds \right)\nonumber\\
&& + \frac{1}{\Delta_n} \int_{t_i}^{t_{i + 1}} (a^2(X_s) - a^2(X_{t_i})) \sum_{j = 1}^M \lambda^{(j)}_s ds  + \frac{a^2(X_{t_i})}{\Delta_n} \int_{t_i}^{t_{i + 1}}\sum_{j = 1}^M (\lambda^{(j)}_s - \lambda_{t_i}^{(j)}) ds \nonumber\\
 &&+ 2 \left(b(X_{t_i})+a(X_{t_i})\sum_{j = 1}^M \lambda^{(j)}_{t_i}\right) Z_{t_i}  + 2 \left(b(X_{t_i})+a(X_{t_i}) \sum_{j = 1}^M \lambda^{(j)}_{t_i}\right) J_{t_i},
\end{eqnarray}
\begin{equation}\label{eq: Eti}  
E_{t_i} :=  \frac{2}{\Delta_n} Z_{t_i} J_{t_i} + \frac{1}{\Delta_n} \left(J_{t_i}^2 - \int_{t_i}^{t_{i + 1}} a^2(X_s) \sum_{j = 1}^M \lambda^{(j)}_s ds\right).
\end{equation}
The term $A_{t_i}$ is small, whereas $B_{t_i}$ (which is the same as in the previous section) and $E_{t_i}$ are centered. 
%
%
Moreover, let us introduce the quantity
\begin{equation*}
\sum_{j = 1}^M \mathbb{E}[\lambda^{(j)}_{t_i} | X_{t_i} ]= \sum_{j = 1}^M \frac{\int_{\mathbb{R}^M} z_j \pi(X_{t_i}, z_1,\ldots , z_M) dz_1,\ldots , dz_M}{\int_{\mathbb{R}^M}\pi(X_{t_i}, z_1,\ldots , z_M) dz_1,\ldots , dz_M},
\end{equation*}
where $\pi$ is the invariant density of the process $(X, \lambda)$, whose existence has been discussed in Section \ref{S: ergodicity}; and \begin{equation}
C_{t_i} :=  a^2(X_{t_i})\sum_{j = 1}^M ( \lambda^{(j)}_{t_i} - \mathbb{E}[\lambda^{(j)}_{t_i} | X_{t_i} ]).
\label{eq: def C}
\end{equation}
It comes the following decomposition:
\begin{equation}\label{eq: Tipourg}
T_{t_i} = \frac{1}{\Delta_n} (X_{t_{i + 1}} - X_{t_i})^2= \sigma^2(X_{t_i}) + a^2(X_{t_i}) \sum_{j = 1}^M \mathbb{E}[\lambda^{(j)}_{t_i} | X_{t_i} ] + A_{t_i} + B_{t_i} + C_{t_i} + E_{t_i}.
\end{equation}
In the last decomposition of the squared increments, we have isolated the sum of the volatility plus the jump coefficient times the conditional expectation of the intensity with respect to $X_{t_i}$, which is an object on which we can finally use the same approach as before. Thus, as previously, the other terms need to be evaluated. The term $A_{t_i}$ is small and $B_{t_i}$ and $E_{t_i}$ are centered. Moreover, the just added term $C_{t_i}$ is clearly centered, by construction, if conditioned with respect to the random variable $X_{t_i}$ and, as we will see in the sequel, it is enough to get our main results. Here it is important to remark that the natural choice would have been to estimate directly $\sigma^2(X_{t_i}) + a^2(X_{t_i}) \sum_{j = 1}^M \lambda_{t_i}$, rather than $\sigma^2(X_{t_i}) + a^2(X_{t_i}) \sum_{j = 1}^M \mathbb{E}[\lambda^{(j)}_{t_i} | X_{t_i} ]$. The problem is that $\lambda_{t_i}$ has its own randomness  and is not observed and so the method used before for the estimation of the volatility coefficient can not work anymore. However, replacing $\lambda_{t_i}^{{ (j)}}$ with $\mathbb{E}[\lambda^{(j)}_{t_i} | X_{t_i} ]$, our goal turns out being the estimation of $g$ where
$g(X_{t_i}) : = \sigma^2(X_{t_i}) + a^2(X_{t_i}) \sum_{j = 1}^M \mathbb{E}[\lambda^{(j)}_{t_i} | X_{t_i} ],$  is now a function of the observation $X_{t_i}$ and so its non-parametric estimation can be accomplished using the same method as in previous section (see details below).  \\
As explained above Assumption \ref{ass: statio}, the Foster-Lyapunov condition in the exponential frames implies the existence of bounded moments for $\lambda$ and so we also get $\mathbb{E}[\lambda^{(j)}_{t_i} | X_{t_i} ] < \infty$, for any $j \in \left \{1,\ldots , M \right \}$.
\\

The properties here above listed are stated in Proposition \ref{prop: size A B C E} below, whose proof can be found in the appendix.

\begin{proposition}
Suppose that Assumptions 1,2,3 hold. Then, 
\begin{enumerate}
\item $\forall \tilde{\varepsilon} > 0$, $\mathbb{E}[A_{t_i}^2] \le c \Delta_{n,i}^{1 - \tilde{\varepsilon}},  \qquad \mathbb{E}[A_{t_i}^4] \le c \Delta_{n,i}^{1 - \tilde{\varepsilon}};$
\item $\mathbb{E}[B_{t_i}|\mathcal{F}_{t_i}] = 0, \qquad \mathbb{E}[B^2_{t_i}|\mathcal{F}_{t_i}] \le c \sigma_1^4, \qquad \mathbb{E}[B_{t_i}^4] \le c; $
\item $ \mathbb{E}[E_{t_i}|\mathcal{F}_{t_i}] = 0, \qquad \mathbb{E}[E^2_{t_i}|\mathcal{F}_{t_i}] \le \frac{c a_1^4}{\Delta_{n,i}} \sum_{j = 1}^M \lambda^{(j)}_{t_i}, \qquad \mathbb{E}[E_{t_i}^4] \le \frac{c}{\Delta_{n,i}^3};$
\item $\mathbb{E}[C_{t_i}|X_{t_i}] = 0, \qquad \mathbb{E}[C^2_{t_i}] \le c, \qquad \mathbb{E}[C_{t_i}^4] \le c.$
\end{enumerate}
\label{prop: size A B C E}
\end{proposition}

{\modch From Proposition \ref{prop: size A B C E} one can see in detail how small the bias term $A_{t_i}$ is. Moreover, it sheds light on the fact that the Brownian term and the jump term are centered with respect to the filtration $(\mathcal{F}_t)$ while $C$ is centered with respect to the $\sigma$-algebra generated by the process $X$. }
\subsection{Non-adaptive estimator}

Based on variables we have just introduced, we propose a nonparametric estimation procedure for the function
\begin{equation}\label{eq:defg}
g(x) := \sigma^2(x) + a^2(x) f(x)
\end{equation}
with 
\begin{equation}\label{eq:f}
f(x)=\frac{\sum_{j = 1}^M \int_{\mathbb{R}^M}  z_j \pi(x, z_1,\ldots , z_M) dz_1,\ldots , dz_M}{\int_{\mathbb{R}^M}\pi(x, z_1,\ldots , z_M) dz_1,\ldots , dz_M} 
\end{equation}
on a closed interval A. One can see that the estimation of $g$ is a natural way to approach the problem of the estimation of the jump coefficient. The same idea can be found for example in \cite{Schmisser}, where a L\'evy-driven stochastic differential equation is considered. The reason why in the above mentioned work the density does not play any role is that it is assumed to be one. \\
We consider $S_m$ the linear subspace of $L^2(A)$ defined in the previous section for $m \in \cM_n$ and satisfying Assumption \ref{ass: subspace}.
The contrast function is defined almost as before, since this time we no longer need to truncate the contribution of the jumps. It is, for $t \in \tilde{S}_n$,
\begin{equation}\gamma_{n, M}(t):= \frac{1}{n} \sum_{i = 0}^{n - 1} (t(X_{t_i})- T_{t_i})^2\one_A(X_{t_i}) 
\label{E:contrast pour g}
\end{equation}
and the $T_{t_i}$ are given in Equation \eqref{eq: Tipourg} this time. 
The associated mean squares contrast estimator is
\begin{equation}
\w{g}_m := \arg\min_{t \in S_m} \gamma_{n, M}(t).
\label{eq: def estimator g} 
\end{equation}

We want to bound the empirical risk $\mathbb{E}[\left \| \w{g}_m - g \right \|_n^2 ]$ on the compact $A$. It is the object of the next result, whose proof can be found in Section \ref{S: proof both}.
\begin{proposition}
Suppose that Assumptions 1,2,3,4,5 hold. If $\Delta_n \rightarrow 0$ and $\log n = o(\sqrt{n \Delta_n})$, then the estimator $\w{g}_m$ of $g$ on $A$ satisfies, for any $0<\tilde{\varepsilon},  \varepsilon < 1$,
\begin{equation}\label{eq: propboth}
\mathbb{E}\left[\left \| \w{g}_m - g \right \|_n^2 \right] \le  3 \inf_{t \in \mathcal{S}_m} \left \| t - g \right \|_{\pi^X}^2  +  \frac{C_1 (\sigma_1^4+a_1^4 + 1) D_m^{ 1 + 2 \varepsilon}}{n \Delta_n} +  C_2 \Delta_n^{1 - \tilde{\varepsilon}} + \frac{C_3}{ n^2 \Delta_{n}^\frac{3}{2}},
\end{equation}
with $C_1$, $C_2$ and $C_3$ positive constants.
\label{prop: estim both}
\end{proposition}

As in the previous section, this inequality measures the performance of our estimator $\w{g}_m$ for the empirical norm and the comments given after Proposition \ref{prop: volatility} hold. The main difference between the proof of Proposition \ref{prop: volatility} and the proof of Proposition \ref{prop: estim both} is that, in the first case, we deal with the jumps by introducing the indicator function $\varphi$. In this way, the jump part is small and some rough estimations are enough to get rid of them (see point 3 of Proposition \ref{lemma: size A B E}). From Proposition \ref{prop: size A B C E} we can see that for the estimation of both coefficients, instead, the jump contribution (gathered in $C_{t_i}$ and $E_{t_i}$) is no longer small. However, $C_{t_i}$ and $E_{t_i}$ are both centered (even if with respect to different $\sigma$ algebras) and we can therefore apply on them the same reasoning as we did for $B_{t_i}$, which consists in a more detailed analysis. Hence, proving Proposition \ref{prop: estim both} is more challenging than proving Proposition \ref{prop: volatility}. Evidence of this is for example the fact that, to estimate $g$, a bound on the variance of $C_{t_i}$ relying on mixing properties is required (see Lemma \ref{lemma: variance Cti}).

Finally, let us compare this result with the bound \eqref{eq: propsigma} obtained for the estimator $\w{\sigma}_m^2$. The main difference is that, up to a term $D_m^{2 \varepsilon}$ for $\varepsilon$ arbitrarily small, the second term is of order $D_m/(n\Delta)$ here, instead of $D_m/n$ as it was previously. Consequently, in practice, the risks will depend mainly on $n\Delta$ for the estimation of $g$ and $n$ for the estimation of $\sigma^2$. \\
 It is worth remarking that the reason why  this extra term $D_m^{2 \varepsilon}$ appears relies on the use of the $\beta$-mixing property of the process (see Lemma \ref{lemma: variance Cti}). However, in the intensity is of the jump process is constant (Poisson process) or in the case where the process satisfies some stronger mixing properties (such as the $\rho$-mixing, for example for diffusion processes, see \cite{GCJC2000}), it is possible to improve the result gathered in Proposition \ref{prop: estim both} and to lose the $D_m^{2 \varepsilon}$.

\paragraph{Rates of convergence.}
Assume now that $g_{|A} \in \mathcal{B}_{\alpha, \infty}^2$ with $\alpha\geq 1$, then, taking $t=g^2_m$ (the projection of $g$ on $\mathcal{S}_m$) produces
$\|g_m-g_{|A}\|^2 \leq C(\alpha) D_m^{-2\alpha}$. 
Choosing $D_{m_{\text{opt}}}= (n\Delta)^{1/(2\alpha+  2 \varepsilon + 1)}$, if 
$n \Delta^{2-\tilde{\varepsilon}} \rightarrow 0$ leads to
$$\mathbb{E}\left[\left \| \w{g}_{m_{\text{opt}}} - g \right \|_n^2 \right] \leq (n\Delta)^{-2\alpha/(2\alpha+  2 \varepsilon + 1)}.$$
We obtain,  up to an $\varepsilon$ arbitrarily small, the same rate of convergence for the regularity $\alpha$ as \cite{Schmisser} for the estimation of $\sigma^2+a^2$ in the case of a jump diffusion with a L\'evy process instead of the Hawkes process. 

In practice the regularity of $g$ is unknown and thus it is necessary to choose the best model in a data driven way. This it the subject of the next paragraph. 

\subsection{Adaption procedure}{\label{s: adap g}}
Also for the estimation of $g$ we define a criterion in order to select the best dimension $D_m$ in the sense of the empirical risk. This procedure should be adaptive, meaning independent of $g$ and dependent only on the observations.  The final chosen model minimizes the following criterion:
\begin{equation}\label{eq: m2}
\w{m}_g := \arg\min_{m \in \cM_n} \{\gamma_{n, M} (\w{g}_m) + \pen_g(m)\},
\end{equation}
with $\pen_g(\cdot)$ the increasing function on $D_m$ given by
\begin{equation}\label{eq:peng}
\pen_g(m) := \kappa_2 \frac{D_m^{ 1 + 2 \varepsilon}}{n \Delta_n},
\end{equation}

where $\kappa_2$ is a constant which has to be calibrated and $\varepsilon$ is arbitrarily small. 

To establish an oracle-type inequality for the adaptive estimator $\w{g}_{\w{m}}$, the following further assumption on the discretizion step is essential.
\begin{ass}{\label{ass: step}}
There exists $\epsilon > 0$ such that 
$n^\epsilon log(n) = o (\sqrt{n \Delta})$ for $n \rightarrow \infty$.
\end{ass}
One can be interested in the reason why this condition, stronger than \eqref{eq: cond delta} we previously required, is needed. Note that \eqref{eq: cond delta} is also the condition required in the discretization scheme proposed in \cite{Dion Lemler} and used in \cite{CGCR} for the nonparametric estimation of coefficients for diffusions. As already said, the proof of the adaptive procedure involves the application of Talagrand inequality. To apply Talagrand, we need to get independent, bounded random variables through Berbee's coupling method and truncation. Intuitively the point is that, in general, such variables are built starting from the Brownian part only (in particular from $B_{t_i}$, as in \eqref{eq: def B}). For the adaptive estimation $g$, instead, also the jumps are involved (Talagrand variables depend on $B_{t_i} + C_{t_i} + E_{t_i}$, with $C_{t_i}$ and $E_{t_i}$ as in \eqref{eq: Eti} and \eqref{eq: def C}, respectively). Then, an extra term $n^{\epsilon}$ appears naturally looking for a bound for the jump part (see Lemma \ref{lemma: P omega B complementare} and its proof), which results in the final stronger condition gathered in Assumption \ref{ass: step}.

We analyse the quantity $\mathbb{E}[\left \| \w{g}_{\w{m}} - g \right \|_n^2 ]$ in the following theorem, whose proof is relegated in Section \ref{S: proof both}.
\begin{theorem}
Suppose that Assumptions 1,2,3,4,5,6 hold. If $\Delta_n \rightarrow 0$, then the estimator $\w{g}_{\w{m}_g}$ of $g$ on $A$ satisfies, for any $1<\tilde{\varepsilon} < 0$,
$$\mathbb{E}\left[\left \| \w{g}_{\w{m}_g} - g \right \|_n^2\right] \le  C_1 \inf_{m \in \cM_n} \left \{\inf_{t\in\mathcal{S}_m} \left \| t - g \right \|_{\pi^X}^2 + \text{\rm pen}_g(m)  \right \} + C_2 \Delta_n^{1 - \tilde{\varepsilon}} + \frac{C_3}{n^2 \Delta_n^\frac{3}{2}} + \frac{C_4}{n \Delta_n}$$
where $C_1>1$ is a numerical constants and $C_2, C_3, C_4$ are positive constants depending on $\Delta_n, a_1, \sigma_1$ in particular.
\label{th: estim both adaptive}
\end{theorem}
This result guarantees that our final data-driven estimator $\w g_{\w m}$ realizes automatically the best compromise between the bias term and the penalty term, and thus reaches the same rate obtained when the regularity of the true function $g$ is known. Note that here since it is more difficult to estimate $g$ because we have to deal with the conditional expectation of the intensity $f$, the last two error terms are larger than the ones obtained in Theorem \ref{th: vol adaptive} for the estimation of $\sigma^2$.
 
\section{A strategy to approach the jump coefficient}{\label{S: estim a}}
\label{sec:a}
The challenge is to get an estimator of the coefficient $a(\cdot)$.
Let us first remind the reader of the notation $f(x) := \sum_{j = 1}^M \mathbb{E}[\lambda_{t_i}^{(j)}|X_{t_i} = x]$ (see Equation \eqref{eq:f}) and 
$$
g(x) = \sigma^2(x) + a^2(x) f(x).
$$
Thus, a natural idea is to replace $f$ in the previous equation by an estimator, and then, to study an estimator of $a(\cdot)$ of the form $\displaystyle \frac{\w{g}(x)-\w{\sigma}^2(x)}{\w{f}(x)}$.
This is not a simple issue and let us discuss the estimation of $f$ later in the section. 
Assuming that an estimator of $f$ is known and denoted $\w{f}_h$ where $h>0$ denotes a tunning parameter.

Then, 
we also assume that $f>f_0$ on $A$.
We then define: 
$$\w{a}_z^2:=\frac{\w{g}_{m_2}(x)-\w{\sigma}^2_{m_1}(x)}{\w{f}_h(x)}\one_{\w{f}_h(x) > f_0/2} $$
with $z= (m_1, m_2, h)$. 
Let us study this estimator, for the empirical norm. 
Due to the disjoint support of the two terms and together with Cauchy-Schwarz inequality, we obtain
\begin{eqnarray*}
\|\w{a}_z^2-a^2\|_n^2 &=&\left\| \left(
\frac{(\w{g}_{m_2}-g)}{\w{f}_h}
+ \frac{(\sigma^2-\w{\sigma}^2_{m_1})}{\w{f}_h}+ \frac{(g-\sigma^2)}{f}\frac{f-\w{f}_h}{\w{f}_h}
\right)\one_{\w{f}_h> f_0/2}  \right\|_n^2 + \left\|\frac{g-\sigma^2}{f} \one_{\w{f}_h < f_0/2} \right\|_n^2\\
&\leq&  \frac{12}{f_0^2}
\| \w{g}_{m_2}-g\|_n^2 + \frac{12}{f_0^2}
 \|\sigma^2-\w{\sigma}^2_{m_1}\|_n^2+ 3 \left\|a^2\left(\frac{f-\w{f}_h}{\w{f}_h}\right)
\one_{\w{f}_h> f_0/2}\right\|_n^2 \\
&&+ \frac{1}{n}\sum_{i=0}^{n-1} a^4(X_{t_i}) \one_{\w{f}_h(X_{t_i}) < f_0/2}.
\end{eqnarray*}
Besides, if $\w{f}_h \leq f_0/2$ then $|\w{f}_h-f| > f_0/2$ and as $a^2(\cdot)<a_1^2$ finally:
\begin{eqnarray*}
\E[\|\w{a}_z^2-a^2\|_n^2] 
&\leq&  \frac{12}{f_0^2}
\E[\| \w{g}_{m_2}-g\|_n^2] + \frac{12}{f_0^2}
 \E[\|\sigma^2-\w{\sigma}^2_{m_1}\|_n^2]+ 
 \frac{12 a_1^4}{f_0^2} \E\left[\left\|{f-\w{f}_h}\right\|_n^2\right] \\
 &&+ \frac{a_1^4}{n}\sum_{i=0}^{n-1}  \mathbb{P}(|\w{f}_h(X_{t_i})-f(X_{t_i})|>f_0/2).
\end{eqnarray*}
And by Markov's inequality, we obtain: 
\begin{eqnarray}
\label{eq:OIa}
\E[\|\w{a}_z^2-a^2\|_n^2] 
&\leq&  \frac{12}{f_0^2}\left(
\E[\| \w{g}_{m_2}-g\|_n^2] + 
 \E[\|\w{\sigma}^2_{m_1}-\sigma^2\|_n^2]+ 
2 a_1^4 \E\left[\left\|{f-\w{f}_h}\right\|_n^2\right]\right) .
\end{eqnarray}

This equation teaches us that the empirical risk of the estimator $\w{a}_z$ is upper bounded by the sum of the three empirical risks of the estimators of the functions $g,\sigma^2, f$. The first two are controlled in Theorem \ref{th: vol adaptive} and \ref{th: estim both adaptive}. The last one is more classic. The Nadaraya-Watson estimator can be studied with one or two bandwidth parameters.

Finally, the triplet of parameters $z=(m_1,m_2,h)$ must be chosen in a collection. A first way to do it is to use the model selection proposed in the paper for $ \w{\sigma}^2_{m_1},\w{g}_{m_2}$ and then select $h$ through cross-validation for example, obtaining finally $\w{z}=(\w{m}_\sigma, \w{m}_g, \w{h})$. Another possible way would consist in defining a new selection procedure for the triplet, for example a Goldenshluger-Lepski type, as it is proposed in \cite{comtemarie}.
Nevertheless, the authors mention that it seems numerically less performing than to use the one-bandwidth leave-one-out cross validation method in the Nadaraya-watson estimator context.

Then, we choose to study the first methodology in the next Section, because it is directly implementable using the built adaptive estimators of $\sigma^2$ and $g$.
Besides, the risk bound obtained on $\w{a}_z$ in Equation \eqref{eq:OIa} suggests that the better the three functions $\sigma^2, g, f$ are estimated, the better the estimation of $a$ will be.

\begin{remark}
Let us note here that $f$ can be lower bounded by construction. Indeed, its definition jointly with the fact that $\lambda_{t_i}^{(j)} > \zeta_j$ because of the positiveness of $h_{i, j}$, provides us the wanted lower bound. For $\w{a}_z^2$ to be an estimator, $f_0$ must be known or estimated.
\end{remark}

\paragraph{Estimation of $f$} For sake of simplicity let us assume that $M=1$.
   We have that $f(x_{t_k})=\E[\lambda_{t_k}|X_{t_k}=x_{t_k}]$ for all $k$. 
Thus, $f$ depends on the conditional intensity $\lambda$; thus, the estimator of $f$ will also depend on the estimator of $\lambda$. In addition, to estimate $\lambda$, we need more data: we already observe the process $X$ on discrete times, but we also need to observe the jump times (which are not assumed to be known in the above).

Let us assume that we have at our disposal in addition to the $(X_{t_i})_i$'s the sequence $T_j$'s of jump times. 
Now, as the Hawkes process is assumed to have exponential kernel, this estimation can simply be done using likelihood contrast estimator for example, and we denote $\w{\lambda}_{t}$ the estimator of the intensity process at time $t$ (which do not depend on $X$).
Then, function $f$ can be estimated through a Nadaraya-Watson type estimator defined as
       $$\w{f}_h^{NW}(x)=\sum_{k=1}^n \frac{K_h(x-X_{t_k})}{\sum_{i=1}^n K_h(x-X_{t_i})} \w{\lambda}_{t_k}.$$
 The parameter, $h$ can be chosen using cross-validation for simplicity. 
Under strong assumptions, the risk of $\w{f}_{h}^{NW}$ is bounded by the risk of the numerator and by the risk of the denominator. Nevertheless here, to get a bound for the risk of the estimator we need to get a bound for the numerator which does not seem to be a direct computation. This is not solved in the present paper and will be the object of further considerations.

\section{Numerical experiments}{\label{S: numerical}}
\label{sec:numerical}

In this section, we present our numerical study on synthetic data. 
\subsection{Simulated data}
%
We simulate the Hawkes process $N$ with $M=1$ for simplicity, and here we denote $(T_k)_k$ the sequence of jump times. 
In fact, the multidimensional structure of the Hawkes process allows to consider a lot of kinds of data, but what is impacting the dynamic of $X$ is the cumulative Hawkes process, thus in that sense we do not lose generality taking $M=1$.
In this case, the intensity process is written as
$$\displaystyle \lambda_t= \xi
 +(\lambda_0 -\xi) e^{-\alpha t}
+ \sum_{T_k < t} c e^{- \alpha (t-T_k)}.$$
The initial conditions $X_0,\lambda_0$ should be simulated according to the invariant distribution (and $\lambda_0$ should be larger than $\xi>0$). This measure of probability is not explicit. Thus we choose: $\lambda_0=\xi $ and $X_0=2$ in the examples. Also, the exogenous intensities $\xi$ is chosen equal to $0.5$, the coefficient $c$ is equal to $0.4$ and $\alpha=5$.
    

Then we simulate $(X_\Delta, \dots X_{(n+1)\Delta})$ from an Euler scheme with a constant time step $\Delta_i=\Delta$. Because of the additional jump term (when $a \neq 0$), it is not possible to use classical more sophisticated scheme to the best of our knowledge. 
A simulation algorithm is also detailed in \cite{DLL} Section 2.3.

To challenge the proposed methodology, we investigate different kinds of models. In this section, we present the results for four models, which are the following

\begin{enumerate}[label=(\alph*)]
\item $b(x)= -4x$, $\sigma(x)=1$, $a(x)=\sqrt{2+0.5 \sin(x)}$,
\item $b(x)= -2x + \sin(x)$, $\sigma(x)=\sqrt{(3+x^2)/(1+x^2)}$,  $a(x)=1$, 
\item $b(x)= -2x$, $\sigma(x)=\sqrt{1+x^2}$, $a(x)=1$,
\item $b(x)= -2x$, $\sigma(x)=\sqrt{1+x^2}$, $a(x)=x\one_{[-5,5]}+5\one_{(-\infty,-5)}-5\one_{(5,+\infty)}$. 
\end{enumerate}

The drift is chosen linear to satisfy the assumptions and as it is not of interest to study the estimation of $b$ here, keeping  simple drift coefficient, let us focus on the differences observed due to the coefficients $\sigma$ and $a$.
For example, in models c) and d), $\sigma$ does not satisfy Assumption \ref{ass: X}.
%
Let us now detail the numerical estimation strategy. 

\subsection{Computation of nonparametric estimators}
It is important to remind the reader that the estimation procedures are only based on the observations $(X_{k\Delta})_{k=0, \ldots, n} $.
Indeed, the estimators $\w{\sigma}^2_{\w{m}_\sigma}$ and $\w{g}_{\w{m}_g}$ of $\sigma^2$ and $g$ respectively defined by \eqref{eq: def estimator sigma} and \eqref{eq: def estimator g}, are based on the statistics:
$$T_{k\Delta}= \frac{(X_{(k+1)\Delta}-X_{k\Delta})^2}{\Delta},~ k= 0, \ldots, n-1.$$
\paragraph{Estimation of $\sigma^2$.}
To compute $\w{\sigma}^2_m$ we use a version of the truncated quadratic variation through a function $\varphi$
that vanishes when the increments of the data are too large compared to the standard increments of a continuous diffusion process. Precisely, we choose
\begin{equation}\label{eq:Tphi}
T^{\varphi}_{k\Delta}:=  T_{k\Delta} \times\varphi\left(\frac{X_{(k+1)\Delta}-X_{k\Delta}}{ \Delta^\beta} \right)  
;\quad  \varphi(x) =\begin{cases} & 1 \quad |x| <1 \\ & e^{1/3+ 1/(|x|^2-4)}  \\ & 0 \quad |x| \geq 2 \end{cases}.
\end{equation}
This choice for the smooth function $\varphi$ is discussed in \cite{Unbiased}.

\paragraph{Estimation of $g$.}
As far as the estimation of $g:= \sigma^2+ a^2 \times f$ is concerned, we do not know the true conditional expectations $f(x_{t_k})=\E[\lambda_{t_k}|X_{t_k}=x_{t_k}]$ for all $k$. Thus we compare the estimations of $g$ to the approximate function 
$
\tilde{g}(x)=\sigma^2(x)+a^2(x)\times \w{f}^{NW}_{\w{h}(x)}$
where the function $f(x)=\dfrac{\int z\pi(x,z) dz}{\pi_X(x)}$, which corresponds to $\E[\lambda|X=x]$, is estimated with the classical Nadaraya-Watson estimator $\w{f}^{NW}_{{h}}(x)$, where $h$ is the bandwidth parameter. 
To do so, we use the R-package
\texttt{ksmooth}. Then, $\w{h}$ is chosen through a cross-validation leave-one-out procedure.


\paragraph{Choice of the subspaces of $L^2(A)$}
The spaces $\mathcal{S}_m$ are generated by the trigonometric basis. The maximal dimension $N_n$ is chosen equal to $20$ for this study. The theoretical dimension $\lfloor\sqrt{n\Delta}/n^\varepsilon\log(n)\rfloor$ is often too small in practice since we have to consider higher dimension to estimate non-regular functions.

In the theoretical part, the estimation is done on a fixed compact interval $A$. Here it is slightly different. We consider for each model the random data range as the estimation interval. This is more adapted to a real-life data set situation.









\subsection{Details on the calibration of the constants}

Let us remind the reader that the two penalty functions $\pen_\sigma$ are given in Equation \eqref{eq:pensig} and $\pen_g$ given in Equation \eqref{eq:peng}. 
We consider here the limit scenario where $\varepsilon=0$ and the penalties are both linear in the dimension.
They depend on constants named $\kappa_1, \kappa_2$. These constants need to be chosen once for all for each estimator in order to compute the final adaptive estimators $\w{\sigma}^2_{\w{m}_{\textcolor{blue}{\sigma}}}$ and $\w{g}_{\w{m}_{\textcolor{blue}{g}}}$. We explain now how these choices are made.

\paragraph{Choice for the universal constants.}

In order to choose the universal constants $\kappa_1$ and $\kappa_2$
we investigate models varying $b, a, \sigma^2$ (different from those used to validate the procedure later on) for $n \in \{100, 1000, 10000\}$ and $\Delta \in \{0.1, 0.01\}$. We compute Monte-Carlo estimators of the risks $ \E[ \|\w{\sigma}^2_{\w{m}_{\textcolor{blue}{\sigma}}}-\sigma^2 \|_n^2]$ and $ \E[ \|\w{g}^2_{\w{m}_{\textcolor{blue}{g}}}-\tilde{g} \|_n^2]$. We choose to do $N_{\text{rep}}=1000$ repetitions to estimate this expectation by the average: 
 $$
 \frac{1}{N_{\text{rep}}}\sum_{k=1}^{N_{\text{rep}}} \|\w{\sigma}^{2,(k)}_{\w{m}_{\textcolor{blue}{\sigma}}}-\sigma^2 \|_n^2 \quad \text{and} \quad \frac{1}{N_{\text{rep}}}\sum_{k=1}^{N_{\text{rep}}} \|\w{g}^{(k)}_{\w{m}_{\textcolor{blue}{\sigma}}}-\tilde{g} \|_n^2.$$
Finally, comparing the risks as functions of $\kappa_1, \kappa_2$ leads to select values making a good compromise overall experiences.
Applying this procedure, we finally choose $\kappa_1=100$ and $\kappa_2=100$.



\paragraph{Choice for the threshold $\beta$.}

The parameter $\beta$ appears in Equation \eqref{eq:Tphi}. This parameter helps the algorithm to decide if the process has jumped or not. The theoretical range of values is $(1/4, 1/2)$. We choose to work with $\beta=1/4+ 0.01$.

\paragraph{Choice for the bandwidth $h$.}
The bandwidth $h$ in the Nadaraya-Watson estimator of the conditional expectation is chosen through a leave-one-out cross-validation procedure.
Since the true conditional expectation is unknown, we focus on the estimation of $\widetilde{g}$, which depends on this estimator anyway. Indeed it is the estimation procedure of $g$ that is evaluated. 
Other choices for the best bandwidth exist as the  Goldenshluger and Lepski method \cite{GL} or a Penalized Comparison to Overfitting \cite{PCO}. 

\subsection{Results: estimation of the empirical risk}

As for the calibration phase, we compute Monte-Carlo estimators of the empirical risks. We choose to do $N_{\text{rep}}=1000$ repetitions to estimate this expectation by the average on the simulations.
In the risk tables \ref{tab:errorsig} and \ref{tab:errorg}, we present for the three models and different values of $(\Delta, n)$: the average of the estimated risk over $1000$ simulations (MISE) and the standard deviation in the brackets. 

Also, we print the result for the \textit{oracle} function in both cases. Indeed, as on simulations we know functions $\sigma^2, \tilde{g}$, we can compute the estimator in the collection $\cM_n= \{1, \ldots, N_n\}$ which minimises in $m$ the errors $\|\w{\sigma}^2_{{m}}- \sigma^2\|_n^2$ and $\|\w{g}_{{m}}- \tilde{g}\|_n^2$. Let us denote the oracle estimators 
$\w{\sigma}^2_{m^*}$ and $\w{g}_{m^*}$ respectively. 
These are not true estimators as they are not available in practice. Nevertheless, it is the benchmark. 
The goal of this numerical study is thus to see how close the risk results of $\w{\sigma}^2_{\w{m}_{\textcolor{blue}{\sigma}}},~ \w{g}^2_{\w{m}_{\textcolor{blue}{g}}}$ are to the risks of 
these two \textit{oracle} functions.

Let us detail the result for each estimator. 

\paragraph{Estimation of $\sigma^2$.}

\begin{figure}
\centering
\includegraphics[width=13cm, height=7cm]{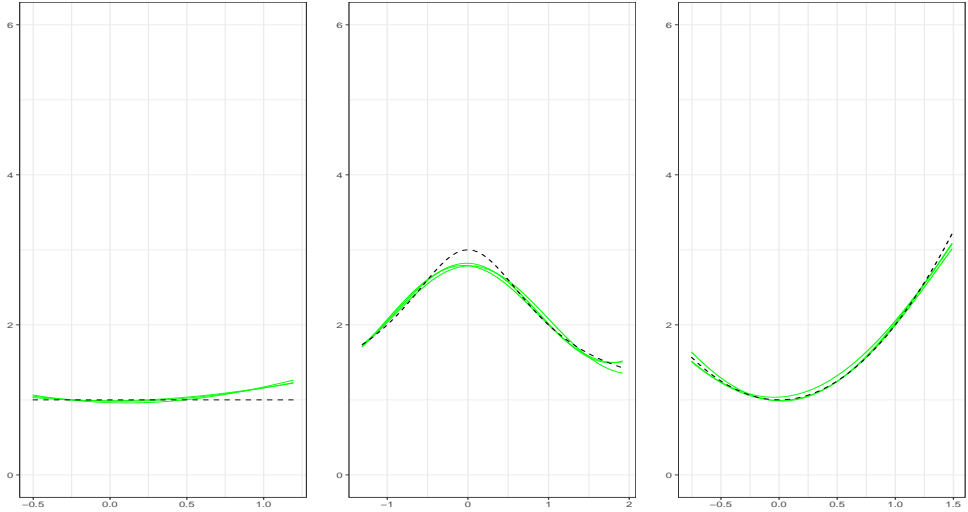}
\caption{Models (a),(b),(c) with $n=10000$, $\Delta=0.01$. Three final estimators are plain green (plain line), true $\sigma^2$ plain black (dotted line)}
\label{fig:explsigma}
\end{figure}

Figure \ref{fig:explsigma} shows for models (a),(b),(c), three estimators $\w{\sigma}^2_{\w{m}_{\textcolor{blue}{\sigma}}}$ in green (light grey) and the true function $\sigma^2$ in black (dotted line). We can appreciate here the good reconstruction of the function $\sigma^2$ by our estimator.

Table \ref{tab:errorsig} sums up the results of the estimator $\w{\sigma}^2_{\w{m}_{\textcolor{blue}{\sigma}}}$ for the different models and different parameter choices. We present also the results for the \textit{oracle} estimator $\w{\sigma}^2_{m^*}$ as it has been said previously.

The estimations of the MISE and the standard deviation are really close to the oracle ones. As it has been shown in the theoretical part, we can notice that the MISE decreases when $n$ increases. Besides, as the variance term is proportional to $1/n$ when $n$ is fixed and large enough, we can see the clear influence of $\Delta$ from $0.1$ to $0.01$, the MISEs are divided at least by $10$. Model (c) seems to be the more challenging for the procedure.

\begin{table}
\centering
{\small
{\setlength{\tabcolsep}{4pt}
\begin{tabular}{c||cccccc}
\hline
$\Delta,n$ &  \multicolumn{2}{|c|}{ $\Delta=0.1 $  $n=1000$} & \multicolumn{2}{|c|}{ $\Delta=0.1  $  $n=10000 $}  & \multicolumn{2}{|c}{ $\Delta=0.01  $  $n= 10000$}\\
\hline
Model &  $\w{\sigma}_{\w{m}_{\textcolor{blue}{\sigma}}}$ & $\w{\sigma}_{{m}^*}$ &  $\w{\sigma}_{\w{m}_{\textcolor{blue}{\sigma}}}$ & $\w{\sigma}_{{m}^*}$ &$\w{\sigma}_{\w{m}_{\textcolor{blue}{\sigma}}}$ & $\w{\sigma}_{{m}^*}$\\
\hline
\hline
(a) & 0.410 (0.280) & 0.361 (0.285) &  0.385 (0.122)& 0.278 ( 0.088) & 0.015 (0.028) & 0.010 (0.023)  \\
\hline
(b) & 0.187 (1.678) &0.107 (0.989) & 0.046 (1.162) & 0.027 (1.014) & 0.005 (0.015) & 0.005 (0.008) \\
\hline
(c) & 1.201 (0.216) & 0.798 (0.208) & 0.452 (0.062)&  0.366 (0.042) & 0.015 (0.012) & 0.008 (0.007)\\
\hline
 \end{tabular}}}
 \caption{Estimation on a compact interval. Average and standard deviation of the estimated risks $\|\w{\sigma}^2_{\w{m}_{\textcolor{blue}{\sigma}}} - \sigma^2\|^2_{n}$ and $\|\w{\sigma}^2_{{m}^*} - \sigma^2\|^2_{n}$ computed over $1000$ repetitions.}
 \label{tab:errorsig}
 \end{table}

 \paragraph{Estimation of $\tilde{g}$.}

Figure \ref{fig:expg} shows for each of the three models (a),(b),(c), three estimators $\w{g}_{\w{m}_{\textcolor{blue}{g}}}$ of $\tilde g$ in green (light gray) and function $\tilde{g}$ in black (dotted line).
The beams of the three realizations of the estimator are satisfying.

\begin{figure}
\centering
\includegraphics[width=13cm, height=7cm]{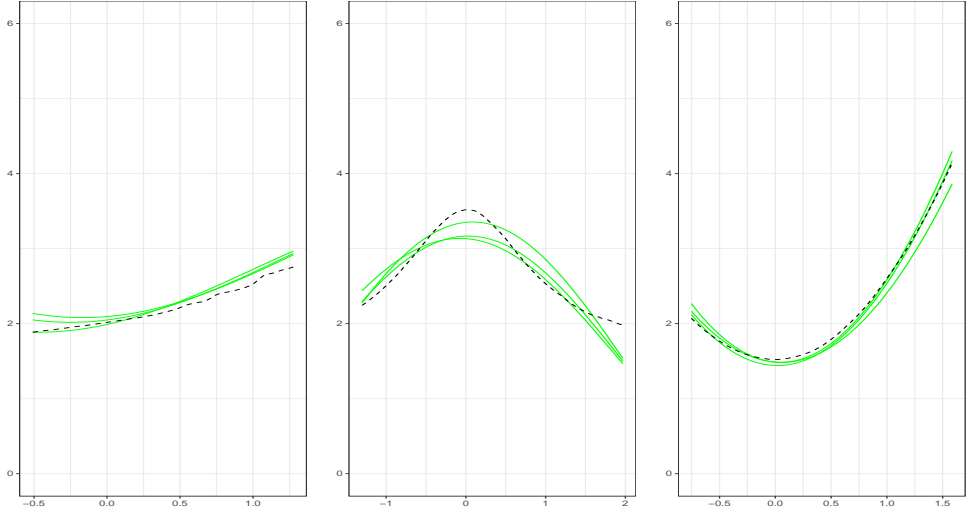}
\caption{Models (a),(b),(c) with $n=10000$, $\Delta=0.01$. Three final estimators of $\tilde g$ are plain green (plain line) and $\tilde g$ plain black (dotted line).}
\label{fig:expg}
\end{figure}

We observe that the procedure has difficulties in Model (a), and we confirm that impression in Table \ref{tab:errorg} below with the estimation of the risk. 
But for the two other models, the estimators seem closer to the true function. The estimation appears to work better in Model (c) than in Model (b), and this is also corroborated by the estimation of the risk given in Table \ref{tab:errorg}. 

Table \ref{tab:errorg} gives the Mean Integrated Squared Errors (MISEs) of the estimator $\w g_{\w{m}_{\textcolor{blue}{g}}}$ obtained from our procedure  and of the \textit{oracle} estimator $\w g_{m^*}$, which is the best one in the collection for the three different models with different values of $\Delta$ and $n$. 

As expected, we observe that the MISEs are smaller when $n$ increases and $\Delta$ decreases. The different Models (a), (b), (c) gives relatively good results even if, as already said, it seems a little bit more difficult to estimate correctly $g$ in Model (a), probably because the volatility $\sigma^2$ is constant in this case. For the two other models, the estimators seem to be better. Compared with the results on the estimation of $\sigma^2$, the variance is proportional to $1/(n\Delta)$, and thus, the risks are greater in general.

\begin{table}
\centering
{\setlength{\tabcolsep}{4pt}
\begin{tabular}{c||c|c|c|c|c|c}
\hline
$\Delta,n$ &  \multicolumn{2}{|c|}{ $\Delta=0.1 $  $n=1000$} &  \multicolumn{2}{|c|}{ $\Delta=0.1  $  $n=10000 $}  & \multicolumn{2}{|c}{ $\Delta=0.01  $  $n= 10000$}\\
\hline
Model &  $\w{g}_{\w{m}_{\textcolor{blue}{g}}}$ & $\w{g}_{{m}^*}$ & $\w{g}_{\w{m}_{\textcolor{blue}{g}}}$ & $\w{g}_{{m}^*}$& $\w{g}_{\w{m}_{\textcolor{blue}{g}}}$ & $\w{g}_{{m}^*}$\\
\hline
\hline
(a)& 1.363 (0.715) & 0.895 (0.606) & 0.948 (0.193)&0.735 (0.195)&0.129 (0.141)& 0.109 (0.120)\\
\hline
(b)& 0.915 (0.520) & 0.474 (0.393) &0.313 (0.174)&0.198 (0.079)&0.240 (0.100)&0.098 (0.072)\\
\hline
(c)& 0.707 (0.964) & 0.311 (0.320) &0.236 (0.202)&0.099 (0.056)&0.073 (0.130)&0.035 (0.035)\\
\hline
 \end{tabular}}
 \caption{Estimation on a compact interval. Average and standard deviation of the estimated risks $\|\w{g}_{\w{m}_{\textcolor{blue}{g}}} - \tilde{g}\|^2_{n}$ and $\|\w{g}_{{m}^*} - \tilde{g}\|^2_{n}$ computed over $1000$ repetitions.}
 \label{tab:errorg}
 \end{table}

%

\subsection{Estimation of $a^2$}

As explained in Section \ref{S: estim a}, the challenge is to get an approximation of the coefficient $a$ from the two previous estimators. A main numerical issue is that, according to the theoretical and numerical results, the best setting for the estimation of $\sigma^2$ and $g$ are not the same. Indeed, the smallest $\Delta$ is, the best the estimation of $\sigma^2$ is, as only large $n$ is important, and on the contrary, $n\Delta$ needs to be large to estimate $g$ properly. 

To overcome this difficulty, we choose a thin discretization of the trajectories of $X$. 
We simulate here discrete path of the process $X$ at first with $\Delta= 10^{-3},~ n= 10^5$. 
Then, we first compute $\w{g}_{\w{m}_{\textcolor{blue}{g}}}$ the estimator of $\tilde g$ on all the observations. Secondly, we compute $\w{\sigma}^2_{\w{m}_{\textcolor{blue}{\sigma}}}$ the estimator of $\sigma^2$ from a subsample of the discretized observations (one over ten observations thus $\Delta= 0.01, ~n=10000$).

We finally compute 
the estimator 
$$
\w{a}^2(x)=\dfrac{\w{g}_{\w{m}_{\textcolor{blue}{g}}}(x)-\w\sigma_{\w{m}_{\textcolor{blue}{\sigma}}}^2(x)}{\w{f}^{NW}_{\w{h}}(x)}.
$$
This procedure is presented in Section \ref{S: estim a}. 
We have plugged-in $\w{a}^2$ the final estimators of $\sigma^2, g$.

We present on Figure \ref{fig:expla} the results obtained on model (d) in which neither $\sigma^2$ nor $a$ are constant. Indeed, for the three other models, our procedure has difficulties estimating properly $g$, $\sigma^2$ and $a^2$, when one of the diffusion jump process parameters is constant. 
We see that the final estimator $\w a_{\w z}^2$ is not so far from the true function $a^2$ even if there are some fluctuations around the true function. This is understandable because we add the errors coming from the estimations of $\sigma^2$ and $g$ as we can see on Inequality (\ref{eq:OIa}). Moreover, it should not be forgotten that we do not know exactly $g$ and that we already make an error by estimating $\widetilde g$ instead of g, this error is then reflected in the estimate of $a^2$.

\begin{figure}
\centering
\includegraphics[width=13cm,height=7cm]{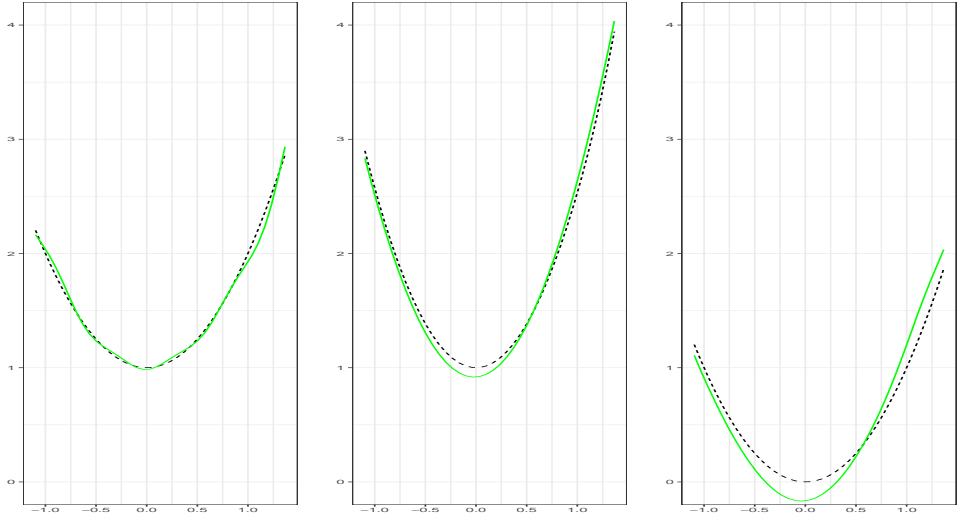}
\caption{Model (d). Final estimators $\w{g}_{\w{m}_2}$, $\w{\sigma}^2_{\w{m}_1}$ and $\w{a}$ are plain {\modch green (plain line)}, and true parameters $\tilde g$, $\sigma^2$ and $a^2$ in plain {\modch black (dotted line)} from left to right respectively.}
\label{fig:expla}
\end{figure}

\section{Discussion}
\label{sec:discussion}

This paper investigates the jump-diffusion model with jumps driven by a Hawkes process. 
This model is interesting to complete the collection of jump-diffusion models and consider dependency in the jump process. The dynamic of the trajectories obtained from this model is impacted by the Hawkes process, which acts independently of the diffusion process. 

This work focuses on the estimation of the unknown coefficients $\sigma^2$ and $a$. We propose a classical adaptive estimator of $\sigma^2$ based on the truncated increments of the observed discrete trajectory. This allows estimating the diffusion coefficient when no jump is detected. 

Then, we estimate the sum $g:=\sigma^2+ a^2 \times f$. Indeed, it is this function and not $\sigma^2+a^2$ that can be estimated. The multiplicative term $f$ is the sum of the conditional expectations of the jump process. 
This function can be estimated separately through a Nadaraya-Watson estimator. The proposed estimator of $g$ is built using all increments of the quadratic variation this time.

Furthermore, a main issue is to reach the jump coefficient $a$ from the two first estimators $\w{\sigma}^2_{\w{m}_\sigma}$ and $\w{g}_{\w{m}_g}$ for which the theoretical and numerical results are convincing.
The last section of this article answered this question partially. It is simple to build an estimator of $a$ from the two previous ones and the estimator of the unknown conditional intensity function $f$. 

Nevertheless, this is possible only if the jumps of the Hawkes process are observed, which is the case of the simulation study. Then, when real-life data arises, the jump times of the counting process must be known to be able to reach $a$ with our methodology. Otherwise, the issue remains an open question.


Then, the proposed estimator $\w{a}_z$, with $z=(m_1,m_2,h)$, is a quotient of estimators and the denominator must be lower bounded to ensure the proper definition of the estimator. This could be theoretically and numerically carefully studied and be the object for further works.

Finally, our analysis sheds light on the importance to further investigate the conditional intensity function $f$, dependent on the invariant density $\pi$. A future perspective would be to propose a kernel estimator for the invariant density $\pi$ and to study its behavior and its asymptotic properties deeply, following the same approach as in \cite{Strauch} and \cite{Chapitre 4}. A projection method is instead considered in \cite{Schmisser_Krell} to estimate the invariant density associated with a piecewise deterministic Markov process. Consequently, it will be possible to discuss the properties of the related estimator of $f$.

From the nonparametric estimation point of view, it should be interesting to extend the present estimation work to estimation on the real line instead of on a compact interval. 
\cite{CGC2020} brings a solution to deal with the estimation of the drift function on the all real line from repeated observations. The procedure may be extended to the present framework in future works.


\section*{Acknowledgements}

The authors particularly grateful for the constructive comments and suggestions for improvements made by the referees of the journal.

\section{Proofs}
\label{sec:proofs}
This section is devoted to the proofs of the results stated in Sections \ref{S: volatility} and \ref{S: volatility + jumps}. \\
One may observe that, concerning non-adaptive estimators, the proof of both Propositions \ref{prop: volatility} and \ref{prop: estim both} relies on the same scheme.  It consists in introducing the set $\Omega_n$ as in \eqref{eq: omega}, on which the norms $\left \| \cdot \right \|_{\pi^X}$ and $\left \| \cdot \right \|_{n}$ are equivalent, and to bound the risk on $\Omega_n$ and $\Omega_n^c$, respectively. On $\Omega_n^c$, a rough bound on the quantities we are considering is enough, as the probability of $\Omega_n^c$ is very small (see \eqref{eq: proba omega n c}). Hence, the idea to bound the risk on $\Omega_n^c$ in Proposition \ref{prop: volatility} and \ref{prop: estim both} is basically the same.
On $\Omega_n$, instead, there are main differences. Indeed, in Proposition \ref{prop: volatility}, it is enough to upper bound roughly both the bias and the jump terms, (to deal more in detail only with the Brownian part), while in Proposition \ref{prop: estim both} a in-depth study is required for $B_{t_i}$, $C_{t_i}$ and $E_{t_i}$. 
Such difference between the proofs for the estimation of $\sigma$ and $g$ is more highlighted in the analysis of the adaptive procedure. The proof of both Theorems \ref{th: vol adaptive} and \ref{th: estim both adaptive}, indeed, heavily relies on Talagrand inequality and, as for the non-adaptive procedure, for the estimation of $\sigma$ what really matters is the contribution of $B_{t_i}$, while for the estimation of $g$ also $C_{t_i}$ and $E_{t_i}$ are involved. It implies that, for the proof of Theorem \ref{th: estim both adaptive}, we are using Berbee's coupling method to get independent variables and truncation to make them bounded, starting from some variables in which also the jumps contribute; which is challenging.

\subsection{Proof of volatility estimation}{\label{S: proof volatility}}
Here we prove all the results stated in Section \ref{S: volatility}. We start proving Proposition \ref{prop: volatility}.
\subsubsection{Proof of Proposition \ref{prop: volatility}}
\begin{proof}
We want to obtain an upper bound for the empirical risk $\mathbb{E}[\left \| \w{\sigma}^2_m - \sigma^2 \right \|_n^2]$. First of all we remark that, if $t$ is a deterministic function, then it is $\mathbb{E}[\left \| t \right \|^2_n] = \left \| t \right \|^2_{\pi^X}$. \\
\\
By the definition of $T_{t_i}$ we have that 
\begin{eqnarray*}
\gamma_{n, M}(t)&:=& \frac{1}{n} \sum_{i = 0}^{n - 1} \left(t(X_{t_i})- T_{t_i}\varphi_{\Delta_{n,i}^\beta}(\Delta_i X)\right)^2{\modar \one_A(X_{t_i}) }\\
&=& \frac{1}{n} \sum_{i = 0}^{n - 1} \left(t(X_{t_i})- \sigma^2(X_{t_i}) - (\tilde{A}_{t_i} + B_{t_i} + E_{t_i}\varphi_{\Delta_{n,i}^\beta}(\Delta_i X) )\right)^2{\modar \one_A(X_{t_i}) }\\
&=& \left \| t - \sigma^2 \right \|_n^2 + \frac{1}{n} \sum_{i = 0}^{n - 1} (\tilde{A}_{t_i} + B_{t_i} + E_{t_i}\varphi_{\Delta_{n,i}^\beta}(\Delta_i X))^2{\modar \one_A(X_{t_i}) }\\
&&- \frac{2}{n} \sum_{i = 0}^{n - 1} \left(\tilde{A}_{t_i} + B_{t_i} + E_{t_i}\varphi_{\Delta_{n,i}^\beta}(\Delta_i X)\right)\left(t(X_{t_i})- \sigma^2(X_{t_i})\right){\modar \one_A(X_{t_i}). }
\end{eqnarray*}
As $\w{\sigma}^2_m$ minimizes $\gamma_{n, M}(t)$, for any $\sigma^2_m \in \mathcal{S}_m$ it is $\gamma_{n, M}(\w{\sigma}_m^2) \le \gamma_{n, M}(\sigma_m^2)$ and therefore
$$\left \| \w{\sigma}^2_m - \sigma^2 \right \|_n^2 \le \left \| \sigma_m^2 - \sigma^2 \right \|_n^2 + \frac{2}{n} \sum_{i = 0}^{n - 1} (\tilde{A}_{t_i} + B_{t_i} + E_{t_i}\varphi_{\Delta_{n,i}^\beta}(\Delta_i X))(\w{\sigma}^2_m(X_{t_i})- \sigma^2_m(X_{t_i})){\modar ,}$$
{\modar where in the last sum we can remove the indicator since $\w{\sigma}_m$ and $\sigma_m$ are compactly supported on $A$.}
Let us denote the contrast function
\begin{equation}
\nu_n (t) := \frac{1}{n} \sum_{i = 0}^{n - 1} B_{t_i} t (X_{t_i}).
\label{eq: def nu}
\end{equation} 
In the sequel, we will repeatedly use that, for $d>0$, it is $2 x y \leq x^{2} / d+d y^{2}$. It follows
\begin{eqnarray*}
\left \| \w{\sigma}^2_m - \sigma^2 \right \|_n^2 &\leq & \left \| \sigma^2_m - \sigma^2 \right \|_n^2 + \frac{d}{n} \sum_{i = 0}^{n - 1} \left(\tilde{A}_{t_i} +  E_{t_i}\varphi_{\Delta_{n,i}^\beta}(\Delta_i X) \right)^2 + \frac{1}{d} \left \| \w{\sigma}^2_m - \sigma^2_m \right \|_n^2   \nonumber\\
&&+ 2 \nu_n( {\sigma}^2_m - \w{\sigma}^2_m ).
\end{eqnarray*}
The linearity of the function $\nu_n$ in $t$ implies that 
$$
\abs{
2\nu_n (\w{\sigma}^2_m- {\sigma}^2_m)} = 
 2\|\w{\sigma}^2_m- {\sigma}^2_m \|_{\pi^X}
 \abs{\nu_n((\w{\sigma}^2_m- {\sigma}^2_m)/\|\w{\sigma}^2_m- {\sigma}^2_m\|_{\pi^X})} \leq 
2\|\w{\sigma}^2_m- {\sigma}^2_m \|_{\pi^X}\sup_{t \in \mathcal{B}_m} 
\abs{\nu_n(t)},
 $$
then, using again that $2 x y \le \frac{x^2}{d} + d y^2$, we obtain the upper bound
$$2{\abs{\nu_n (\w{\sigma}^2_m- {\sigma}^2_m) }}\leq 
\frac{1}d\|\w{\sigma}^2_m- {\sigma}^2_m \|^2_{\pi^X} + d \sup_{t \in \mathcal{B}_m} \nu^2_n(t)
 $$
 where $\mathcal{B}_m = \left \{ t \in S_m : \left \| t \right \|_{\pi^X}^2 \le 1\right \}$.
Finally, using Cauchy-Schwarz's inequality leads to
\begin{eqnarray}
\label{eq: prima di definire omega n}
\left \| \w{\sigma}^2_m - \sigma^2 \right \|_n^2 &\leq & \left \| \sigma^2_m - \sigma^2 \right \|_n^2 + \frac{2d}{n} \sum_{i = 0}^{n - 1} \tilde{A}_{t_i}^2 +  \frac{2d}{n} \sum_{i = 0}^{n - 1} E_{t_i}^2\varphi^2_{\Delta_{n,i}^\beta}(\Delta_i X)  + \frac{1}{d} \left \| \w{\sigma}^2_m - \sigma^2_m \right \|_n^2   \nonumber\\
&&
+ d \sup_{\mathcal{B}_m} \nu_n^2(t) + \frac{1}{d} \left \| \w{\sigma}^2_m - \sigma^2_m \right \|_{\pi^X}^2.
\end{eqnarray}
Let us set 
\begin{equation}\label{eq: omega}
\Omega_n := \left \{ \omega, \forall t \in \tilde{S}_n \backslash \{0 \}, \left|\frac{\left \| t \right \|_{n}^2}{\left \| t \right \|_{\pi^X}^2} - 1\right| \le \frac{1}{2} \right \}, 
\end{equation}
on which the norms $\left \| \cdot \right \|_{\pi^X}$ and $\left \| \cdot \right \|_{n}$ are equivalent. We now act differently to bound the risk on $\Omega_n$ and $\Omega_n^c$. 

\paragraph{Bound of the risk on $\Omega_n$.}
On $\Omega_n$, it is
$$\left \| \w{\sigma}^2_m - \sigma^2_m \right \|_{\pi^X}^2 \le 2 \left \| \w{\sigma}^2_m - \sigma^2_m \right \|_{n}^2 \le 4 \left \| \w{\sigma}^2_m - \sigma^2 \right \|_{n}^2 + 4 \left \| \sigma^2 - \sigma^2_m \right \|_{n}^2,$$
where in the last estimation we have used triangular inequality. In the same way we get
$$ \left \| \w{\sigma}^2_m - \sigma^2_m \right \|_{n}^2 \le 2 \left \| \w{\sigma}^2_m - \sigma^2 \right \|_{n}^2 + 2 \left \| \sigma^2 - \sigma^2_m \right \|_{n}^2.$$
Replacing them in \eqref{eq: prima di definire omega n} we obtain
\begin{eqnarray*}
\left \| \w{\sigma}^2_m - \sigma^2 \right \|_n^2 &\le & \left \| \sigma_m^2 - \sigma^2 \right \|_n^2 + \frac{2d}{n} \sum_{i = 0}^{n - 1} \tilde{A}_{t_i}^2 + \frac{2d}{n} \sum_{i = 0}^{n - 1} (E_{t_i}\varphi_{\Delta_{n,i}^\beta}(\Delta_i X))^2 +   d \sup_{t \in \mathcal{B}_m} \nu_n^2(t)\\
&&+ \frac{6}{d} \left \| \w{\sigma}^2_m - \sigma^2 \right \|_{n}^2 + \frac{6}{d} \left \| \sigma^2 - \sigma^2_m \right \|_{n}^2.
\end{eqnarray*}

We need $d$ to be more than $6$. {We take the optimal choice for $d$, which corresponds to $d= 12$, obtaining}
\begin{equation}
\left \| \w{\sigma}^2_m - \sigma^2 \right \|_n^2 \le {3} \left \| \sigma^2_m - \sigma^2 \right \|_n^2 + \frac{{ 48} }{n} \sum_{i = 0}^{n - 1} \tilde{A}_{t_i}^2 + \frac{{48}}{n} \sum_{i = 0}^{n - 1} (E_{t_i}\varphi_{\Delta_{n,i}^\beta}(\Delta_i X))^2 + { 24} \sup_{t \in \mathcal{B}_m} \nu_n^2(t).
\label{eq: d replaced}
\end{equation}
We denote as $(\psi_l)_l$ an orthonormal basis of $S_m$ for the $L^2_{\pi^X}$ norm (thus $\int_{\mathbb{R}} \psi^2_l (x) \pi^X (x) dx = 1$). Each $t \in \mathcal{B}_m$ can be written 
$$t = \sum_{l = 1}^{D_m} \alpha_l \psi_l, \qquad \mbox{with } \sum_{l = 1}^{D_m} \alpha_l^2 \le 1.$$
Then 
\begin{equation}
\sup_{t\in \mathcal{B}_m} \nu_n^2(t) = \sup_{\sum_{l = 1}^{D_m} \alpha_l^2 \le 1} \nu_n^2\left(\sum_{l = 1}^{D_m} \alpha_l \psi_l \right) \le \sup_{\sum_{l = 1}^{D_m} \alpha_l^2 \le 1} \left(\sum_{l = 1}^{D_m} \alpha_l^2\right)\left(\sum_{l = 1}^{D_m} \nu_n^2 (\psi_l) \right) = \sum_{l = 1}^{D_m} \nu_n^2 (\psi_l).
\label{eq: calcolo nu}
\end{equation}
To study the risk we need to evaluate the expected value. From \eqref{eq: d replaced}, \eqref{eq: calcolo nu} and using the first and the third points of Proposition \ref{lemma: size A B E}, we get
\begin{equation}
\E\left[\left \| \w{\sigma}^2_m - \sigma^2 \right \|_n^2 \one_{\Omega_n}\right] \le { 3} \mathbb{E}\left[ \left \| \sigma^2_m - \sigma^2 \right \|_n^2 \right] + c \Delta_n^{1 - \tilde{\varepsilon}} + c \Delta_n^{4 \beta - 1} + { 24} \sum_{l = 1}^{D_m} \mathbb{E}[\nu_n^2 (\psi_l)].
\label{eq: risk Omega unfinished}
\end{equation}
By the definition \eqref{eq: def nu} of $\nu_n$ it is
$$\nu_n (\psi_l) = \frac{1}{n} \sum_{i = 0}^{n - 1}B_{t_i} \psi_l (X_{t_i}).$$
As $B_{t_i}$ is conditionally centered, using the second point of Proposition \ref{lemma: size A B E}, it is
$$\sum_{l = 1}^{D_m} \mathbb{E}[\nu_{n}^2 (\psi_l)] \le \frac{c}{n^2} \sum_{i = 0}^{n - 1} \sum_{l = 1}^{D_m} \mathbb{E}[\psi^2_l(X_{t_i}) \mathbb{E}[B^2_{t_i}|\mathcal{F}_{t_i}]] \le \frac{c}{n^2} \sum_{i = 0}^{n - 1} \sum_{l = 1}^{D_m} \sigma_1^4 \mathbb{E}[\psi^2_l(X_{t_i})] \le \frac{c \sigma_1^4 D_m}{n}.$$
 Replacing the inequality here above in \eqref{eq: risk Omega unfinished} it yields
 $$\mathbb{E}\left[\left \| \w{\sigma}^2_m - \sigma^2 \right \|_n^2 \one_{\Omega_n}\right] \le { 3} \mathbb{E}\left[ \left \| \sigma^2_m - \sigma^2 \right \|_n^2 \right] + c \Delta_n^{4 \beta - 1 } + \frac{c \sigma_1^4 D_m}{n}.$$
 As { for any deterministic $t$ it is $\mathbb{E}[\left \| t \right\|_n] = \left \| t \right\|_{\pi^X}$, it follows}
\begin{equation}
\mathbb{E}[\left \| \w{\sigma}^2_m - \sigma^2 \right \|_n^2 \one_{\Omega_n}] \le { 3} \inf_{t \in \mathcal{S}_m} \left \| t - \sigma^2 \right \|_{\pi^X}^2 + c \Delta_n^{4 \beta - 1 } + \frac{c \sigma_1^4 D_m}{n}.
\label{eq: risk on Omega n}
\end{equation}
\paragraph{Bound of the risk on $\Omega_n^c$.}
The complementary space $\Omega_n^c$ of $\Omega_n$ given in Equation \eqref{eq: omega} is defined as:
$$\Omega_{n}^c  =\left\{\omega \in \Omega,\;\;\exists t^* \in \tilde{S}_{n}\backslash\{0\}, \left| \frac{\|t^*\|^2_{n}}{\|t^*\|_{\pi^X}^2}-1 \right| > 1/2 \right\}.
$$

Let us set $e = (e_{t_0},\ldots , e_{t_{n - 1}})$, where $e_{t_i} := T_{t_i}\varphi_{\Delta_{n,i}^\beta}(\Delta_i X) - \sigma^2 (X_{t_i}) = \tilde{A}_{t_i} + B_{t_i} + E_{t_i}\varphi_{\Delta_{n,i}^\beta}(\Delta_i X)$. Moreover 
$$\Pi_m T \varphi = \Pi_m (T_{t_0}\varphi_{\Delta_{n,0}^\beta}(\Delta_0 X),\ldots , T_{t_{n - 1}}\varphi_{\Delta_{n,n-1}^\beta}(\Delta_{n - 1} X)) = (\w{\sigma}^2_m (X_{t_0}),\ldots , \w{\sigma}^2_m (X_{t_{n - 1}})),$$
where $\Pi_m$ is the Euclidean orthogonal projection over $S_m$. Then, according to the projection definition, 
\begin{eqnarray*}
\left \| \w{\sigma}^2_m - \sigma^2 \right \|_n^2 &=& \left \| \Pi_m T \varphi - \sigma^2 \right \|_n^2 = \left \| \Pi_m T \varphi - \Pi_m \sigma^2 \right \|_n^2 + \left \| \Pi_m \sigma^2 - \sigma^2 \right \|_n^2 \\
&&\le \left \| T \varphi - \sigma^2 \right \|_n^2 + \left \| \sigma^2 \right \|_n^2 = \left \| e \right \|_n^2 + \left \| \sigma^2 \right \|_n^2.
\end{eqnarray*}
Therefore, from Cauchy -Schwarz inequality and the boundless of $\sigma^2(x)$,
\begin{eqnarray*}
\E\left[\left \| \w{\sigma}^2_m - \sigma^2 \right \|_n^2 \one_{\Omega_n^c}\right] &\le & \mathbb{E}\left[\left \| e \right \|_n^2 \one_{\Omega_n^c}\right] + \mathbb{E}\left[\left \| \sigma^2 \right \|_n^2 \one_{\Omega_n^c}\right] = \frac{1}{n } \sum_{i = 0}^{n - 1} \mathbb{E}[e_{t_i}^2 \one_{\Omega_n^c}  ] + \frac{1}{n } \sum_{i = 0}^{n - 1} \mathbb{E}[\sigma^4(X_{t_i}) \one_{\Omega_n^c} ] \\
&\leq &  \frac{1}{n } \sum_{i = 0}^{n - 1} \mathbb{E}[e_{t_i}^4]^\frac{1}{2} \mathbb{P}(\Omega_n^c)^\frac{1}{2} + \sigma_1^4 \mathbb{P}(\Omega_n^c).
\end{eqnarray*}
From Lemma 6.4 in \cite{Dion Lemler}, {if $\frac{n \Delta_n}{(\log n)^2} \rightarrow \infty$ and $N_n \le \frac{n \Delta_n}{(\log n)^2}$ for [DP] and [W] and $N_n^2 \le \frac{n \Delta_n}{(\log n)^2}$ for the collection [T], }then
\begin{equation}
\mathbb{P}(\Omega_n^c) \le \frac{c_0}{n^4}.
\label{eq: proba omega n c}    
\end{equation}
In the hypothesis of our proposition we have requested that $ \log n = o(\sqrt{n \Delta_n})$. As for $n$ going to $\infty$ we have $ \frac{(\log n)^2}{ n \Delta_n} < \frac{ \log n}{\sqrt{n \Delta_n}} \rightarrow 0$, the first condition in Lemma 6.4 in \cite{Dion Lemler} hold true. Regarding the bound on {$N_n$, we have required Assumption 5} and so we can apply the here above mentioned lemma, which yields \eqref{eq: proba omega n c}. \\
We are left to evaluate $\mathbb{E}[e_{t_i}^4]$. From Proposition \ref{lemma: size A B E} it follows
$$\mathbb{E}\left[e_{t_i}^4\right] \le \mathbb{E}\left[\tilde{A}_{t_i}^4 + B_{t_i}^4+ E_{t_i}^4 \varphi^4_{\Delta_{n,i}^\beta}(\Delta_i X) \right] \le c \Delta_n^{1 - \tilde{\varepsilon}} + c + c \Delta_{n}^{8 \beta -3} \le c \Delta_{n}^{0 \land 8 \beta -3} .$$
Putting the pieces together it yields
\begin{equation}
\E\left[\left \| \w{\sigma}^2_m - \sigma^2 \right \|_n^2 \one_{\Omega_n^c}\right] \le  \frac{c \Delta_{n}^{0 \land 4 \beta -\frac{3}{2}}}{n^2} + \frac{c}{n^4} \le \frac{c \Delta_{n}^{0 \land 4 \beta -\frac{3}{2}}}{n^2}.
\label{eq: risk on Omega n c}
\end{equation}
From \eqref{eq: risk on Omega n} and \eqref{eq: risk on Omega n c} it follows
\begin{equation*}
\E\left[\left \| \w{\sigma}^2_m - \sigma^2 \right \|_n^2 \right] \le {6} \inf_{t \in \mathcal{S}_m} \left \| t - \sigma^2 \right \|_{\pi^X}^2 +  \frac{C_1 \sigma_1^4 D_m}{n} +  C_2 \Delta_n^{4 \beta - 1} + \frac{C_3 \Delta_{n}^{0 \land 4 \beta -\frac{3}{2}}}{n^2}.
\end{equation*}
\end{proof}


\subsubsection{Proof of Theorem \ref{th: vol adaptive}}
\begin{proof}
{For simplicity in notation we denote $\w{m}_\sigma= \w{m}$ in the proof.}
We analyse the quantity $\mathbb{E}[\left \| \w{\sigma}^2_{\w{m}} - \sigma^2 \right \|_n^2 ]$, acting again in different way depending on whether or not we are on $\Omega_n$. On $\Omega_n^c$ the proof can be led as before, getting
\begin{equation}
\mathbb{E}\left[\left \| \w{\sigma}^2_{\w{m}} - \sigma^2 \right \|_n^2 \one_{\Omega_n^c}\right] \le \frac{c \Delta_{n}^{0 \land 4 \beta - \frac{3}{2}}}{ n^2 }.
\label{eq: adapt risk omega n c}
\end{equation}
Now we investigate what happens on $\Omega_n$. By the definition of $\w{m}$ it is
$$\gamma_{n, M} (\w{\sigma}_{\w{m}}) + \pen (\w{m}) \le \gamma_{n, M} (\w{\sigma}_m) + \pen (m)\le \gamma_{n, M} ({\sigma}_m) + \pen (m)$$
and so, acting as before \eqref{eq: d replaced}, we get
\begin{eqnarray}\label{eq: on Omega, intermedio}
\mathbb{E}\left[\left \| {\sigma}^2_{\w{m}} - \sigma^2 \right \|_n^2 \one_{\Omega_n}\right] \le { 3}  \mathbb{E}[\left \| {\sigma}^2_{m} - \sigma^2 \right \|_n^2] + \frac{{ 48}}{n} \sum_{i = 0}^{n - 1} \mathbb{E}[\tilde{A}_{t_i}^2] + \frac{{ 48}}{n} \sum_{i = 0}^{n - 1} \E[(E_{t_i} \varphi_{\Delta_{n,i}^\beta}(\Delta_i X))^2]  \nonumber\\
 + { 24} \, \E\left[\sup_{t \in \mathcal{B}_{m, \w{m}}}\nu^2_{n} (t)\right] + { 12} \pen(m) - {12} \mathbb{E}[\pen(\w{m})],
 \end{eqnarray}
where $\nu_n$ has been defined in \eqref{eq: def nu} and 
$$\mathcal{B}_{m, m'} := \left \{ h \in S_m + S_{m'} : \left \| h \right \|_{\pi^X} \le 1 \right \}.$$
We want to control the term $\mathbb{E}[\sup_{t \in \mathcal{B}_{m, \w{m}}}(\nu_{n} (t))^2]$ and, to do that, we introduce the function $p(m, m')$ which is such that 
\begin{equation}
p(m, m') = \frac{1}{{ 24}}(\pen (m) + \pen (m')).
\label{eq: pmm}
\end{equation}
It is  
$$\mathbb{E}\left[\sup_{t \in \mathcal{B}_{m, \w{m}}}\nu_{n} (t)^2\right] \le  \mathbb{E}\left[ p(m, \w{m})\right] +  \sum_{m' \in \cM_n} \mathbb{E}\left[\left(\sup_{t \in \mathcal{B}_{m, m'}}(\nu_{n} (t))^2- p(m, m')\right)_+\right] .$$
In order to bound the second term in the right hand side here above we want to use Lemma 7 in \cite{Schmisser noisy}. We can remark that, for any $p \ge 2$, $\mathbb{E}[|B_{t_i}|^p] \le \frac{c}{\Delta_n^p} \mathbb{E}[Z_{t_i}^{2 p}] + c \sigma_1^{2 p}$. According to Proposition 4.2 in Barlow and Yor \cite{Barlow_Yor}  there exists a constant $c$ such that, for any $p> 0$, 
$$\E\left[Z_{t_i}
^{2p}\right] \le { c} \Delta_n^p \sigma_1^{2p}. $$
It follows
$$\E[|B_{t_i}|^p] \le { c} \sigma_1^{2p} .$$
By Lemma 7 in \cite{Schmisser noisy} there exists a constant $k$ such that, for any $m, m' \in \cM_n$, 
\begin{equation}
\E\left[\left(\sup_{t \in \mathcal{B}_{m, m'}}\nu^2_{n} (t)- k c \sigma_1 p(m, m')\right)_+\right] \le c \frac{e^{-(D_m + D_{m'})}}{n}.
\label{eq: resultat Schmisser adapt}
\end{equation}

We have said, in the definition of the penalization function $\pen_\sigma$ given in Subsection \ref{S: adaptive volatility}, that the constant $\kappa_1$ has to be calibrated. In particular, we need it to be such that $\frac{\kappa_1}{{ 24}} \geq  k c \sigma_1$, where $\sigma_1$ is the upper bound for the volatility provided in the second point of Assumption \ref{ass: X} and $k$ and $c$ are as in Lemma 7 of \cite{Schmisser noisy}.
We underline that Lemma 7 in \cite{Schmisser noisy} has been proved for a noisy diffusion. However, the same reasoning applies for a jump diffusion (see the proof of Theorem 13 in \cite{Schmisser}) and for our framework as well, as it is based on a projection argument and on algebraic computations which still hold true. \\
{We remark that Assumption (iii) of Section 2.2 of \cite{Schmisser noisy}, on the cardinality of the support of the basis, holds true only for the collections [DP] and [W]. However, Lemma 7 of \cite{Schmisser noisy} still holds true for the collection [T], up to add the condition $\frac{N_n^3}{n} \le 1$ as in the first point of Assumption 5. Indeed Lemma 8 of \cite{Schmisser noisy} (on which the proof of Lemma 7 relies) does not change considering the collection [T]. Then $\bar{r}_{m,m'}$, as introduced in the proof of Lemma 7 in \cite{Schmisser noisy}, is now bounded by an extra $D= \max(D_m, D_m')$, which implies an extra $D$ in the definition of both $\eta_0$ and $\eta_k$. In particular now we have, using the notation in Lemma 7 of \cite{Schmisser noisy}, $\eta_k := 2^{-k}(\sqrt{c_3 x_k} + c_4 D x_k)$. It follows, after having replaced $x_k$,
$$\eta^2 = (\sum_{k = 0}^{\infty} \eta_k)^2 \le c \gamma^2 (\frac{D}{n} + \frac{\tau}{n} + \frac{D^4}{n^2} + \frac{D^2 \tau^2}{n^2}).$$
Now, $\frac{D^4}{n^2} \le \frac{D}{n}$ as we have assumed $\frac{N_n^3}{n} \le 1$ in Assumption 5. Then, following again the proof of Lemma 7 in \cite{Schmisser noisy} but substituting the variable $\tau$ with $y$ such that $\tau = c \gamma^2 (\frac{y}{n} + D^2 \frac{y^2}{n^2})$, we get
\begin{align*}
E & = C \gamma^2 e^{- D}(\frac{1}{n} \int_0^{\infty} e^{-y} dy + \frac{2}{n^2} \int_0^\infty D^2 y e^{-y} dy ) \\
& \le c \frac{\gamma^2}{n} e^{-D} (1 + \frac{D^2}{n}).
\end{align*}
However, Assumption 5 implies $\frac{D^2}{n} \le 1$ and so we get that the extra part due to the choice of the collection [T] is negligible. We recover $E \le c \frac{\gamma^2}{n} e^{-D}$, as in Lemma 7 of \cite{Schmisser noisy} and as we wanted.} \\
From \eqref{eq: resultat Schmisser adapt} and the fourth point of Assumption \ref{ass: subspace} we get
$$ \sum_{m' \in \cM_n} \E\left[\left(\sup_{t \in \mathcal{B}_{m, m'}}\nu_{n}^2(t)- p(m, m')\right)_+\right] \le \frac{c}{n  } \sum_{m' \in \cM_n} e^{-(D_m + D_{m'})} \le \frac{c}{n  }.  $$
It provides us, using also \eqref{eq: risk on Omega n c} and Proposition \ref{lemma: size A B E}, 

\begin{eqnarray*}
\mathbb{E}\left[\left \| \w{\sigma}^2_{\w{m}} - \sigma^2 \right \|_n^2\right] &\le& { 3 \,} \mathbb{E}\left[\left \| {\sigma}^2_{m} - \sigma^2 \right \|_n^2\right] + c \Delta_n^{4 \beta - 1} + \frac{c }{n^4 } + c \pen (m)+ \frac{c \Delta_n^{0 \land (4 \beta - \frac{3}{2})}}{n^2} + \frac{c}{n  } \\
&\le & C_1 \inf_{m \in \cM_n} \left \{ \inf_{t \in \mathcal{S}_m} \| t - \sigma^2 \|_{\pi^X}^2 + \pen(m)  \right\} 
+ C_2 \Delta_n^{4 \beta - 1} + \frac{C_3 \Delta_n^{ 4 \beta -\frac{3}{2}}}{n^2 } + \frac{C_4}{n }.
\end{eqnarray*}

\end{proof}

\subsection{Proof of results on estimation of $g$}{\label{S: proof both}}
In this section we prove the results stated in Section \ref{S: volatility + jumps}.
\subsubsection{Proof of Proposition \ref{prop: estim both}}
\begin{proof}
The proof follows the same scheme than the proof of Proposition \ref{prop: volatility}.
We want to upper bound the empirical risk $\mathbb{E}[\left \|\w{g}_m - g \right \|^2_n]$.
By the definition of $T_{t_i}$ we have that 
$$\gamma_{n, M}(t):= \frac{1}{n} \sum_{i = 0}^{n - 1} (t(X_{t_i})- T_{t_i})^2{\modar \one_A(X_{t_i}) } = \frac{1}{n} \sum_{i = 0}^{n - 1} (t(X_{t_i})- g(X_{t_i}) - (A_{t_i} + B_{t_i} + C_{t_i} + E_{t_i} ))^2{\modar \one_A(X_{t_i}) }$$
\begin{eqnarray*}
\gamma_{n, M}(t)&=& \left \| t - g \right \|_n^2 + \frac{1}{n} \sum_{i = 0}^{n - 1} (A_{t_i} + B_{t_i} + C_{t_i} + E_{t_i})^2 {\modar \one_A(X_{t_i}) }\\
&&- \frac{2}{n} \sum_{i = 0}^{n - 1} (A_{t_i} + B_{t_i} + C_{t_i} + E_{t_i})(t(X_{t_i})- g(X_{t_i}))
{\modar \one_A(X_{t_i}). }
\end{eqnarray*}
As $\w{g}_m$ minimizes $\gamma_{n, M}(t)$, for any $g_m \in \mathcal{S}_m$ it is $\gamma_{n, M}(\w{g}_m) \le \gamma_{n, M}(g_m)$ and therefore
$$\left \| \w{g}_m - g \right \|_n^2 \le \left \| g_m - g \right \|_n^2 + \frac{2}{n} \sum_{i = 0}^{n - 1} (A_{t_i} + B_{t_i} + C_{t_i} + E_{t_i})(\w{g}_m(X_{t_i})- g_m(X_{t_i})).$$
Using Cauchy-Schwarz inequality and the fact that, for $d > 0$, $2 x y \le \frac{x^2}{d} + d y^2$, we get
\begin{eqnarray}
\left \| \w{g}_m - g \right \|_n^2 & \le &\left \| g_m - g \right \|_n^2 + \frac{2d}{n} \sum_{i = 0}^{n - 1} A_{t_i}^2 + \frac{1}{d} \left \| \w{g}_m - g_m \right \|_n^2 +  2 d \sup_{\mathcal{B}_m} \nu_{n,1}^2(t) \nonumber \\
&&+  \frac{1}{d} \left \| \w{g}_m - g_m \right \|_{\pi^X}^2 +2 d \sup_{\mathcal{B}_m} \nu_{n, 2}^2(t) ,
\label{eq: prima di definire omega n 2}
\end{eqnarray}
where $\mathcal{B}_m = \left \{ t \in S_m : \left \| t \right \|_{\pi^X}^2 \le 1\right \}$ and
\begin{equation}
\nu_{n,1} (t) := \frac{1}{n} \sum_{i = 0}^{n - 1} ( B_{t_i} + E_{t_i})t (X_{t_i})
, \quad \nu_{n,2} (t) := \frac{1}{n} \sum_{i = 0}^{n - 1} C_{t_i} t (X_{t_i}).
\label{eq: def nu 2}
\end{equation}
We still denote
$\Omega_n$ the space
on which the norms $\left \| \cdot \right \|_{\pi^X}$ and $\left \| \cdot \right \|_{n}$ are equivalent given by Equation \eqref{eq: omega}. We now act differently to bound the risk on $\Omega_n$ and $\Omega_n^c$. 

\paragraph{Bound of the risk on $\Omega_n$.}
On $\Omega_n$, it is
$$\left \| \w{g}_m - g_m \right \|_{\pi^X}^2 \le 2 \left \| \w{g}_m - g_m \right \|_{n}^2 \le 4 \left \| \w{g}_m - g \right \|_{n}^2 + 4 \left \| g - g_m \right \|_{n}^2,$$
where in the last estimation we have used triangular inequality. Replacing it in \eqref{eq: prima di definire omega n 2} we get
\begin{eqnarray*}
\left\| \w{g}_m - g \right\|_n^2 &\le &  \left\| g_m - g \right\|_n^2 + \frac{2 d}{n} \sum_{i = 0}^{n - 1} A_{t_i}^2  +  2d \sup_{\mathcal{B}_m} \nu_{n,1}^2(t) + 2d \sup_{\mathcal{B}_m} \nu_{n, 2}^2(t) \\
&&+ \frac{6}{d} \left \| \w{g}_m - g \right \|_{n}^2 + \frac{6}{d} \left \| g - g_m \right \|_{n}^2.
\end{eqnarray*}
{ As before, we take $d=12$. It yields}
\begin{equation}
\left \| \w{g}_m - g \right \|_n^2 \le { 3} \left \| g_m - g \right \|_n^2 + \frac{{ 48}}{n} \sum_{i = 0}^{n - 1} A_{t_i}^2 + { 48} \sup_{t \in \mathcal{B}_m} \nu_{n,1}^2(t) + { 48} \sup_{t \in \mathcal{B}_m} \nu_{n, 2}^2(t).
\label{eq: d replaced 2}
\end{equation}
{We now need introduce a orthonormal basis of $S_m$. Hence, we consider $(\tilde{\psi}_k)_k$, an orthonormal basis of $S_m$ for the $L^2_{\pi^X}$ norm, as before.}
 Each $t \in \mathcal{B}_m$ can be written 
$$t = \sum_{l = 1}^{D_m} \alpha_l \tilde{\psi}_l, \qquad \mbox{with } \sum_{l = 1}^{D_m} \alpha_l^2 \le 1.$$

Then, for $j = 1$ and $j=2$,
\begin{align}
&\sup_{t\in \mathcal{B}_m} \nu_{n,j}^2(t) = \sup_{\sum_{l = 1}^{D_m} \alpha_l^2  \le 1} \nu_{n,j}^2\left(\sum_{l = 1}^{D_m} \alpha_l  \tilde{\psi}_l\right) \\
& \le \sup_{\sum_{l = 1}^{D_m} \alpha_l^2  \le 1} \left(\sum_{l = 1}^{D_m} \alpha_l^2  \right)\left(\sum_{l = 1}^{D_m} \nu_{n,j}^2 (\tilde{\psi}_l) \right) = \sum_{l = 1}^{D_m} \nu_{n,j}^2 (\tilde{\psi}_l),
\label{eq: calcolo nu 2}
\end{align}
{where we have also used Cauchy-Schwartz inequality.}
To study the risk we need to evaluate the expected value. From \eqref{eq: d replaced 2}, \eqref{eq: calcolo nu 2} and using the first point of Proposition \ref{prop: size A B C E}, we get
\begin{equation}
\mathbb{E}\left[\left \| \w{g}_m - g \right \|_n^2 \one_{\Omega_n}\right] \le { 3} \mathbb{E}\left[ \left \| g_m - g \right \|_n^2 \right] + c \Delta_n^{1 - \tilde\varepsilon} + { 48} \sum_{l = 1}^{D_m} \mathbb{E}[\nu_{n,1}^2 (\tilde{\psi}_l)] + {48} \sum_{l = 1}^{D_m} \mathbb{E}\left[\nu_{n,2}^2 (\tilde{\psi}_l)\right].
\label{eq: risk Omega unfinished 2}
\end{equation}
By the definition \eqref{eq: def nu 2} of $\nu_{n,1}$ and the points 2 and 3 of Proposition \ref{prop: size A B C E}, it is
\begin{eqnarray*}
\sum_{l = 1}^{D_m} \mathbb{E}[\nu_{n,1}^2 (\tilde{\psi}_l)] &\le&  \frac{c}{n^2} \sum_{i = 0}^{n - 1} \sum_{l = 1}^{D_m} \mathbb{E}\left[\tilde{\psi}^2_l(X_{t_i}) \mathbb{E}[B^2_{t_i} + E^2_{t_i}|\mathcal{F}_{t_i}]\right]  \\
&\leq &
 \frac{c}{n^2} \sum_{i = 0}^{n - 1} \sum_{l = 1}^{D_m} \mathbb{E}\left[\tilde{\psi}^2_l(X_{t_i})(c \sigma_1^4 + \frac{c a_1^4}{\Delta_{n,i}} \sum_{j = 1}^M |\lambda^{(j)}_{t_i}|)\right] 
 \end{eqnarray*}
We observe that the first term in the right hand side here above is
$$\frac{c \sigma_1^4}{n^2} \sum_{i = 0}^{n - 1} \sum_{l = 1}^{D_m} \mathbb{E}[\tilde{\psi}^2_l(X_{t_i}] \le \frac{c D_m}{n}.$$
Regarding the second term, we remark that, { as $\left \| \tilde{\psi}_l \right \|_\infty \le  D_m$ and its norm $2$ is bounded by 1, it is
\begin{eqnarray*}
\mathbb{E}\left[\tilde{\psi}^2_l(X_{t_i})  \frac{c a_1^4}{\Delta_{n,i}} \sum_{j = 1}^M |\lambda^{(j)}_{t_i}|\right] &\le & \frac{c a_1^4}{\Delta_{n,i}} \sum_{j = 1}^M \mathbb{E}[ \tilde{\psi}^{2p}_l(X_{t_i})]^\frac{1}{p} \mathbb{E}[ |\lambda_{t_i}^{(j)}|^q]^\frac{1}{q} \\
&\le & \frac{c a_1^4}{\Delta_{n,i}} \sum_{j = 1}^M D_m^{2 \epsilon} \mathbb{E}[ \tilde{\psi}^{2}_l(X_{t_i})]^\frac{1}{2} \mathbb{E}[|\lambda^{(j)}_{t_i}|^{\frac{1 + \epsilon}{\epsilon}}]^{\frac{\epsilon}{1 + \epsilon}} \\
&\le & \frac{c a_1^4}{\Delta_{n,i}} D_m^{2 \epsilon},
\end{eqnarray*}
where we have used Holder inequality with $p = 1 + \epsilon$, for $\epsilon > 0$ arbitrarily small, and the boundedness of the moments of $\lambda$.} It follows
$$\frac{c}{n^2} \sum_{i = 0}^{n - 1} \sum_{l = 1}^{D_m} \mathbb{E}[\tilde{\psi}^2_l(X_{t_i}) \frac{c a_1^4}{\Delta_{n,i}} \sum_{j = 1}^M |\lambda^{(j)}_{t_i}|] \le \frac{c D_m^{1 + 2 \epsilon} a_1^4}{n \Delta_{n,i}}.$$
Hence, 
\begin{equation}
\sum_{l = 1}^{D_m} \mathbb{E}\left[\nu_{n,1}^2 (\tilde{\psi}_l)\right] \le \frac{c(\sigma_1^4 + a_1^4) D_m^{ 1 + 2 \epsilon}}{n \Delta_{n,i}}.
\label{eq: nu 1}
\end{equation}
In order to evaluate $\mathbb{E}[\nu_{n,2}^2 (\tilde{\psi}_l)]$, the following lemma will be useful:
\begin{lemma}
Suppose that A1-A3 hold true. Then, {for any $\varepsilon > 0$ arbitrarily small, }
$$\text{\rm Var} \left(\frac{1}{n} \sum_{i = 0}^{n - 1} C_{t_i} 
{ \tilde{\psi}_l(X_{t_i})} \right) \le \frac{c {D_m^{2 \varepsilon}}}{n \Delta_n}.$$
\label{lemma: variance Cti}
\end{lemma}
The proof of Lemma \ref{lemma: variance Cti} is in the appendix.
Lemma \ref{lemma: variance Cti} yields
\begin{equation}\label{eq:controlnu2Dm}
\sum_{l = 1}^{D_m} \mathbb{E}\left[\nu_{n,2}^2 (\tilde{\psi}_l)\right] \le \frac{c { D_m^{1 + 2 \varepsilon}}}{n \Delta_n}.
\end{equation}
Replacing the inequality here above and \eqref{eq: nu 1} in \eqref{eq: risk Omega unfinished 2} we get, using also that $\Delta_{n,i} \ge c \Delta_{min}$ and the fact that there exist $c_1$ and $c_2$ for which $c_1 \le \frac{\Delta_n}{\Delta_{min}} \le c_2 $,
$$\mathbb{E}\left[\left \| \w{g}_m - g \right \|_n^2 \one_{\Omega_n}\right] \le { 3} \mathbb{E}\left[ \left \| g_m - g \right \|_n^2 \right] + c \Delta_n^{1 - \tilde\varepsilon} + \frac{c (\sigma_1^4 + a_1^4 + 1) { D_m^{1 + 2 \varepsilon}}}{n \Delta_n}.$$
As the choice $g_m \in \mathcal{S}_m$ is arbitrary, we obtain
\begin{equation}
\mathbb{E}\left[\left \| \w{g}_m - g \right \|_n^2 \one_{\Omega_n}\right] \le { 3} \inf_{t \in \mathcal{S}_m} \left \| t - g \right \|_{\pi^X}^2 + c \Delta_n^{1 - \tilde\varepsilon} + \frac{c (\sigma_1^4 + a_1^4 + 1) {D_m^{1 + 2 \varepsilon}}}{n \Delta_n}.
\label{eq: risk on Omega n 2}
\end{equation}
\paragraph{Bound of the risk on $\Omega_n^c$.} 
Let us set $e = (e_{t_0},\ldots , e_{t_{n - 1}})$, where $e_{t_i} := T_{t_i} - g (X_{t_i}) = A_{t_i} + B_{t_i} + C_{t_i} + E_{t_i}$. Moreover 
$$\Pi_m T = \Pi_m (T_{t_0},\ldots , T_{t_{n - 1}}) = (\w{g}_m (X_{t_0}),\ldots , \w{g}_m (X_{t_{n - 1}})),$$
where $\Pi_m$ is the Euclidean orthogonal projection over $S_m$. Then, according to the projection definition, 
\begin{eqnarray*}
\left \| \w{g}_m - g \right \|_n^2 &=& \left \| \Pi_m T - g \right \|_n^2 = \left \| \Pi_m T - \Pi_m g \right \|_n^2 + \left \| \Pi_m g - g \right \|_n^2 \\
&\leq &\left \| T - g \right \|_n^2 + \left \| g \right \|_n^2 = \left \| e \right \|_n^2 + \left \| g \right \|_n^2.
\end{eqnarray*}
Therefore, from Cauchy -Schwarz inequality,
\begin{eqnarray*}\mathbb{E}[\left \| \w{g}_m - g \right \|_n^2 \one_{\Omega_n^c}] &\le& \mathbb{E}[\left \| e \right \|_n^2 \one_{\Omega_n^c}] + \mathbb{E}[\left \| g \right \|_n^2 \one_{\Omega_n^c}] = \frac{1}{n } \sum_{i = 0}^{n - 1} \mathbb{E}[e_{t_i}^2 \one_{\Omega_n^c}  ] + \frac{1}{n } \sum_{i = 0}^{n - 1} \mathbb{E}[g(X_{t_i})^2 \one_{\Omega_n^c} ]\\
&\leq &\frac{1}{n } \sum_{i = 0}^{n - 1} \mathbb{E}[e_{t_i}^4]^\frac{1}{2} \mathbb{P}(\Omega_n^c)^\frac{1}{2} + \frac{1}{n } \sum_{i = 0}^{n - 1} \mathbb{E}[g(X_{t_i})^4]^\frac{1}{2} \mathbb{P}(\Omega_n^c)^\frac{1}{2}
\end{eqnarray*}
Moreover, using the boundedness of both $a$ and $\sigma$ and the fact that $\mathbb{E}[|\lambda_{t_i}|^4] < \infty$, we obtain $\mathbb{E}[g(X_{t_i})^4] < \infty$. We are left to evaluate $\mathbb{E}[e_{t_i}^4]$. From Proposition \ref{prop: size A B C E} it follows
$$\mathbb{E}[e_{t_i}^4] \le \mathbb{E}[A_{t_i}^4 + B_{t_i}^4+ C_{t_i}^4 + E_{t_i}^4 ] \le c \Delta_n^{1 - \tilde\varepsilon} + c + c + \frac{c}{ \Delta_{n_i}^3} \le \frac{c}{ \Delta_{n}^3}.$$
Putting the pieces together it yields
\begin{equation}
\mathbb{E}[\left \| \w{g}_m - g \right \|_n^2 \one_{\Omega_n^c}] \le \frac{c}{ \Delta_{n}^\frac{3}{2}} \frac{1}{n^2} + \frac{c}{n^2} \le \frac{c}{ n^2 \Delta_{n}^\frac{3}{2}}.
\label{eq: risk on Omega n c 2}
\end{equation}
From \eqref{eq: risk on Omega n 2} and \eqref{eq: risk on Omega n c 2} it follows
\begin{equation*}
\mathbb{E}[\left \| \w{g}_m - g \right \|_n^2 ] \le {3 \,} \mathbb{E}[ \left \| g_m - g \right \|_n^2 ] +  \frac{C_1(\sigma_1^4 + a_1^4 + 1) {D_m^{1 + 2 \varepsilon}}}{n \Delta_n} +  C_2 \Delta_n^{1 - \tilde\varepsilon} + \frac{C_3}{ n^2 \Delta_{n}^\frac{3}{2}}.
\end{equation*}

\end{proof}

\subsubsection{Proof of Theorem \ref{th: estim both adaptive}}
\begin{proof}
{For simplicity in notation we denote $\w{m}_g= \w{m}$ in the proof.}\\
We act again in different way depending on whether or not we are on $\Omega_n$. On $\Omega_n^c$ the proof can be led as before, getting
\begin{equation}
\mathbb{E}\left[\left \| \w{g}_{\w{m}} - g \right \|_n^2 \one_{\Omega_n^c}\right] \le \frac{c}{ n^2 \Delta_{n}^\frac{3}{2}}.
\label{eq: adapt risk omega n c 2}
\end{equation}
Now we investigate what happens on $\Omega_n$. In particular, we analyse what happens on $\mathcal{O} \subset \Omega_n$, a set which will be defined later (see \eqref{eq: def O}). 
By the definition of $\w{m}$ we have
$$\gamma_{n, M} (\w{g}_{\w{m}}) + \pen (\w{m}) \le \gamma_{n, M} (\w{g}_m) + \pen (m)  \le \gamma_{n, M} ({g}_m) + \pen (m)$$
and so, acting as to obtain Equation \eqref{eq: d replaced 2}, we get
\begin{eqnarray*}
\mathbb{E}\left[\left \| \w{g}_{\w{m}} - g \right \|_n^2 \one_{\mathcal{O}}\right] &\le & { 3} \mathbb{E}[\left \| g_m - g \right \|_n^2] + \frac{{ 48}}{n} \sum_{i = 0}^{n - 1} \mathbb{E}[A_{t_i}^2] + { 48} \, \mathbb{E}\left[\sup_{t \in \mathcal{B}_{m, \w{m}}}\nu^2_{n} (t) \one_{\mathcal{O}}\right] \\
&&+ {12} \pen(m) - { 12} \mathbb{E}[\pen(\w{m})],
\end{eqnarray*}
where
$$\nu_{n} (t) := \frac{1}{n} \sum_{i = 0}^{n - 1} (B_{t_i} + C_{t_i} + E_{t_i}) t(X_{t_i}),$$ 
and
$$\mathcal{B}_{m, m'} := \left \{ h \in S_m + S_{m'} : \left \| h \right \|_{\pi^X} \le 1 \right \}.$$
In order to control the term $\mathbb{E}[\sup_{t \in \mathcal{B}_{m, \w{m}}}\nu^2_{n} (t) \one_{\mathcal{O}}]$, we introduce the function $p(m, m')$:
\begin{equation*}
p(m, m') \le \frac{1}{{ 48}}(\pen (m) + \pen (m')).
\end{equation*}
It is  
$$\mathbb{E}\Big[\sup_{t \in \mathcal{B}_{m, \w{m}}} \nu^2_{n} (t)\one_{\mathcal{O}}\Big] \le  \mathbb{E}[ p(m, \w{m})] +  \sum_{m' \in \cM_n} \mathbb{E}\left[\left(\sup_{t \in \mathcal{B}_{m, m'}}\nu_{n}^2(t)- p(m, m') \right)_+\one_{\mathcal{O}}\right] .$$
Replacing it in \eqref{eq: on Omega, intermedio} and using the first point of Proposition \ref{prop: size A B C E} we get
\begin{eqnarray}
\E\left[\left\| \w{g}_{\w{m}} - g \right\|_n^2 \one_{\mathcal{O}} \right] &\le & { 3} \mathbb{E}\left[\left \| {g}_{m} - g \right \|_n^2\right] + c \Delta_n^{1 - \tilde\varepsilon} +  { 48} \mathbb{E}[ p(m, \w{m})]+ { 12} \pen (m) \nonumber \\
&& - { 12} \E[\pen(\w{m})] 
 + { 48} \sum_{m' \in \cM_n} \mathbb{E}\left[\left(\sup_{t \in \mathcal{B}_{m, \w{m}}}\nu_{n}^2(t)- p(m, m')\right)_+\one_{\mathcal{O}}\right].
 \label{eq: da sostituire finale}
 \end{eqnarray}
We have introduced the function $p(m, m')$ with the purpose to use Talagrand inequality on the last term in the right hand side of the equation here above.
We recall the following version of the Talagrand inequality, which has been stated in \cite{Schmisser} and proved by Birg\'e and Massart (1998) \cite{BirMas98}
(corollary 2 p.354) and Comte and Merlev\`ede (2002) \cite{ComMer02} (p.222-223).
\begin{lemma}
Let $T_1,\ldots , T_{\modar p}$ be independent random variables with values in some Polish space $\mathcal{X}$ and $v_p : \mathcal{B}_{m,m'} \rightarrow \mathbb{R}$ such that
$$v_p (r) := \frac{1}{p} \sum_{j = 1}^p [r (T_j) - \mathbb{E}[r(T_j)]]. $$
Then,
\begin{equation}
\mathbb{E}\left[\left(\sup_{r \in \mathcal{B}_{m,m'}} |v_p(r)|^2 - 2H^2\right)_+\right] \le c \left(\frac{v}{p} e^{- c \frac{p H^2}{v}} + \frac{M^2}{p^2} e^{- c \frac{p H}{M}}\right),
\label{eq: Talagrand in Klein Rio}
\end{equation}
with $c$ a universal constant and where
$$\sup_{r \in \mathcal{B}_{m,m'}} \left \| r \right \|_\infty \le M, \quad \mathbb{E}[\sup_{r \in \mathcal{B}_{m,m'}}|v_p (r)|] \le H, \quad \sup_{r \in \mathcal{B}_{m,m'}} \frac{1}{p} \sum_{j = 1}^p \text{\rm Var}(r (T_j)) \le v.$$
\label{lemma: Talagrand in Klein Rio}
\end{lemma}
We observe that in Talagrand lemma here above the random variables $T_1$,\ldots , $T_{\modar p}$ are supposed to be independent.
Starting from our variables we can get independent variables through Berbee's coupling method. We recall it below, it is proved by Viennet in Proposition 5.1 of \cite{Viennet} while an analogous statement in continuous time can be found in \cite{Chapitre 4}. 
\begin{lemma}
Let $(M_t)_{t \ge 0}$ be a stationary and exponentially $\beta$ mixing
process observed at discrete times $0 = t_0 \le t_1 \le\ldots \le t_n = T$. Let $p_n$ and $q_n$ be two integers such that $n = 2 p_n q_n$. For
any $j \in \left \{ 0, 1 \right \}$ and $1 \le k \le p_n$ we consider the random variables
$$U_{k,j} :=(M_{t_{(2(k - 1) + j) q_n + 1}},\ldots , M_{t_{(2 k - 1 + j) q_n}}) .$$
There exist random variables $M^*_{t_0},\ldots , M^*_{t_n}$ such that
$$U^*_{k,j} := (M^*_{t_{(2(k - 1) + j) q_n + 1}},\ldots , M^*_{t_{(2 k - 1 + j) q_n}})$$
satisfy the following properties.
\begin{itemize}
    \item For any $j \in \left \{ 0, 1 \right \}$, the random vectors $U^*_{1,j},\ldots , U^*_{p_n,j}$ are independent.
    \item For any $(j, k) \in \left \{ 0, 1 \right \} \times \left \{ 1,\ldots , p_n \right \}$, $U_{k,j}$ and $U^*_{k,j}$ have the same distribution.
    \item For any $(j, k) \in \left \{ 0, 1 \right \} \times \left \{ 1,\ldots , p_n \right \}$, $\mathbb{P}(U_{k,j} \neq U^*_{k,j} ) \le \beta_M(q_n \Delta_{min})$, where $\beta_M$ is the $\beta$-mixing coefficient of the process $(M_t)$. \end{itemize}
\label{lemma: berbee}
\end{lemma}
	We want to apply Berbee's coupling lemma to the random vectors $M_{t_i}=(B_{t_i}, C_{t_i}, E_{t_i}, X_{t_i})$, that we write as a function of $(X_t, \lambda_t)$, which is stationary and exponentially $\beta$- mixing, as discussed in Section \ref{S: ergodicity}. We define the $\sigma$ algebra
	\begin{equation}
		\tilde{\mathcal{F}}_{t_i} := \sigma(X_s, \lambda_s, s\in (t_i, t_{i + 1}]),
		\label{eq: sigma algebra}
	\end{equation}
	completed with the null sets. Because of the exponentially $\beta$-mixing of $(X_t, \lambda_t)$ we know it is
	$$\beta(\tilde{\mathcal{F}}_{t_i}, \tilde{\mathcal{F}}_{t_j}) \le c e^{- \gamma |t_j - t_i|}.$$
	Writing the dynamic of $\lambda=(\lambda^{(1)},\dots,\lambda^{(d)})$ in the matrix form, $d\lambda_t=-\alpha(\lambda_t-\zeta)dt+c dN_t$, and since $c$ is invertible, we can get $dN_t$ as a
	 function of $d \lambda_t$ and $\lambda_t$. Then, using the invertibility of $\sigma$ with the second line of \eqref{eq: model}, we can write $dW_t$ as a function of $dX_t,X_t,d\lambda_t,$ and $\lambda_t$.
	Now, by the definition of $B_{t_i}$, $ C_{t_i}$ and $E_{t_i}$ it follows that $(B_{t_i} , C_{t_i} , E_{t_i})$ is measurable with respect to $\tilde{\mathcal{F}}_{t_i}$. We can therefore use Berbee's coupling  on $M_{t_i}=(B_{t_i} , C_{t_i} , E_{t_i}, X_{t_i})$. Let $q_n$ be the size of the blocks, which we will specify later. As it may happen that $2q_n$ does not divide $n$, we set 
	 $p_n=\lfloor n/(2q_n) \rfloor $ and remove from the definition of the contrast function \eqref{E:contrast pour g} the data corresponding to the indexes $i \in \{2p_n q_n,\dots,n-1\}$.
	This modification avoids dealing with a last block having a different size, and we can apply Berbee's lemma to  $(M_{t_i})_{i=0,\dots,2p_nq_n}$.	
	 It yields to the construction of variables $(U_{k,j}^*)_{ k\in\{1,\dots ,p_n\} ; j=0,1}$ such that $U_{k,j}^*$ has the same law as $U_{k,j}=(B_{t_{(2(k - 1) + j) q_n + l}},
	C_{t_{(2(k - 1) + j) q_n + l}},
	E_{t_{(2(k - 1) + j) q_n + l}},
	X_{t_{(2(k - 1) + j) q_n + l}})_{l \in \{1,\dots,q_n\}} $, and for $j\in\{0,1\}$, the random variables 
	$(U_{k,j}^*)_{1\le k  \le p_n} $ are independent.
	Let us set
	$$\Omega^*:= \left \{ \omega, \forall j, \forall k, U_{k,j} = U_{k,j}^*  \right \},$$
	by Berbee's coupling lemma it comes that 
	$$\mathbb{P}(\Omega^{*,c}) \le 2 p_n \beta_Z (q_n \Delta_{min}) \le c \frac{n}{q_n} e^{- \gamma q_n \Delta_{min}}.$$
It is enough to take $q_n := \lfloor \frac{5}{\gamma \Delta_{min}} \log n \rfloor $ in \eqref{eq: P berbee} to get 
	\begin{equation}
		\mathbb{P}(\Omega^{*,c}) \le \frac{c}{n^4 \log n}.
		\label{eq: P berbee}
	\end{equation}
	
	For $t \in \mathcal{B}_{m, m'}$, and $(j,k) \in \{0,1\}\times \{1,\dots,p_n\}$ we define both
	\begin{equation} \label{E:def_t_Ukl}
		t(U_{k,j}^*)=\frac{1}{q_n} \sum_{l = 1}^{q_n}(B^*_{t_{(2(k - 1) + j) q_n + l}} + C^*_{t_{(2(k - 1) + j) q_n + l}} + E^*_{ t_{(2(k - 1) + j) q_n + l}}) t(X^*_{t_{(2(k - 1 )+ j) q_n + l}}),
	\end{equation}
	and $t(U_{k, j})$ the analogous quantity based on $U_{k,j}$. 
	
	We want to apply Talagrand inequality on $v^*_n(t) := v_{p_n}^{0, *}(t) + v_{p_n}^{1, *}(t)$, where 
	\begin{equation}
	v_p^{0, *}(t) = \frac{1}{p} \sum_{k = 1}^{p} t(U^*_{k, 0}), \qquad v_p^{1, *}(t) = \frac{1}{p} \sum_{k = 1}^{p} t(U^*_{k, 1}).
	   \label{eq: def vj}
	\end{equation}
	With these definitions, we have on the set $\Omega^*$, $v^*_n(t) =\nu_n(t)$ for all $t \in \mathcal{B}_{m,m'}$. 
	Now we want to compute the constants $M$, $v$ and $H$ as defined in Lemma \ref{lemma: Talagrand in Klein Rio}. The random variables $	t(U_{k,j}^*)$  are not bounded, hence, to compute $M$, we introduce the following set
	\begin{equation}
		\Omega_B:= \left \{ \omega : \forall j, \forall k, \forall t \in  \mathcal{B}_{m, m'}, \, |t(U^*_{k,j})| \le \tilde{c} n^{\varepsilon_0} D^\frac{1}{2} \right \},
		\label{eq: Omega B def}
	\end{equation}
	with $D := D_m + D_{m'}$ and some $\varepsilon_0 > 0$. 
	The following lemma is proven in the appendix.
	\begin{lemma}
		Suppose that A1-A3 hold. Then there exists $c > 0$ such that
		$$\mathbb{P}(\Omega_B^c) \le \frac{c}{n^4}.$$
		\label{lemma: P omega B complementare}
	\end{lemma}
	We introduce bounded version of the random variables $t(U_{k,j}^*)$ by setting for $M>0$,
	\begin{equation*}
		t^{(M)}(U_{k,j}^*)=t(U_{k,j}^*)\vee -M \wedge M.
	\end{equation*}
	With the choice $M:= \tilde{c} n^{\varepsilon_0} D^\frac{1}{2}$, we have on the event $\Omega_B$ that
	$t^{(M)}(U_{k,j}^*)=	t(U_{k,j}^*)$ , $\forall j, \forall k, , \forall t \in  \mathcal{B}_{m, m'}$.
	We set
	\begin{equation}
		\mathcal{O}:= \Omega_n \cap \Omega_B \cap \Omega^*.
		\label{eq: def O}
	\end{equation}
	From \eqref{eq: proba omega n c}, \eqref{eq: P berbee} and Lemma \ref{lemma: P omega B complementare} it follows
	$$\mathbb{P}(\mathcal{O}^c) \le \frac{c}{n^4}.$$
	
	We act on $\mathcal{O}^c$ as we did on $\Omega_n^c$, getting
	\begin{equation}
		\mathbb{E}[\left \| \w{g}_{\w{m}} - g \right \|_n^2 \one_{\mathcal{O}^c}] \le \frac{c}{n^2 \Delta_{n}^\frac{3}{2}}.
		\label{eq: risk O c both}
	\end{equation}

	On the other side, on $\mathcal{O}$ we are really going to use Talagrand's inequality to control 
	\begin{equation}
		\sum_{m' \in \cM_n} \mathbb{E}\left[\left(\sup_{t \in \mathcal{B}_{m, m'}}\nu_{n} (t)^2- p(m, m')\right)_+\one_{\mathcal{O}}\right].  
		\label{E:a controler par Talagrand}
	\end{equation}
	On $\mathcal{O}$, we have $\nu_n(t) = v_{p_n}^{0, *}(t) + v_{p_n}^{1, *}(t)$ and
	\begin{align*}
		v^{j,*}_{p} (t)&=\frac{1}{p}\sum_{k=1}^p t^{(M)}(U^*_{k,j})=
		\frac{1}{p}\sum_{k=1}^p \left(  t^{(M)}(U^*_{k,j}) -\E[t^{(M)}(U^*_{k,j})] \right) + E[t^{(M)}(U^*_{0,0})]
		\\ 
		&= v^{j,*}_{p} (t^{(M)}) + \E[t^{(M)}(U^*_{0,0})],
	\end{align*} 
	{ where $v^{j,*}_{p}$ has been defined in \eqref{eq: def vj}.}

Using $\E[t(U^*_{0,0})]=0$, we deduce 
$$ |\E[t^{(M)}(U^*_{0,0})]| \le \E\left[ |(t-t^{(M)})(U^*_{0,0})| \one_{\Omega_B^c}\right] \le \E\left[ t(U^*_{0,0})^2\right]^{1/2} \P( \Omega_B^c)^{1/2} .$$

We need the following Lemma whose proof is postponed to the Appendix.

\begin{lemma}
		We have $\sup_{t \in \mathcal{B}_{m,m'}}\E\left[ t(U^*_{0,0})^2\right] \le c \frac{D^\delta}{q_n \Delta_n}$, for $\delta$ arbitrarily small and some constant $c$. 
		\label{lemma: variance U00}
	\end{lemma}
	Using Lemma \ref{lemma: P omega B complementare} and Lemma \ref{lemma: variance U00}, we deduce $
	\sup_{t \in \mathcal{B}_{m,m'}} | E[t^{(M)}(U^*_{0,0})] | \le c\frac{D^{\delta/2}}{(q_n \Delta_n)^{1/2}n^2} 
	\le c \frac{D^{\delta/2}}{(\ln(n))^{1/2}n^2} \le c \frac{p(m,m')}{n}$. 
	Hence to control the term \eqref{E:a controler par Talagrand} it is sufficient to get an upper bound, for $n$ large enough on 
	$$\sum_{m' \in \cM_n} \mathbb{E}\left[\left(\sup_{t \in \mathcal{B}_{m, m'}}v^{j,*}_{p} (t^{(M)})^2- \frac{1}{4} p(m, m')\right)_+ \right],$$
	for $j=0,1$. 
	We can apply Lemma \ref{lemma: Talagrand in Klein Rio} to this term. {To this purpose, we need to compute the constants $M$, $v$ and $H$ appearing therein.} By construction, we can use $ M= c n^{ \varepsilon_0} D^\frac{1}{2}$, and by Lemma \ref{lemma: variance U00} and stationarity we can take $v=c\frac{D^\delta}{q_n \Delta_n}$.
	In order to compute $H^2$ we observe it is 
		\begin{multline*}
			\mathbb{E}[\sup_{t \in \mathcal{B}_{m,m'}}|v_p^{j,*} (t^{(M)})|]
			\le  	\mathbb{E}[
			\sup_{t \in \mathcal{B}_{m,m'}}|v_p^{j,*} (t^{(M)})|\one_\mathcal{O}] + 
			\mathbb{E}[\sup_{t \in \mathcal{B}_{m,m'}}|v_p^{j,*} (t^M)|\one_{\mathcal{O}^c} ] ,
			\\   
			\le	\mathbb{E}[\sup_{t \in \mathcal{B}_{m,m'}} |v_p^{j,*} (t)|\one_{\mathcal{O}}] +
			\sup_{t \in \mathcal{B}_{m,m'}}|\E[t^{(M)}(U_{0,0}^*)]| + 	
			M \P(\mathcal{O}^c),  \quad \text{(using $|t^{M}(U_{j,k}) |\le M$),}
			\\
			\le	\mathbb{E}[\sup_{t \in \mathcal{B}_{m,m'}} |v_p^{j,*} (t)|] +
			c\frac{D^{\delta/2}}{q_n \Delta_n n^2}+c\frac{D^{1/2}n^{\varepsilon_0}}{n^4}
			\le  \sqrt{ \mathbb{E}[\sup_{t \in \mathcal{B}_{m,m'}}(v_p^*)^2(t)] }+c \frac{D^{1/2}}{n^2}.
		\end{multline*}
	To find an upper bound for the right hand side here above we act similarly to how we did before \eqref{eq: calcolo nu 2}: we introduce the orthonormal basis $(\bar{\psi}_k)_k$, such that each $t \in \mathcal{B}_{m, m'}$ can be written as the following
	$$t = \sum_{l = 1}^{D} \bar{\alpha}_l \bar{\psi}_l, \qquad \mbox{with } \sum_{l = 1}^{D} \bar{\alpha}_l^2 \le 1.$$
	Similarly to \eqref{eq: calcolo nu 2}, we have 
		\begin{eqnarray*}
		\sup_{t \in \mathcal{B}_{m, m'}} (v_p^{*})^2(t) &=& \sup_{\sum_{l = 1}^{D} \bar{\alpha}_l^2 \le 1} (v_p^{*})^2\left(\sum_{l = 1}^{D} \bar{\alpha}_l \bar{\psi}_l \right) \le \sup_{\sum_{l = 1}^{D} \bar{\alpha}_l^2 \le 1} \left(\sum_{l = 1}^{D} \bar{\alpha}_l^2\right)\left(\sum_{l = 1}^{D} (v_p^{*})^2 (\bar{\psi}_l) \right) \\
		&=& \sum_{l = 1}^{D} (v_p^{j,*})^2  (\bar{\psi}_l).
		\end{eqnarray*}
	Acting exactly as we did in order to get \eqref{eq: nu 1} and Lemma \ref{lemma: variance Cti} on $v_{n,1}^2$ and $v_{n,2}^2$ (as for Equation \eqref{eq:controlnu2Dm}) we obtain
	$$ \sqrt{\mathbb{E}[\sup_{t \in \mathcal{B}_{m,m'}}(v_p^{j,*})^2(t)]} \le c \sqrt{\frac{D^{ 1 + 2 \varepsilon}}{n \Delta_n}}. $$
	In turn we have $ \mathbb{E}[\sup_{t \in \mathcal{B}_{m,m'}} |v_p^{j,*} (t^{(M)})| ] \le c 
	\sqrt{\frac{D^{ 1 + 2 \varepsilon}}{n \Delta_n}}=: H$.

We now use Talagrand inequality as in Lemma \ref{lemma: Talagrand in Klein Rio}. It follows  
	\begin{eqnarray*}
		\mathbb{E}\left[\left(\sup_{t \in \mathcal{B}_{m, \w{m}}}\nu^*_{n} (t)^2- 2H^2\right)_+ \one_{\mathcal{O}}\right] &\le & 
		\frac{D^{\delta}}{p_n q_n \Delta_n} \exp \left(- c \frac{D^{ 1 + 2 \varepsilon} p_n q_n \Delta_n}{n \Delta_n D^\delta}\right) + \frac{c n^{2 \varepsilon_0} D }{ p_n^2} 
		\exp\left(- c\frac{ p_n D^{\frac{1}{2} { + \varepsilon}}}{ \sqrt{n \Delta_n} n^{\varepsilon_0} D^\frac{1}{2}  }\right)\\
		&\le& \frac{c D^\delta}{n \Delta_n } \exp (- \frac{c}{2} D^{1 { + 2 \varepsilon} - \delta}) + \frac{c n^{2 \varepsilon_0} D }{ p_n^2} \exp \left(- \frac{c \sqrt{n} { D^{\varepsilon}}}{2q_n \sqrt{\Delta_n} n^\varepsilon_0}\right),
	\end{eqnarray*}
where we used $\frac{2p_nq_n}{n} \xrightarrow{n\to\infty} 1$. We recall that $q_n =  \lfloor c \frac{\log n}{\Delta_{min}}\rfloor$. We observe that, as $\Delta_{min}$ and $\Delta_n$ differs only from a constant,
	$\frac{c \sqrt{n}}{\sqrt{\Delta_n}q_n n^{\varepsilon_0}} \ge \frac{c \sqrt{n \Delta_n}}{\log n n^{\varepsilon_0}}$. Moreover, it goes to $\infty$ for $n$ going to infinity as we have assumed that $(\log n) n^\varepsilon = o(\sqrt{n \Delta_n})$ for some $\varepsilon$ and the constant $\varepsilon_0$ can be arbitrarily small. Therefore, the second term here above is negligible compared to the first one. It follows, using also the definition of $p(m, \w{m})$, the fact that for $D > 1$ it is 
	 ${ D^\delta e^{- \frac{c}{2}D^{1 + 2 \varepsilon - \delta}} <}D^\delta e^{- \frac{c}{2}D^{1 - \delta}}  < c'' e^{- c'D^{1 - \frac{\delta}{2}}}$
	and the fourth point of Assumption \ref{ass: subspace}, 
	$$ \sum_{m' \in \cM_n} \mathbb{E}\left[\left(\sup_{t \in \mathcal{B}_{m, \w{m}}}\nu^*_{n} (t)^2- p(m, m')\right)_+\right] \le \frac{c''}{n \Delta_n} \sum_{m' \in \cM_n} D^\delta e^{- c'D^{1 - \delta}}  \le  \frac{c'' \, \Sigma (c')}{n \Delta_n}.  $$
	Replacing it in the equivalent of \eqref{eq: da sostituire finale}, considering that we are now on $\mathcal{O}$, it follows
	\begin{eqnarray*}
		\mathbb{E}\left[\left \| \w{g}_{\w{m}} - g \right \|_n^2 \one_{\mathcal{O}}\right] &\le & { 3} \mathbb{E}[\left \| {g}_{m} - g \right \|_n^2] + c \Delta_n^{1 - \tilde\varepsilon} +  { 48} \mathbb{E}[ p(m, \w{m})] \\
		&&+ { 12} \pen (m) - { 12}\mathbb{E}[\pen(\w{m})] + \frac{c}{n \Delta_n}. 
	\end{eqnarray*}
	It provides us, using also \eqref{eq: risk O c both}, 
	\begin{eqnarray*}
		\mathbb{E}[\left \| \w{g}_{\w{m}} - g \right \|_n^2] &\le & {3} \mathbb{E}[\left \| \w{g}_{m} - g \right \|_n^2] + c \Delta_n^{1 - \tilde\varepsilon} + \frac{c}{n^2 \Delta_n^\frac{3}{2}} + c \pen (m)+ \frac{c}{n \Delta_n} \\
		& \leq &
		c_1 \inf_{m \in \cM_n} \left \{ \inf_{t \in \mathcal{S}_m}\| t - g \|_{\pi^X}^2 + \pen(m)  \right \} + C_2 \Delta_n^{1 - \tilde\varepsilon} + \frac{C_3}{n^2 \Delta_n^\frac{3}{2}} + \frac{C_4}{n \Delta_n}.
	\end{eqnarray*}
\end{proof}

\newpage
\appendix

\section{Appendix}
\label{sec:appendix}
For the following proofs, the lemma stated and proved below is a very helpful tool. It provides the size of the increments of both $X$ and $\lambda$.

\begin{lemma}
Suppose that A1-A3 hold. Then, there exist $c_1$ and $c_2$ positive constants such that, 
 for all $t> s$, $|t - s| < 1$ the following hold true
\begin{enumerate}
\item For all $p \ge 2$,  $\mathbb{E}[|X_t - X_s|^p] \le c_1 | t-s |  $.
\item For all $p \ge 2$ and for any $j \in \left \{ 1 , \ldots , M \right \}$, $\mathbb{E}[|\lambda^{(j)}_t - \lambda^{(j)}_s |^p ]\le c_2 | t - s| $.
\item $\mathbb{E}[| \lambda_t - \lambda_s |  | \mathcal{F}_{s} ] \le c_3 |t- s|(1 + |\lambda_s |) $, where $\lambda = (\lambda^{(1)}, \ldots , \lambda^{(M)})$ and $|\cdot|$ stands for the euclidean norm.
\item For any $j \in \left \{ 1 , \ldots , M \right \}$, $\sup_{h \in [0,1]} \mathbb{E}[|\lambda^{(j)}_{ s + h} || \mathcal{F}_{s}] \le |\lambda^{(j)}_s | + c |h| (1 + {\modar  |\lambda_s|}) $.
\end{enumerate}
\label{lemma: incrementi}
\end{lemma}

\begin{proof}
We start proving the first point. From the dynamic \eqref{eq: model} of the process X we have
$$| X_t - X_s|^p \le c\left| \int_s^t b(X_u) du\right|^p + c\left| \int_s^t \sigma(X_u) dW_u\right|^p +  c\left| \int_s^t a(X_{u^-}) \sum_{j = 1}^M dN_u^{(j)}\right|^p = I_1 + I_2 + I_3. $$
From Jensen inequality, the polynomial growth of $b$ and the fact that $X$ has bounded moments it follows
\begin{equation}
\mathbb{E}[I_1] \le c|t-s|^{p-1} \int_s^t \mathbb{E}[| b(X_u)|^p] du \le c|t-s|^p.
\label{eq: stima I1}
\end{equation}
Using Burkholder-Davis-Gundy inequality, Jensen inequality and Assumption \ref{ass: X}.2. on $\sigma$ it is 
\begin{equation}
\mathbb{E}[I_2] \le c\mathbb{E}\left[\left(\int_s^t \sigma^2(X_u) du\right)^\frac{p}{2}\right] \le c|t-s|^{\frac{p}{2}-1} \int_s^t \mathbb{E}[| \sigma (X_u)|^p] du \le c \sigma_1^p |t-s|^\frac{p}{2}.
\label{eq: stima I2}
\end{equation}
{\modch To evaluate $I_3$, Kunita inequality will be useful. We refer to the Appendix of \cite{Protter} for its proof in a general form, while below (A7) on page 52 of \cite{Contrast} can be found an example of its application in a form closer to the one we are going to use. For a compensated Poisson random measure $\tilde{\mu}= \mu- \bar{\mu}$ and a jump coefficient $l(x,z)$, indeed, Kunita inequality provides the following:
\begin{eqnarray*}
\mathbb{E}\left[\left|\int_0^t \int_{\mathbb{R}} l(X_{s^-},z)\tilde{\mu}(ds, dz)\right|^p \right] &\le& c \mathbb{E}\left[\int_0^t \int_{\mathbb{R}} |l(X_{s^-},z)|^p \bar{\mu}(ds, dz)\right] \\
&&+ c \mathbb{E}\left[\left|\int_0^t \int_{\mathbb{R}} l^2(X_{s^-},z) \bar{\mu}(ds, dz)\right|^\frac{p}{2}\right].
\end{eqnarray*}
We remark that, up to change the constant $c$ in the right hand side, the equation here above holds with the measure $\mu$ instead of the compensated one $\tilde{\mu}$. In the sequel we will apply Kunita inequality on the measure $dN_u^{(j)}$ and the compensated one $d\tilde{N}_u^{(j)}$, for $j \in \left \{ 1, ... , M \right \}$. The compensator is in this case $\lambda^{(j)} (u) du$. \\
Using on $I_3$ Kunita inequality together with Jensen inequality and the boundedness of $a$ we get }
\begin{eqnarray}
	\mathbb{E}[I_3] &\le &c \sum_{j = 1}^M \mathbb{E}\left[ \int_s^t |a(X_{u^-})|^p \lambda^{(j)}_u du + \left(\int_s^t a^2(X_{u^-}) \lambda^{(j)}_u du\right)^\frac{p}{2}
	{\modar
		+\left(\int_s^t a(X_{u^-}) \lambda^{(j)}_u du\right)^p}\right]
	\nonumber \\
	&&\leq 
	\sum_{j = 1}^M c|a_1|^p \int_s^t \mathbb{E}[\lambda^{(j)}_u] du + c|a_1|^p |t-s|^{\frac{p}{2}-1}\int_s^t \mathbb{E}[|\lambda^{(j)}_u|^\frac{p}{2}] du
	{\modar +c|a_1|^p |t-s|^{p-1}\int_s^t \mathbb{E}[|\lambda^{(j)}_u|^p] du}
	\nonumber\\
	&&\le c |a_1|^p (|t-s|+ |t-s|^{\frac{p}{2}} ) {\   \le \ } c |a_1|^p |t-s|,
	\label{eq: stima I3}
\end{eqnarray}
{where we have used that $\lambda$ has the moment of any order because of Proposition \ref{prop: Lyapounov}.}
From \eqref{eq: stima I1}, \eqref{eq: stima I2} and \eqref{eq: stima I3}, as $|t-s| < 1$, it follows $\mathbb{E}[|X_t - X_s|^p] \le c_1 | t-s |  $. \\
\textit{Point 2} \\
Concerning the second point, for any $j \in \left \{ 1, \ldots , M \right \} $ it is
$$|\lambda^{(j)}_t - \lambda^{(j)}_s |^p \le c\left|\alpha \int_s^t (\lambda^{(j)} (u) - \zeta_j) du \right|^p + c \left| \int_s^t \sum_{i = 1}^M c_{\modar j,i} dN_u^{(i)}\right|^p.$$
Acting as in the proof of the first point, using as main arguments Jensen inequality, Kunita inequality and the boundedness of the moments of $\lambda$, we easily get the wanted estimation. \\
\textit{Point 3} \\
We consider the dynamic of $\lambda$ gathered in \eqref{eq: model} in matrix form and so we have
$$\lambda_t - \lambda_s  = \alpha \int_s^t (\lambda_u - \zeta) du  +  \int_s^t c dN_u  = : D_s + G_s,$$
where $\lambda_t =(\lambda_t^{(1)}, \ldots , \lambda_t^{(M)})$, $c \in \mathbb{R}^M \times \mathbb{R}^M$.
We start evaluating $D_s$. By adding and subtracting $\lambda_s$ we easily get, denoting as $\mathbb{E}_s[\cdot]$ the quantity $\mathbb{E}[\cdot |\mathcal{F}_s ]$,
$$\mathbb{E}_s[|D_s|] \le c |t-s| (1 + |\lambda_s|) + c \int_s^t \mathbb{E}_s[|\lambda_u - \lambda_s |] ds. $$
On $G_s$ we use compensation formula and we apply the same reasoning as before, getting
$$\mathbb{E}_s[|G_s|] \le \mathbb{E}_s\left[\int_s^t c |\lambda_u | du \right] \le c |t-s||\lambda_s| + c \int_s^t \mathbb{E}_s[|\lambda_u - \lambda_s |] ds. $$
Putting the pieces together it follows
$$\mathbb{E}_s[|\lambda_t - \lambda_s |] \le c |t-s| (1 + |\lambda_s|) + c \int_s^t \mathbb{E}_s[|\lambda_u - \lambda_s |] ds. $$
Finally, Gronwall lemma yields
$$\mathbb{E}_s[|\lambda_t - \lambda_s |] \le c |t-s| (1 + |\lambda_s|) e^c.$$
\textit{Point 4}
We observe that, for any $h \in [0,1]$, 
$$\mathbb{E}_s[|\lambda_{s + h}^{(j)} |] \le |\lambda_s^{(j)} | + \mathbb{E}_s[|\lambda_{s + h}^{(j)} - \lambda_s^{(j)} |] \le |\lambda_s^{(j)} | + c |h| (1 + |\lambda_s|),  $$
where we have used the just showed third point of this lemma.
\end{proof}

\subsection{Proof of Proposition \ref{prop: Lyapounov}}
\begin{proof}
We write $V(x, y) = V_1(x) + V_2(y)$, where $V_1(x) = |x|^m$ for $m$ arbitrarily big and $V_2(y)= e^{\sum_{i,j} m_{i j} | y^{(i j)}| } $. From the definition \eqref{eq: def generator} of $A^{\tilde{z}}$ we have 
$$A^{{ \tilde{Z}}}V = A_1^{{ \tilde{Z}}}V  + A_2^{{ \tilde{Z}}}V,$$
where 
\begin{eqnarray*}
A_1^{{ \tilde{Z}}}V (x,y) &:=& \partial_x V(x,y) b(x) + \frac{1}{2} \sigma^2(x) \partial^2_x V(x,y) +\sum_{j = 1}^M f_j \left(\sum_{k = 1}^M y^{(j k)}\right)[V_1(x + a(x)) - V_1(x)] \\
&=& m |x|^{m - 1} b(x) + \frac{1}{2} \sigma^2(x) m (m - 1) |x|^{m - 2} + \sum_{j = 1}^M f_j \left(\sum_{k = 1}^M y^{(j k)}\right)[|x + a(x)|^m - |x|^m] 
\end{eqnarray*}
is the { jump-diffusion} part and 
\begin{eqnarray*}
A_2^{{ \tilde{Z}}}V (x,y) &:=& A^{{ \tilde{Z}}}V (x,y) - A_1^{{ \tilde{Z}}}V (x,y) \\
&=& - \alpha \sum_{i,j=1}^M y^{(i j)} \partial_{y^{(i j)}} V (x,y) + \sum_{j = 1}^M f_j (\sum_{k = 1}^M y^{(j k)} ) [V_2( y + \Delta_j) - V_2 (y)],
\end{eqnarray*}
 is the { Hawkes} part of the generator. The arguments of the proof of Proposition 4.5 in \cite{8 in DLL} imply that
\begin{equation}
A_2^{{ \tilde{Z}}}V (x,y) = A_2^{{ \tilde{Z}}}V_2 (y) \le - c_1 V_2(y) + c_2 \one_{K_1} (y),
\label{eq: stima A2 Z}
\end{equation}
with $c_1$ and $c_2$ some positive constants and $K_1$ some compact of $\mathbb{R}^{M \times M}$. Moreover, denoting $\bar{f} (y) := \sum_{j = 1}^M f_j (\sum_{k = 1}^M y^{(j k)})$ the total jump rate, it is
$$A_1^{{ \tilde{Z}}}V (x,y) = m |x|^{m - 1} b(x) + \frac{1}{2} \sigma^2(x) m (m - 1) |x|^{m - 2} +\bar{f} (y) [|x + a(x)|^m - |x|^m].$$
From the drift condition on $b$ gathered in the fourth point of Assumption \ref{ass: X} and the boundedness of both $\sigma^2$ and $a$ it follows
\begin{equation}
A_1^{{ \tilde{Z}}}V (x,y) \le -d m |x|^m + c |x|^{m - 2} + \bar{f}(y) (c_1 |x|^{m - 1} + \ldots + c_m).
\label{eq: stima A1 Z}
\end{equation}
We observe that, for any $x$ such that $|x| > r$, $|x|^{m - 2}$ is negligible compared to $|x|^m = V_1(x)$. To study the last term in the right hand side of \eqref{eq: stima A1 Z}, we choose $1 < p < 2$ and $q> 2$ such that $p(m - 1) < m$ (i e $p< 1 + \frac{1}{m - 1}$) and $\frac{1}{p} + \frac{1}{q} = 1$. Then, 
$$\bar{f}(y) (c_1 |x|^{m-1} + \ldots + c_m) \le \frac{c}{p} (c_1 |x|^{ m-1} + \ldots + c_m)^p + \frac{c}{q} \bar{f}(y)^q. $$
The first term is again negligible compared to $|x|^m = V_1(x)$, being $p(m-1) < m$. To estimate the second one we observe that, for each $y \in \mathbb{R}^M$ the total jump rate $\bar{f}(y)$ can be seen as $\sum_{i = 1}^M (\zeta_i +  \sum_{j} y^{(i j)})$ (see page 12 in \cite{DLL}). Therefore, it is 
$$\bar{f}(y) \le \bar{c} + \tilde{c} \sum_{i,j = 1}^M | y^{(i j)}|  \le \bar{c} + \tilde{c}_2 \log (V_2(y)), $$
which is negligible with respect to the negative term of \eqref{eq: stima A2 Z} $- c_1 V_2(y) $.
The same reasoning applies for $\frac{c}{q} \bar{f}(y)^q$. It follows that
$$A_1^{\tilde{z}}V (x,y) \le -d m |x|^m  + o(V_1(x)) + o(V_2(y))$$
which, together with \eqref{eq: stima A2 Z}, conclude the proof { of the first part. Regarding the boundedness of the moments, { we use that the Lyapunov function $V$ admits a finite integral with respect to the stationary probability of $\tilde{Z}=(X,Y)$. As the process $\lambda$  can be recovered as a linear function of $Y$, the ergodicity $\tilde{Z}$ of implies the ergodicity of $Z=(X,\lambda)$ as well, and we have existence of bounded moments of any order for both $X$ and $\lambda$ under the stationary law.}}
\end{proof}

\subsection{Proof of Lemma \ref{lemma: boundpiX}}
\begin{proof}
	We first recall the representation of $\pi^X(x)$ given in the proof of Proposition 3.7 in \cite{DLL}. For $t>0$, let 
	$L_t=\sup\{ s \le t \mid \exists j, ~ \Delta N^{(j)}_s=1 \}$ be the time of the last jump before $t$, with $L_t=0$ if there is no such jump. Then, we have for all $x \in \mathbb{R}$
	$$
	\pi^X(x)=\int_{\mathbb{R}\times\mathbb{R}^{M\times M}}
	\pi(d z) \E_z\left[p_{t-L_t}(X_{L_t},x) \right],
	$$
	where $(p_s)_{s>0}$ is the family of transition densities associated to the stochastic differential equation $dX_t=b(X_t)dt +\sigma(X_t) d W_t$.
	From Proposition 1.2 of \cite{Gobet02}, we know that there exist constants $c$, $C$ such  that for all $s>0$, $(u,x)\in \mathbb{R}^2$,
	$$
	p_{s}(u,x) \le c s^{-1/2} e^{ - \frac{(x-u)^2}{Cs} } e^{Csu^2}.
	$$
We deduce that
 $p_{s}(u,x) \le  c s^{-1/2} e^{ - \frac{(x-u)^2}{Cs} } e^{2Cs(x-u)^2+2Csx^2}
 \le c s^{-1/2} e^{- \frac{(x-u)^2}{2Cs} } e^{2Csx^2}$, if $s$ is smaller than $1/(2C)$. Choosing $t < 1/(2C)$, it yields
 \begin{equation}
 \pi^X(x) \le c \int_{\mathbb{R}\times\mathbb{R}^{M\times M}} \pi(d z)
 \E_z\left[\frac{1}{\sqrt{t-L_t}}\right]e^{2 C t x^2}.
 \label{E:lemma_maj_piX_formuleDLL} 
 \end{equation}
 We now give an upper bound for $\E_z\left[\frac{1}{\sqrt{t-L_t}}\right]$. Writing $(t-L_t)^{-1/2}=\frac{1}{2} \int_0^t \frac{1_{\{s \le L_t\}} ds}{(t-s)^{3/2}} + t^{-1/2}$, we have
 \begin{equation}
  \E_z\left[\frac{1}{\sqrt{t-L_t}}\right]  \le \frac{1}{2} \int_0^t \frac{P_z(L_t \ge s) }{(t-s)^{3/2}}ds + \frac{1}{t^{1/2}}.
  \label{E:lemma_maj_piX_Ez}
 \end{equation}
From the definition of $L_t$, it is
$P_z(L_t \ge s)=P_z(\exists u \in [s,t], \exists j ~:~ \Delta N^{(j)}_u=1)$ which implies
\begin{align*}
	P_z(L_t \ge s) &\le P_z\left(\sum_{j=1}^M \int_{[s,t]} d N_u^{(j)} \ge 1\right) \le \E_z \left[ \sum_{j=1}^M \int_{[s,t]} d N_u^{(j)} \right]
	\\
	& \le \E_z \left[\sum_{j=1}^M \int_s^t \lambda_u^{(j)} du\right]= \sum_{j=1}^M \int_s^t  \E_z [\lambda_u^{(j)}] du.
\end{align*}  
Using Lemma \ref{lemma: incrementi} we get that if $u\le 1 $, $\E_z [\lambda_u^{(j)}] \le c(1+|\lambda_0(z)|)$, where $\lambda_0^j(z)=f_j(\sum_{i=1}^My^i)$ for $z=(x,y)$. Thus, if $t$ is  chosen smaller than $1$, we have 
$P_z(L_t \ge s) \le c(1+|\lambda_0(z)|)(t-s)$. Using this control with \eqref{E:lemma_maj_piX_Ez} we deduce,
$$
 \E_z\left[\frac{1}{\sqrt{t-L_t}}\right]\le \frac{c}{2} \int_0^t \frac{(1+|\lambda_0(z)|)}{(t-s)^{1/2}}ds + \frac{1}{t^{1/2}} \le  \frac{c}{\sqrt{t}} (1+|\lambda_0(z)|).
$$
From \eqref{E:lemma_maj_piX_formuleDLL}, we obtain
$\pi^X(x)\le \frac{ce^{2 C t x^2}}{\sqrt{t}} \int_{\mathbb{R}\times\mathbb{R}^{M\times M}} \pi(d z)
(1+|\lambda_0(z)|)$. By Proposition \ref{prop: Lyapounov}, we know that the intensity $\lambda$ has finite moments of any order under the stationary measure, and thus $ \int_{\mathbb{R}\times\mathbb{R}^{M\times M}} \pi(d z)
(1+|\lambda_0(z)|)\le c < \infty$. This gives $\pi^X(x)\le \frac{ce^{2 C t x^2}}{\sqrt{t}}$ for any sufficiently small $t$, and the lemma follows. 
\end{proof}
\begin{remark}
The proof of Lemma \ref{lemma: boundpiX} heavily relies  on the integrability near zero of the supremum of the heat kernel  and is thus limited to the dimension $1$. We do not know if it is possible to extend this result to higher dimension for the process $X$. However, it is certainly possible to extend this proof to more general situations, as for instance the case where the jump intensity depends on $X$. 
\end{remark}

\subsection{Proof of Lemma \ref{lemma: proba varphi -1}}
\begin{proof}
By the definition of $\varphi$, for any $k \ge 1$ $|\varphi_{\Delta_{n,i}^\beta}(\Delta_i X) - 1|^k$ is different from zero only if $|\Delta_i X| > \Delta_{n,i}^\beta$. Therefore, 
\begin{eqnarray}
\mathbb{E}[|\varphi_{\Delta_{n,i}^\beta}(\Delta_i X) - 1|^k] &\le & c \mathbb{E}[\one_{\left \{ |\Delta_i X| > \Delta_{n,i}^\beta \right \}}] \nonumber\\
&=& c \mathbb{E}\left[\one_{\left \{ |\Delta_i X| > \Delta_{n,i}^\beta, |J_{t_i}| \le \frac{\Delta_{n,i}^\beta}{2} \right \}}\right] + c \mathbb{E}\left[\one_{\left \{ |\Delta_i X| > \Delta_{n,i}^\beta, |J_{t_i}| > \frac{\Delta_{n,i}^\beta}{2} \right \}}\right].
\label{eq: strat varphi -1}
\end{eqnarray}
We denote as $\Delta_i X^c$ the increment of the continuous part of $X$, which is
$$\Delta_i X^c := X_{t_{i + 1}}^c - X_{t_i}^c = \int_{t_i}^{t_{i + 1}} b (X_s) ds + Z_{t_i}.$$
The first term in the right hand side of \eqref{eq: strat varphi -1} is \begin{equation}
c \mathbb{E}\left[\one_{\left \{ |\Delta_i X^c| >\frac{\Delta_{n,i}^\beta}{2} \right \}}\right] = c\mathbb{P}(|\Delta_i X^c| >\frac{\Delta_{n,i}^\beta}{2}) \le \frac{c \, \mathbb{E}\left[|\Delta_i X^c|^r\right]}{\Delta_{n,i}^{\beta r}} \le c \Delta_{n,i}^{r (\frac{1}{2} - \beta)}, 
\label{eq: stima X^c}
\end{equation}
where we have used Markov inequality and a classical estimation for the continuous increments of $X$ (see for example point 6 of Lemma 1 in \cite{Unbiased}). In order to evaluate the second term in the right hand side of \eqref{eq: strat varphi -1}, instead, we have to introduce the set
$$N_{i,n}:= \left \{ \sum_{j = 1}^M |\Delta_i N^{(j)}| := \sum_{j = 1}^M |N^{(j)}_{t_{i + 1}} - N^{(j)}_{t_i}| \le \frac{4 \Delta_{n,i}^\beta}{a_1}  \right \}.$$
We observe that, on $N_{i,n}^c$, there exists $j \in \left \{ 1, \ldots , M \right \}$ such that $|\Delta_i N^{(j)}| \neq 0$. Therefore,
\begin{equation}
\mathbb{P} (N_{i,n}^c) \le { \sum_{j=1}^M}  \mathbb{P}(|\Delta_i N^{(j)}| \ge 1)  \le { \sum_{j=1}^M} \mathbb{E}[|\Delta_i N^{(j)}|]   \le c  { M} \Delta_{n,i} .
\label{eq: proba Nin c}
\end{equation}
On $N_{i,n}$, instead, $\forall j$ $|\Delta_i N^{(j)}| = 0$ and so $(N_{i,n}) \cap \left \{ |J_{t_i}| > \frac{\Delta_{n,i}^\beta}{2} \right \} = \emptyset$.
It follows that the second term in the right hand side of \eqref{eq: strat varphi -1} is
$$c \mathbb{E}\left[\one_{\left \{ |\Delta_i X| > \Delta_{n,i}^\beta, |J_{t_i}| > \frac{\Delta_{n,i}^\beta}{2}, N_{i,n} \right \}}\right] + c \mathbb{E}[\one_{\left \{ |\Delta_i X| > \Delta_{n,i}^\beta, |J_{t_i}| > \frac{\Delta_{n,i}^\beta}{2}, N_{i,n}^c \right \}}] \le c \mathbb{P}( N_{i,n}^c) \le c\Delta_{n,i}.$$
Putting the pieces together, as $r$ is arbitrary, it follows
$$\mathbb{E}\left[|\varphi_{\Delta_{n,i}^\beta}(\Delta_i X) - 1|^k\right]  \le c \Delta_{n,i}.$$
\end{proof}

\subsection{Proof of Lemma \ref{lemma: esp sauts}}
\begin{proof}
Again, we act differently depending on whether the jumps are big or not:
\begin{equation}
	\mathbb{E}[|J_{t_i}|^q \varphi^k_{\Delta_{n,i}^\beta}(\Delta_i X)] = \mathbb{E} \left[|J_{t_i}|^q \varphi^k_{\Delta_{n,i}^\beta}(\Delta_i X)\one_{\left \{ |J_{t_i}| > {\modar 3\Delta_{n,i}^\beta}\right \}}\right] + \mathbb{E} \left[|J_{t_i}|^q \varphi^k_{\Delta_{n,i}^\beta}(\Delta_i X)\one_{\left \{ |J_{t_i}| \le {\modar 3\Delta_{n,i}^\beta}\right \}}\right].
	\label{eq: salti inizio}
\end{equation}
By the definition of $\varphi$ it is different from $0$ only if $|\Delta_i X| \le 2 \Delta_{n,i}^\beta$. As $\Delta_i X = \Delta_i X^c + J_{t_i}$, it is
$$\mathbb{E}\left[|J_{t_i}|^q \varphi^k_{\Delta_{n,i}^\beta}(\Delta_i X)\one_{\left \{ |J_{t_i}| > {\modar 3\Delta_{n,i}^\beta}\right \}}\right] \le \mathbb{E} \left[|J_{t_i}|^q \one_{\left \{ |\Delta_i X^c| > {\modar\Delta_{n,i}^\beta}\right \}}\right] \le \mathbb{E} [|J_{t_i}|^{q p_1}]^\frac{1}{p_1} \mathbb{E}\left[\one_{\left \{ |\Delta_i X^c| > {\modar \Delta_{n,i}^\beta}\right \}}\right]^\frac{1}{p_2}  $$
$$\le c \Delta_{n,i}^{\frac{1}{p_1}} \, \Delta_{n,i}^{\frac{r}{p_2}(\frac{1}{2} - \beta) } { \le} \,  c \Delta_{n,i}^{r(\frac{1}{2} - \beta) - \varepsilon},$$
where we have used first of all Holder inequality and then Kunita inequality and \eqref{eq: stima X^c}. We remark it is possible to use Kunita inequality only for $q p_1 \ge 2$. However, as in the estimation here above the power of $\Delta$ is arbitrarily large, we can always choose $p_1$ such that, for any $q \ge 1$, $q p_1$ is bigger than $2$. In order to evaluate the second term of \eqref{eq: salti inizio}, we introduce again the set $N_{i,n}$ {\modar defined in the proof of Lemma \ref{lemma: proba varphi -1}.} On $N_{i,n}$ the increments $\Delta_i N^{(j)}$ are null and so $|J_{t_i}| = 0$. On $N_{i,n}^c$ instead, using also \eqref{eq: proba Nin c}, we have
$$\mathbb{E}\left[|J_{t_i}|^q \varphi^k_{\Delta_{n,i}^\beta}(\Delta_i X)\one_{\left \{ |J_{t_i}| \le {\modar 3\Delta_{n,i}^\beta}, N_{i,n}^c \right \}}\right] \le c \Delta_{n,i}^{\beta q} \mathbb{P}(N_{i,n}^c) \le c \Delta_{n,i}^{1 + \beta q}.$$
By the arbitrariness of $r$ it follows
$$\mathbb{E} [|J_{t_i}|^q \varphi^k_{\Delta_{n,i}^\beta}(\Delta_i X)] \le c \Delta_{n,i}^{1 + \beta q},$$
as we wanted.
\end{proof}

\subsection{Proof of Proposition \ref{lemma: size A B E}}
\begin{proof}
As the second point is useful in order to prove the first one, we start proving point 2. \\
\textit{Point 2} \\
By definition we know that $B_{t_i}$ is centered. 
In the sequel we denote as $\mathbb{E}_i[\cdot]$ the conditional expected value $\mathbb{E}[\cdot | \mathcal{F}_{t_i} ]$.
Regarding the second moment, it is 
$$\mathbb{E}_i[B_{t_i}^2] \le \frac{1}{\Delta_{n,i}^2} \mathbb{E}_i\left[Z_{t_i}^4 + (\int_{t_i}^{t_{i + 1}} \sigma^2(X_s) ds )^2\right] \le \frac{c}{\Delta_{n,i}^2} \mathbb{E}_i \left[\left(\int_{t_i}^{t_{i + 1}} \sigma^2(X_s) ds \right)^2\right] \le c \sigma_1^4$$
where we have used, sequentially, BDG inequality, Jensen inequality and the boundedness of $\sigma$.
Using the same arguments we show the following:
$$\mathbb{E}_i[B_{t_i}^4] \le \frac{1}{\Delta_{n,i}^4} \mathbb{E}_i\left[Z_{t_i}^8 + \left(\int_{t_i}^{t_{i + 1}} \sigma^2(X_s) ds \right)^4\right] \le \frac{c}{\Delta_{n,i}^4} \mathbb{E}_i\left[\left(\int_{t_i}^{t_{i + 1}} \sigma^2(X_s) ds \right)^4\right] \le c \sigma_1^8.$$
\textit{Point 1} \\
We analyse the behaviour of
$$\tilde{A}_{t_i} =  \sigma^2(X_{t_i}) (\varphi_{\Delta_{n,i}^\beta}(\Delta_i X) - 1) + A_{t_i} \varphi_{\Delta_{n,i}^\beta}(\Delta_i X) + B_{t_i}(\varphi_{\Delta_{n,i}^\beta}(\Delta_i X) - 1).$$
From Holder inequality, the boundedness of $\sigma$ and a repeated use of Lemma {\modchi \ref{lemma: proba varphi -1}} we get
$$\mathbb{E}[\tilde{A}_{t_i}^2] \le c \sigma_1^4 \Delta_{n,i} + \mathbb{E}[A_{t_i}^2 \varphi^2_{\Delta_{n,i}^\beta}(\Delta_i X)] + \mathbb{E}[B_{t_i}^{2 p}]^{\frac{1}{p}} c \Delta_{n,i}^{\frac{1}{q}}. $$
We evaluate the moments of $B_{t_i}$ acting as in the proof of the first point and we choose $p$ big and $q$ next to $1$, getting 
\begin{equation}
\mathbb{E}[\tilde{A}_{t_i}^2] \le c \sigma_1^4 \Delta_{n,i} + \mathbb{E}[A_{t_i}^2 \varphi^2_{\Delta_{n,i}^\beta}(\Delta_i X)] + c \sigma_1^2 \Delta_{n,i}^{1 - \tilde{\varepsilon}},
\label{eq: Atilde A}
\end{equation}
for $\tilde{\varepsilon} > 0$ arbitrarily small. We are left to study $A_{t_i}^2\varphi^2_{\Delta_{n,i}^\beta}$. From its definition, recalling that $\varphi$ is a bounded function, we obtain
\begin{eqnarray*}
\mathbb{E}[A_{t_i}^2 \varphi^2_{\Delta_{n,i}^\beta}(\Delta_i X)]  &\le & \frac{c}{\Delta_{n,i}^2} \mathbb{E}\left[\left(\int_{t_i}^{t_{i + 1}} b(X_s) ds \right)^4\right] + \frac{c}{\Delta_{n,i}^2} \mathbb{E}\left[\left(Z_{t_i} + J_{t_i}\right)^2\left(\int_{t_i}^{t_{i + 1}} b(X_s) - b(X_{t_i}) ds \right)^2\right] \\
&&+ \frac{c}{\Delta_{n,i}^2} \mathbb{E}\left[\left(\int_{t_i}^{t_{i + 1}} \sigma^2(X_s) - \sigma^2(X_{t_i}) ds \right)^2\right] + 4 \mathbb{E}[b^2(X_{t_i})Z_{t_i}^2] =: \sum_{j = 1}^4 I_j.
\end{eqnarray*}
Using Jensen inequality, the polynomial growth of $b$ and the existence of bounded moments of $X$ we get
\begin{equation}
I_1 \le \frac{c}{\Delta_{n,i}^2} \Delta_{n,i}^3 \int_{t_i}^{t_{i + 1}} \mathbb{E}[b^4(X_s)] ds \le c \Delta_{n,i}^2.
\label{eq: I1 (carre)}
\end{equation}
On $I_2$ we use first of all Holder inequality. Then, on the first we use B.D.G. and Kunita inequalities, as in \eqref{eq: stima I2} and \eqref{eq: stima I3}, while on the second the finite increments theorem, the boundedness of $b'$ and the first point of Lemma \ref{lemma: incrementi}:
\begin{eqnarray}
I_2 &\le & \frac{c}{\Delta_{n,i}^2} \mathbb{E}[(Z_{t_i} + J_{t_i})^{ 8}]^\frac{1}{2} \mathbb{E}\left[\left(\int_{t_i}^{t_{i + 1}} b(X_s) - b(X_{t_i}) ds \right)^4\right]^\frac{1}{2}
\nonumber
\\
&\le & \frac{c}{\Delta_{n,i}^2}  \Delta_{n,i}^\frac{1}{2} \Delta_{n,i}^{\frac{3}{2}} \mathbb{E}\left[\int_{t_i}^{t_{i + 1}} c |X_s - X_{t_i}|^4 ds \right]^\frac{1}{2} \le c \Delta_{n,i}. 
\label{eq: I2 (carre)}
\end{eqnarray}
In order to study the behaviour of $I_3$, Jensen inequality, the finite increment theorem, the boundedness of the derivative of $\sigma^2$ and the first point of Lemma \ref{lemma: incrementi} will be once again useful.
\begin{equation}
I_3 \le \frac{c}{\Delta_{n,i}^2} \Delta_{n,i} \mathbb{E}\left[\int_{t_i}^{t_{i + 1}} c |X_s - X_{t_i}|^2 ds\right] \le c \Delta_{n,i}.
\label{eq: I3 (carre)}
\end{equation}
From Holder inequality, the polynomial growth of $b$, the boundedness of the moments of $X$ and BDG inequality we obtain
\begin{equation}
I_4 \le c \mathbb{E}[b(X_{t_i})^4]^\frac{1}{2} \mathbb{E}[Z_{t_i}^4]^\frac{1}{2} \le c \Delta_{n,i}.
\label{eq: I4 (carre)}
\end{equation}
Putting the pieces together it follows that, for any $\tilde{\varepsilon} > 0$,
$$\mathbb{E}[\tilde{A}_{t_i}^2] \le c \Delta_{n,i}^{1 - \tilde{\varepsilon}} .$$
We now evaluate $\mathbb{E}[\tilde{A}_{t_i}^4]$. Acting as above \eqref{eq: Atilde A} it easily follows
$$\mathbb{E}[\tilde{A}_{t_i}^4] \le c \Delta_{n,i}^{1 - \tilde{\varepsilon}} + \mathbb{E}[A_{t_i}^4 \varphi_{\Delta_{n,i}^\beta}(\Delta_i X)].$$
Replacing the definition of $A_{t_i}$ we get that $\mathbb{E}[A_{t_i}^4 \varphi^4_{\Delta_{n,i}^\beta}(\Delta_i X)]$ is again the sum of 4 terms, that we now denote as $\tilde{I}_1, \ldots, \tilde{I}_4$. Using exactly the same arguments as in the study of $\mathbb{E}[A_{t_i}^4 \varphi^4_{\Delta_{n,i}^\beta}(\Delta_i X)]$ we easily get
$$\tilde{I}_1 \le \frac{c}{\Delta_{n,i}^4} \Delta_{n,i}^7 \int_{t_i}^{t_{i + 1}} \mathbb{E}[b^8(X_s)] ds \le c \Delta_{n,i}^4,$$
\begin{eqnarray*}
\tilde{I}_2 &\le & \frac{c}{\Delta_{n,i}^4} \mathbb{E}[(Z_{t_i} + J_{t_i})^4]^\frac{1}{2} \mathbb{E}\left[\left(\int_{t_i}^{t_{i + 1}} b(X_s) - b(X_{t_i}) ds \right)^8\right]^\frac{1}{2}\\
&\le &\frac{c}{\Delta_{n,i}^4} (\Delta_{n,i} + \Delta_{n,i}^\frac{1}{2}) \Delta_{n,i}^{\frac{7}{2}} \mathbb{E}\left[\int_{t_i}^{t_{i + 1}} c |X_s - X_{t_i}|^8 ds \right]^\frac{1}{2} \le c \Delta_{n,i}, 
\end{eqnarray*}
$$\tilde{I}_3 \le \frac{c}{\Delta_{n,i}^4} \Delta_{n,i}^3 \mathbb{E}\left[\int_{t_i}^{t_{i + 1}} c |X_s - X_{t_i}|^4 ds\right] \le c \Delta_{n,i}, $$
$$\tilde{I}_4 \le \mathbb{E}[b(X_{t_i})^8]^\frac{1}{2} \mathbb{E}[Z_{t_i}^8]^\frac{1}{2} \le c \Delta_{n,i}^2.$$
The four equations here above provide the wanted result. \\

\textit{Point 3} \\
To prove the estimations on the jumps gathered in the third point of Proposition \ref{prop: volatility} we repeatedly use Lemma \ref{lemma: esp sauts}. Using also Holder inequality with $p$ big and $q$ next to $1$, BDG inequality, the polynomial growth of $b$ and the boundedness of the moments of $X$ it is
\begin{eqnarray*}
\mathbb{E}[|E_{t_i}|\varphi_{\Delta_{n,i}^\beta}(\Delta_i X)] &\le & c \mathbb{E}[|b (X_{t_i})||J_{t_i}|\varphi_{\Delta_{n,i}^\beta}(\Delta_i X)] + \frac{c}{\Delta_{n,i}} \mathbb{E}[|Z_{t_i}||J_{t_i}|\varphi_{\Delta_{n,i}^\beta}(\Delta_i X)] \\
&&+ \frac{c}{\Delta_{n,i}} 
\mathbb{E} \left[|J_{t_i}|^2\varphi_{\Delta_{n,i}^\beta}(\Delta_i X)\right] 
\\
&\le & c \mathbb{E}[|b (X_{t_i})|^p]^\frac{1}{p} \mathbb{E}[|J_{t_i}|^q \varphi^q_{\Delta_{n,i}^\beta}(\Delta_i X)]^\frac{1}{q} + \frac{c}{\Delta_{n,i}} \mathbb{E}[|Z_{t_i}|^p]^\frac{1}{p} \mathbb{E}[|J_{t_i}|^q \varphi^q_{\Delta_{n,i}^\beta}(\Delta_i X)]^\frac{1}{q} \\
&&+ \frac{c}{\Delta_{n,i}} \Delta_{n,i}^{1 + 2 \beta}
\end{eqnarray*}
thus, because, as $\beta \in (0, \frac{1}{2})$, we can always find an $\varepsilon > 0$ such that $\frac{1}{2} + \beta - \varepsilon > 2 \beta$. { Hence, we set $1/q = 1 - \varepsilon$}. It comes
$$\mathbb{E}[|E_{t_i}|\varphi_{\Delta_{n,i}^\beta}(\Delta_i X)] \le c\Delta_{n,i}^{\frac{1}{q} + \beta} + \frac{c}{\Delta_{n,i}} \Delta_{n,i}^{\frac{1}{2}} \Delta_{n,i}^{\frac{1}{q} + \beta} + c\Delta_{n,i}^{2 \beta} = c\Delta_{n,i}^{1 + \beta - \varepsilon} + c\Delta_{n,i}^{\frac{1}{2} + \beta - \varepsilon} + c\Delta_{n,i}^{2 \beta} =  c\Delta_{n,i}^{2 \beta}.$$

In analogous way we obtain 
\begin{eqnarray*}
\mathbb{E}[|E_{t_i}|^2\varphi_{\Delta_{n,i}^\beta}(\Delta_i X)]
&\le &c \mathbb{E}[|b (X_{t_i})|^{2p}]^\frac{1}{p} \mathbb{E}[|J_{t_i}|^{2q} \varphi^q_{\Delta_{n,i}^\beta}(\Delta_i X)]^\frac{1}{q} + \frac{c}{\Delta_{n,i}^2} \mathbb{E}[|Z_{t_i}|^{2p}]^\frac{1}{p} \mathbb{E}[|J_{t_i}|^{2q} \varphi^q_{\Delta_{n,i}^\beta}(\Delta_i X)]^\frac{1}{q} \\
&&+ \frac{c}{\Delta_{n,i}^2} \mathbb{E}[|J_{t_i}|^{4} \varphi_{\Delta_{n,i}^\beta}(\Delta_i X)] \\
&\le &c\Delta_{n,i}^{\frac{1}{q} + 2\beta} + \frac{c}{\Delta_{n,i}
^2} \Delta_{n,i} \Delta_{n,i}^{\frac{1}{q} + 2\beta} +  \frac{c}{\Delta_{n,i}
^2} \Delta_{n,i}^{1 + 4 \beta} \\
&=& c\Delta_{n,i}^{1 +2\beta - \varepsilon} + c\Delta_{n,i}^{2 \beta - \varepsilon} + c\Delta_{n,i}^{4 \beta - 1} =  c\Delta_{n,i}^{4 \beta - 1},
\end{eqnarray*}
where the last inequality is, again, consequence of the fact that we can always find $\varepsilon > 0$ for which $2 \beta - \varepsilon > 4 \beta - 1$. Finally, acting as before, 
$$\mathbb{E}[|E_{t_i}|^4 \varphi_{\Delta_{n,i}^\beta}(\Delta_i X)] \le c\Delta_{n,i}^{1 +4\beta - \varepsilon} + \frac{c}{\Delta_{n,i}^4} \Delta_{n,i}^2 \Delta_{n,i}^{1 + 4 \beta - \varepsilon} + \frac{c}{\Delta_{n,i}^4} \Delta_{n,i}^{1 + 8 \beta} =  c\Delta_{n,i}^{8 \beta - 3}.$$

\end{proof}

\subsection{Proof of Proposition \ref{prop: size A B C E}}
\begin{proof}
\textit{Point 1} \\
Regarding the first point, we first introduce $\tilde{b}(X_s) := b(X_s) + a(X_{s^-}) \sum_{j = 1}^M \lambda_s^{(j)} ds$. We observe that, as $b$ has polynomial growth, $a$ is bounded and both $\lambda$ and $X$ have bounded moments of any order, then $\tilde{b}$ has bounded moments of any order as well. Recalling that $A_{t_i}$ is given as in \eqref{eq: Ati} we can denote 
$$A_{t_i} =: \sum_{j = 1}^7 \bar{I}_j.$$
Replacing $\tilde{b}$ with $b$, we already know from \eqref{eq: I1 (carre)}, \eqref{eq: I2 (carre)}, \eqref{eq: I3 (carre)} and \eqref{eq: I4 (carre)} that
\begin{equation}
\mathbb{E}[\bar{I}_1^2 + \bar{I}_2^2 + \bar{I}_3^2 + \bar{I}_6^2] \le c \Delta_{n,i}.
\label{eq: bar I start}
\end{equation}
We now consider $\bar{I}_4$. From Assumption \ref{ass: X} we know the function $a$ is Lipschitz and with bounded derivative. Therefore, we use the finite increments theorem followed by the first point of Lemma \ref{lemma: incrementi}. It provides us, using also Jensen inequality and Holder inequality with $q$ big and $p$ next to $1$, 
\begin{eqnarray}
\mathbb{E}[\bar{I}_4^2] &\le& \frac{c}{\Delta_{n,i}^2} \Delta_{n,i} \int_{t_i}^{t_{i + 1}} \mathbb{E}\left[(a^2(X_s) - a^2(X_{t_i}))^2\left(\sum_{j = 1}^M \lambda_s^{(j)}\right)^2\right] ds
\nonumber \\
&\le & \frac{c}{\Delta_{n,i}} \int_{t_i}^{t_{i + 1}} \mathbb{E}[(a^2(X_s) - a^2(X_{t_i}))^{2p}]^\frac{1}{p} \mathbb{E}\left[\left(\sum_{j = 1}^M \lambda_s^{(j)} \right)^{2q}\right]^\frac{1}{q} ds \nonumber  \\
&\le & \frac{c}{\Delta_{n,i}} \int_{t_i}^{t_{i + 1}} \Delta_{n,i}
^{\frac{1}{p}} ds \le c \Delta_{n,i}^{1 - \tilde{\varepsilon}},
\label{eq: bar I 4}
\end{eqnarray}
where we have also used the boundedness of the moments of $\lambda$ { and set $1 / p=1-\tilde{\varepsilon}$}. On $\bar{I}_5$ we use that $a(x) \le a_1$ and the second point of Lemma \ref{lemma: incrementi}, getting 
\begin{equation}
\mathbb{E}[\bar{I}_5^2] \le \frac{c}{\Delta_{n,i}^2} \Delta_{n,i} \int_{t_i}^{t_{i + 1}} \sum_{j = 1}^M \mathbb{E}[(\lambda_s^{(j)} - \lambda_{t_i}^{(j)})^2] ds \le c \Delta_{n,i}.
\label{eq: bar I 5}
\end{equation}
To conclude the proof of the bound on $\mathbb{E}[A_{t_i}^2]$ we are left to evaluate $\bar{I}_7$. We do that through Holder and Kunita inequalities. It yields 
\begin{equation}
\mathbb{E}[\bar{I}_7^2] \le c \mathbb{E}[{\tilde{b}}(X_{t_i})^2 J_{t_i}^2] \le \mathbb{E}[{ \tilde{b}}(X_{t_i})^{2p}]^\frac{1}{p} \mathbb{E}[J_{t_i}^{2q}]^\frac{1}{q} \le c \Delta_{n,i}^{1 - \tilde{\varepsilon}},
\label{eq: bar I 7}
\end{equation}
where in the last inequality we have chosen $p$ big and $q$ next to $1$. {In particular, we have taken $1/q = 1 - \Tilde{\varepsilon}$.}
From \eqref{eq: bar I start}, \eqref{eq: bar I 4}, \eqref{eq: bar I 5} and \eqref{eq: bar I 7} it follows 
$$\mathbb{E}[A_{t_i}^2] \le c \Delta_{n,i}^{1 - \tilde{\varepsilon}}.$$
Concerning the fourth moment of $A_{t_i}$, as before we know from Proposition \ref{lemma: size A B E} that
\begin{equation}
\mathbb{E}[\bar{I}_1^4 + \bar{I}_2^4 + \bar{I}_3^4 + \bar{I}_6^4] \le c \Delta_{n,i}.
\label{eq: bar I fourth start}
\end{equation}
Acting as in \eqref{eq: bar I 4} we get
\begin{equation}
\mathbb{E}[\bar{I}_4^4] \le \frac{c}{\Delta_{n,i}^4} \Delta_{n,i}^3 \int_{t_i}^{t_{i + 1}} \mathbb{E}\left[(a^2(X_s) - a^2(X_{t_i}))^4\left(\sum_{j = 1}^M \lambda_s^{(j)} \right)^4\right] ds
\label{eq: bar I 4 fourth}
\end{equation}
$$\le \frac{c}{\Delta_{n,i}} \int_{t_i}^{t_{i + 1}} \mathbb{E}[(a^2(X_s) - a^2(X_{t_i}))^{4p}]^\frac{1}{p} \mathbb{E}\left[\left(\sum_{j = 1}^M \lambda_s^{(j)} \right)^{4q}\right]^\frac{1}{q} ds \le c \Delta_{n,i}^{1 - \tilde{\varepsilon}},$$
{ where we have chosen $1/p = 1 - \Tilde{\varepsilon}$}.
In the same way, acting as in \eqref{eq: bar I 5} we obtain
\begin{equation}
\mathbb{E}[\bar{I}_5^4] \le \frac{c}{\Delta_{n,i}^4} \Delta_{n,i}^3 \int_{t_i}^{t_{i + 1}} \sum_{j = 1}^M \mathbb{E}[(\lambda_s^{(j)} - \lambda_{t_i}^{(j)})^4] ds \le c \Delta_{n,i}.
\label{eq: bar I 5 fourth}
\end{equation}
We conclude the proof of the point 2 by observing that 
\begin{equation}
\mathbb{E}[\bar{I}_7^4] \le c \mathbb{E}[\tilde{b}(X_{t_i})^{4p}]^\frac{1}{p} \mathbb{E}[J_{t_i}^{4q}]^\frac{1}{q} \le c \Delta_{n,i}^{1 - \tilde{\varepsilon}},
\label{eq: bar I 7 fourth}
\end{equation}
by the boundedness of the moments of $\tilde{b}$ and Kunita inequality. \\

\textit{Point 2} \\
We observe that $B_{t_i}$ is defined in the same way in Section \ref{S: volatility} and Section \ref{S: volatility + jumps}. Therefore, the second point has already been showed in point 2 of Proposition \ref{lemma: size A B E}. \\

\textit{Point 3} \\
By the definition of $E_{t_i}$ it clearly follows $\mathbb{E}_i[E_{t_i}] = 0$. We now analyse 
\begin{equation}
\mathbb{E}_i[E_{t_i}^2] \le \frac{c}{\Delta_{n,i}^2} \mathbb{E}_i[Z_{t_i}^2 J_{t_i}^2] + \frac{c}{\Delta_{n,i}^2} \mathbb{E}_i[J_{t_i}^4 + \left(\int_{t_i}^{t_{i + 1}} a^2(X_{s^-}) \sum_{j = 1}^M \lambda_s^{(j)} ds\right)^2].
\label{eq: Eti2 start}
\end{equation}
We show that the first term in the right hand side of the equation \eqref{eq: Eti2 start} is negligible if compared to the second one.
By a conditional version of Holder, BDG and Kunita inequalities we get
\begin{equation}
\frac{c}{\Delta_{n,i}^2} \mathbb{E}_i[Z_{t_i}^2 J_{t_i}^2] \le \frac{c}{\Delta_{n,i}^2} \mathbb{E}_i[Z_{t_i}^{2p}]^{\frac{1}{p}} \mathbb{E}_i[J_{t_i}^{2q}]^\frac{1}{q}\le \frac{c}{\Delta_{n,i}^2} \Delta_{n,i} \Delta_{n,i}^\frac{1}{q}  \le c \Delta_{n,i}^{- \varepsilon},
\label{eq: Eti2 termine 1}
\end{equation}
for any $\varepsilon > 0$, { setting $1/q = 1 -{\varepsilon}$}.
To study the last term in the right hand side of \eqref{eq: Eti2 start} we recall it is $J_{t_i} = \int_{t_i}^{t_{i + 1}} a(X_{s^-}) \sum_{j = 1}^M d\tilde{N}_s^{(j)}$. Therefore, from conditional Kunita inequality, we have
\begin{eqnarray*}
\frac{c}{\Delta_{n,i}^2} \mathbb{E}_i\left[J_{t_i}^4 + \left(\int_{t_i}^{t_{i + 1}} a^2(X_{s^-}) \sum_{j = 1}^M \lambda_s^{(j)} ds\right)^2\right]
&\le & \frac{c}{\Delta_{n,i}^2} \mathbb{E}_i \left[\int_{t_i}^{t_{i + 1}} a^4(X_{s^-}) \sum_{j = 1}^M \lambda_s^{(j)} ds \right. \\
&& \left. + 2\left(\int_{t_i}^{t_{i + 1}} a^2(X_{s^-}) \sum_{j = 1}^M \lambda_s^{(j)} ds\right)^2\right] \\
&\le&  \frac{c a_1^4}{\Delta_{n,i}^2}(1 + \Delta_{n,i}) \int_{t_i}^{t_{i + 1}} \mathbb{E}_i\left[\sum_{j = 1}^M \lambda_s^{(j)}\right] ds, 
\end{eqnarray*}

where we have also used Jensen inequality on the last term here above, which is the reason why we get an extra $\Delta_{n,i}$. From the fourth point of Lemma \ref{lemma: incrementi} it follows that the equation here above is upper bounded by $\frac{c a_1^4}{\Delta_{n,i}} \sum_{j = 1}^M \lambda^{(j)}_{t_i}$, plus a negligible term. Replacing it and \eqref{eq: Eti2 termine 1} in \eqref{eq: Eti2 start} it follows
$$\mathbb{E}_i[E_{t_i}^2] \le \frac{c a_1^4}{\Delta_{n,i}} \sum_{j = 1}^M \lambda_{t_i}^{(j)} + c \Delta_{n,i}^{- \varepsilon} \le \frac{c a_1^4}{\Delta_{n,i}} \sum_{j = 1}^M \lambda_{t_i}^{(j)}, $$
where the last inequality is a consequence of the fact that $\lambda$ is always strictly more than zero.
Regarding the fourth moment of $E_{t_i}$, from Kunita, Holder and Jensen inequality we have
$$\mathbb{E}[E_{t_i}^4] \le \frac{c}{\Delta_{n,i}^4} \mathbb{E}[Z_{t_i}^{4p}]^\frac{1}{p} \mathbb{E}[ J_{t_i}^{4q}]^\frac{1}{q} + \frac{c}{\Delta_{n,i}^4} \mathbb{E}_i\left[J_{t_i}^8 + \left(\int_{t_i}^{t_{i + 1}} a^2(X_{s^-}) \sum_{j = 1}^M \lambda_s^{(j)} ds\right)^4\right] $$
$$ \le \frac{c}{\Delta_{n,i}^4}(\Delta_{n,i}^{2} \Delta_{n,i}^{1 - \varepsilon} + \Delta_{n,i} + \Delta_{n,i}^{3} \Delta_{n,i}) \le \frac{c}{\Delta_{n,i}^3}.$$
\textit{Point 4} \\
The result follows directly from the definition of $C_{t_i}$ and the boundedness of $a$ and of the moments of $\lambda$.

\end{proof}

\subsection{Proof of Lemma \ref{lemma: variance Cti}}
\begin{proof}
It is
$$C_{t_i} { \tilde{\psi}_l(X_{t_i})} = a^2(X_{t_i}) \sum_{j = 1}^M (\lambda^{(j)}_{t_i} - \mathbb{E}[\lambda^{(j)}_{t_i} | X_{t_i} ]) { \tilde{\psi}_l(X_{t_i})} = : f(X_{t_i}, \lambda_{t_i}). $$
Since
$$\text{Var}\left(\frac{1}{n} \sum_{i =0}^{n -1} f(X_{t_i}, \lambda_{t_i})\right) \le \frac{1}{n^2} \sum_{i =0}^{n -1}  \sum_{j =0}^{n -1} \text{Cov}(f(X_{t_i}, \lambda_{t_i}), f(X_{t_j},\lambda_{t_j})),$$
we need to estimate the covariance. \\
As explained in Section \ref{S: ergodicity} we know that, under our assumptions, the process $Z := (X, \lambda)$ is $\beta$- mixing with exponential decay. It means that there exists $ \gamma >0$ such that
$$\beta_X (t) \le \beta_Z (t) \le C e^{- \gamma t}.$$
If the process $Y$ is $\beta$- mixing, then it is also $\alpha$-mixing and so the following estimation holds (see Theorem 3 in Section 1.2.2 of \cite{Doukhan})
$$ |\text{Cov}(Y_{t_i}, Y_{t_j})| \le c \left \| Y_{t_i} \right \|_p \left \| Y_{t_j} \right \|_q \alpha^\frac{1}{r} (Y_{t_i}, Y_{t_j}) $$
with { $\alpha$ the coefficient of $\alpha$-mixing and} $p$, $q$ and $r$ such that $ \frac{1}{p} + \frac{1}{q} + \frac{1}{r} =1$.
Using that 
$$\alpha (Z_{t_i}, Z_{t_j}) \le \beta_Z (|{t_i}-{t_j}|) \le c  e^{- \gamma |{t_i}-{t_j}|},$$
 in our case the inequality here above becomes 
$$|\text{Cov}(f(X_{t_i}, \lambda_{t_i}), f(X_{t_j}, \lambda_{t_j}))| \le c { \left \| f(X_{t_i}, \lambda_{t_i}) \right \|_p \left \| f(X_{t_j}, \lambda_{t_j}) \right \|_q} e^{- \frac{1}{r} \gamma |t_i -t_j|}. $$
{
From the definition of $f$ and the boundedness of $a$ and the existence of moments of $\lambda$ we have
\begin{align*}
\left \| f(X_{t_i}, \lambda_{t_i}) \right \|_p & \le c \sum_{j = 1}^M \left \| (\lambda^{(j)}_{t_i} - \mathbb{E}[\lambda^{(j)}_{t_i} | X_{t_i} ]) \right \|_{p p_1} \left \| \tilde{\psi}_l(X_{t_i}) \right \|_{p p_2} \\
& \le  c \sum_{j = 1}^M  \left \| \tilde{\psi}_l(X_{t_i}) \right \|_{p p_2}, 
\end{align*}
with $p_1$ and $p_2$ such that $\frac{1}{p_1} + \frac{1}{p_2} = 1$. We remark that, as we are going to bound both the $L^p$ and the $L^q$ norm of $f(X_{t_i}, \lambda_{t_i})$, it seems natural to chose $p= q$, in order to repeat twice the same estimation. Then, as $\frac{1}{p} + \frac{1}{q} + \frac{1}{r} =1$ and we need $r > 0$, we obtain $p = q > 2$. We can then chose, for ${\varepsilon} > 0$ arbitrarily small, $p p_2 = 2 + {\varepsilon}$. It implies $p_1 = \frac{2 + \varepsilon}{2 + \varepsilon - p}$, which leads us to chose $2 < p = 2 + \tilde{\varepsilon} < 2 + \varepsilon$, for some $\tilde{\varepsilon} < \varepsilon$. Then, using that the $L^2$ norm of $\Tilde{\psi}_l$ is smaller than $1$ and that we can bound $\Tilde{\psi}_l(x)$ by $D_m$, we obtain
$$\left \| \tilde{\psi}_l(X_{t_i}) \right \|_{ 2 + \varepsilon} \le c \left \| \tilde{\psi}_l(X_{t_i}) \right \|_{\infty}^{\varepsilon} \left \| \tilde{\psi}_l(X_{t_i}) \right \|_{2} \le c D_m^{\varepsilon},$$
which provides 
$$\left \| f(X_{t_i}, \lambda_{t_i}) \right \|_p \le c M D_m^{\varepsilon}.$$
In a similar way, it is easy to see that
$$\left \| f(X_{t_i}, \lambda_{t_i}) \right \|_q \le c M D_m^{\varepsilon}.$$}
We now introduce a partition of $(0, T_n]$ {(where $T_n$ is the time horizon)} based on the sets $A_k:= ( k \frac{T_n}{n}, (k+1) \frac{T_n}{n}]$, for which $(0, T_n] = \cup_{k=0}^{n-1} A_k$. 
Now each point $t_i$ in $(0, T_n]$ can be seen as $t_{k, h}$, where $k$ identifies the particular set $A_k$ to which the point belongs while, defining $M_k$ as $|A_k|$, $h$ is a number in $\left \{ 1, \ldots , M_k \right \}$ which enumerates the points in each set. It follows
\begin{eqnarray*}
\frac{c}{n^2} \sum_{i =0}^{n -1}  \sum_{j =0}^{n -1} e^{- \frac{1}{r} \gamma |t_i -t_j|} &\le& \frac{c}{n^2} \sum_{k_1 =0}^{n -1}  \sum_{k_2 =0}^{n -1} \sum_{h_1 =1}^{M_{k_1}}  \sum_{h_2 =1}^{M_{k_2}} e^{- \frac{1}{r} \gamma |t_{k_1, h_1} -t_{k_2, h_2}|} \\
&\le&  \frac{c e^{\frac{{ 2 \gamma}}{r} \frac{T_n}{n}}}{n^2} \sum_{k_1 =0}^{n -1}  \sum_{k_2 =0}^{n -1} \sum_{h_1 =1}^{M_{k_1}}  \sum_{h_2 =1}^{M_{k_2}} e^{- \frac{1}{r} \gamma |k_1 - k_2| \frac{T_n}{n}}, 
\end{eqnarray*}
where the last inequality is a consequence of the following estimation: for each $k_1, k_2 \in \left \{ 0,\ldots , n-1 \right \}$ it is $|t_{k_1, h_1} -t_{k_2, h_2}| \ge |k_1 - k_2|\frac{T_n}{n} - \frac{{ 2}T_n}{n}$. \\
{ We remark here that, as we are considering the general case where the discretization step is not necessarily uniform, we can not replace $\frac{T_n}{n}$ with simply $\Delta_n$: we have to keep it like this and compare it with $\Delta_{min}$ and $\Delta_{max}$, which is equal to $\Delta_n$ by definition.} \\
Now we observe that the exponent does not depend on $h$ anymore, hence the last term here above can be upper bounded by $\frac{c e^{\frac{{ 2\gamma}}{r}  \frac{T_n}{n}}}{n^2} \sum_{k_1 =0}^{n -1}  \sum_{k_2 =0}^{n -1} M_{k_1} M_{k_2}e^{- \frac{1}{r} \gamma |k_1 - k_2|\frac{T_n}{n}} $. \\
Moreover, remarking that the length of each interval $A_k$ is $\frac{T_n}{n}$, it is easy to see that we can always upper bound $M_k$ with $\frac{T_n}{n} \frac{1}{\Delta_{min}}$, with $T_n = \sum_{i = 0}^{n - 1} \Delta_{n,i} \le n \Delta_{n}$ and so $M_k \le \frac{\Delta_{n}}{\Delta_{min}}$, that we have assumed bounded by a constant $c_1$. \\
Furthermore, still using that $T_n \le n \Delta_{n}$, we have $e^{\frac{{ 2\gamma}}{r}  \frac{T_n}{n}} \le e^{\frac{{ 2\gamma}}{r} \Delta_{n} } \le c$ .

To conclude, we have to evaluate
$\frac{c }{n^2} \sum_{k_1 =0}^{n -1}  \sum_{k_2 =0}^{n -1} e^{- \frac{1}{r} \gamma |k_1 - k_2| \frac{T_n}{n}}$. We define $j := k_1 - k_2$ and we apply a change of variable, getting 
\begin{eqnarray*}
\frac{c }{n^2} \sum_{k_1 =0}^{n -1}  \sum_{k_2 =0}^{n -1} e^{- \frac{1}{r} \gamma |k_1 - k_2|\frac{T_n}{n}} &\le &\frac{c }{n^2} \sum_{j = -(n -1)}^{n -1} e^{- \frac{1}{r} \gamma |j|\frac{T_n}{n}} |n - j| \le \frac{c }{n}  \sum_{j = -(n -1)}^{n -1} e^{- \frac{1}{r} \gamma |j| \Delta_{min}} \\
&\le&  \frac{c}{n (1 - e^{- \frac{1}{r} \gamma \Delta_{min}})} \le \frac{c}{T_n} , 
\end{eqnarray*}
as we wanted.

\end{proof}

\subsection{Proof of Lemma \ref{lemma: P omega B complementare}}
\begin{proof}
In order to estimate the probability of the complementary of the set $\Omega_B$, as defined in \eqref{eq: Omega B def}, we first of all observe that $\Omega_B^c \subset \cup_{j,k} 
\left\{ \sup_{t\in \mathcal{B}_{m,m'} } |t(U^*_{k, j})| \ge \tilde{c} n^{\varepsilon_0} D^{\frac{1}{2}} \right\}$.  Now we find an upper bound for the probability of $\Omega_B^c$ focusing on what happens for $j=1$ and $k = 0$. 
Recalling the definition of $U^*_{j,k}$ in \eqref{E:def_t_Ukl} and using that,
as $t\in \mathcal{B}_{m,m'}$ whose dimension is $D$, 
$
\left \| t \right \|_\infty \le c D^{\frac{1}{2}}$,
we can write, {for any $\epsilon >0$ arbitrarily small,}
$$\mathbb{P}( \sup_{t\in \mathcal{B}_{m,m'} }  |t(U^*_{0, 1})| \ge \tilde{c} n^{ \varepsilon_0}  D^{\frac{1}{2}}
) \le \mathbb{P}\left(\frac{1}{q_n} \sum_{k = 1}^{q_n} |B^*_{t_k} +C^*_{t_k} + E^*_{t_k}| \ge \tilde{c} n^{ \varepsilon_0}\right) $$
\begin{equation}
 \le \mathbb{P}\left(\frac{1}{q_n} \sum_{k = 1}^{q_n} |B^*_{t_k}| \ge \frac{\tilde{c}}{3} n^{ \varepsilon_0} \right) + \mathbb{P}\left(\frac{1}{q_n} \sum_{k = 1}^{q_n} |C^*_{t_k}| \ge \frac{\tilde{c}}{3} n^{ \varepsilon_0} \right) + \mathbb{P}\left(\frac{1}{q_n} \sum_{k = 1}^{q_n} |E^*_{t_k}| \ge \frac{\tilde{c}}{3} n^{ \varepsilon_0}\right).
\label{eq: proba omega B both}
\end{equation}
From the definition of $B$ it is 
\begin{equation}
\frac{1}{q_n} \sum_{k = 1}^{q_n} |B^*_{t_k}| \le \frac{c}{q_n \Delta_n} \sum_{k = 1}^{q_n} Z^2_{t_k} + c.
\label{eq: B proba}
\end{equation}
Moreover, using Markov inequality and the boundedness of $\sigma$,
\begin{eqnarray}
\mathbb{P}\left(|Z_{t_k}| \ge c \sigma_1 \Delta_n^\frac{1}{2} \log n\right) &=& \mathbb{P}\left(e^{\frac{|Z_{t_k}|}{\sigma_1 \sqrt{\Delta_n}}} \ge n^c\right) \le \frac{1}{n^c} \mathbb{E}\left[e^{\frac{|Z_{t_k}|}{\sigma_1 \sqrt{\Delta_n}}}\right] \nonumber\\
&\le &\frac{1}{n^c} \mathbb{E}\left[e^{\frac{c'}{\Delta_n \sigma_1^2} \int_{t_k}^{t_{k + 1}} \sigma^2(X_s) ds}\right] \le \frac{c'}{n^c}. 
\label{eq: study Z}
\end{eqnarray}
Therefore, as the constant $c$ in \eqref{eq: B proba} can be moved in the other side of the inequality in the first probability of \eqref{eq: proba omega B both} and so it turns out not being influential, the first probability of \eqref{eq: proba omega B both} is upper bounded by $\frac{q_n}{n^c}$, which is arbitrarily small.
Concerning the second term of \eqref{eq: proba omega B both}, we use Markov inequality and the fact that $C$ has bounded moments. We get, $\forall r \ge 1$,
$$\mathbb{P}\left(\frac{1}{q_n}  \sum_{k = 1}^{q_n} |C^*_{t_k}| \ge \frac{\tilde{c}}{3} n^{ \varepsilon_0}\right) \le { \sum_{k=1}^{q_n}} \mathbb{P}\left(|C^*_{t_k}| \ge \frac{\tilde{c}}{3} n^{ \varepsilon_0}\right) \le c {\sum_{k=1}^{q_n}}\frac{\mathbb{E}[|C^*_{t_k}|^r]}{n^{r { \varepsilon_0}}} \le \frac{c q_n}{n^{r{ \varepsilon_0}} }.$$
Regarding the third term of \eqref{eq: proba omega B both} we observe that, replacing the value of $q_n$ we get 
\begin{equation}
\mathbb{P}\left(\frac{1}{q_n} \sum_{k = 1}^{q_n} |E^*_{t_k}| \ge \frac{\tilde{c}}{3} n^{ \varepsilon_0}\right) = \mathbb{P}\left( \sum_{k = 1}^{q_n} |E^*_{t_k}| \ge \frac{\tilde{c}}{3} n^{ \varepsilon_0} \frac{\log n}{ \Delta_n}\right).
\label{eq: start E}
\end{equation}
We now recall that, from the definition of $E_{t_k}$ it is 
$$\sum_{k = 1}^{q_n} |E^*_{t_k}| \le  \left|\frac{2}{\Delta_n} \sum_{k = 1}^{q_n} Z_{t_k} J_{t_k}  \right| + \left| \frac{1}{\Delta_n} \sum_{k = 1}^{q_n} J_{t_k}^2\right | + \left|  \frac{1}{\Delta_n} \int_0^{t_{q_n}}a(X_{s^-}) \sum_{j = 1}^M \lambda^{(j)} (s) ds \right|$$
$$=: I_1 + I_2 + I_3. $$
The right hand side of \eqref{eq: start E} is upper bounded by 
$$\mathbb{P}\left(I_1 \ge \frac{\tilde{c}}{9} n^{ \varepsilon_0} \frac{\log n}{ \Delta_n} \right) + \mathbb{P}\left(I_2 \ge \frac{\tilde{c}}{9} n^{ \varepsilon_0} \frac{\log n}{ \Delta_n} \right)  + \mathbb{P}\left(I_3 \ge \frac{\tilde{c}}{9} n^{\varepsilon_0} \frac{\log n}{ \Delta_n} \right). $$
Concerning the first one, we observe it is 
$$I_1 \le \frac{1}{\Delta_n} \sum_{k = 1}^{q_n} (Z_{t_k}^2 + J_{t_k}^2) = I_{1,1} + I_{1,2}.$$
The probability that $I_{1,1}$ is bigger than $\frac{\tilde{c}}{9} n^{ \varepsilon_0} \frac{\log n}{ \Delta_n} $ is arbitrarily small as a consequence of \eqref{eq: study Z}. $I_{1,2}$ is instead equal to { $I_2$} and so it is enough to study such a term. From Markov, Holder, BDG and Kunita inequalities we have 
$$\mathbb{P}\left(I_3 \ge \frac{\tilde{c}}{9} n^{ \varepsilon_0} \frac{\log n}{ \Delta_n} \right) \le \frac{\mathbb{E}[(I_3)^r]}{(n^{ \varepsilon_0} \log n \Delta_n^{-1})^r} \le \frac{c \Delta_n^{- r} t_{q_n}^{r}}{(n^{ \varepsilon_0} \log n \Delta_n^{-1})^r} \le \frac{c}{n^{{ \varepsilon_0} r}},$$
where we underline that the order of $t_{q_n}$ is $c q_n \Delta_n = c \frac{\log n}{\Delta_{min}} \Delta_n  {\  \le \ } c \log n$.
It is arbitrarily small. Concerning $I_2$, we want to estimate $\mathbb{P}(\sum_{k = 0}^{q_n - 1} J_{t_k}^2 \ge \frac{c}{9} n^{ \varepsilon_0} \log n) $. We now consider two different possibilities, starting from the definition of the following set
$$A:= \left \{ \exists \tilde{k}\in \left \{ 0,\ldots , q_n - 1 \right \} \, \mbox{ such that } J_{t_{\tilde{k}}}^2 \ge n^\frac{{ \varepsilon_0}}{2} \right \}.$$
Then 
$$\mathbb{P}\left(\sum_{k = 0}^{q_n - 1} J_{t_k}^2 \ge \frac{c}{9} n^{ \varepsilon_0} \log n\right) = \mathbb{P}\left(\sum_{k = 0}^{q_n - 1} J_{t_k}^2 \ge \frac{c}{9} n^{ \varepsilon_0} \log n, A\right) + \mathbb{P}\left(\sum_{k = 0}^{q_n - 1} J_{t_k}^2 \ge \frac{c}{9} n^{ \varepsilon_0} \log n, A^c\right).$$
We observe that Markov inequality and Kunita inequality yield
$$\mathbb{P}\left(\sum_{k = 0}^{q_n - 1} J_{t_k}^2 \ge \frac{c}{9} n^{ \varepsilon_0} \log n, A\right) \le \mathbb{P}(A) \le { \sum_{k = 0}^{q_n - 1}} \frac{\mathbb{E}[(J_{t_{\tilde{k}}})^{2 r}]}{n^\frac{{ \varepsilon_0} r}{2} } \le \frac{\Delta_n q_n}{n^{\frac{{ \varepsilon_0} r}{2} }} = \frac{c \log n}{n^{\frac{{ \varepsilon_0} r}{2} }} ,$$
which is arbitrarily small by the arbitrariness of $r$. We remark that on $A^c$, for every $k \in \left \{ 0,\ldots , q_n - 1 \right \}$, it is $J_{t_k}^2 < n^\frac{\varepsilon}{2} $. Therefore, to have the sum of them bigger than $\frac{c}{9} n^\varepsilon \log n$ we should have at least $\frac{c}{9} \log n n^\frac{\varepsilon}{2} $ jumps. Hence, denoting as $\Delta N_q$ the number of jumps in $[0, t_{q_n}]$, we have 
\begin{eqnarray*}
\mathbb{P}\left(\sum_{k = 0}^{q_n - 1} J_{t_k}^2 
  \geq \frac{c}{9} n^{ \varepsilon_0} \log n, A^c \right) &\le & \mathbb{P}\left(\Delta N_q > \frac{c}{9} 
  n^\frac{ { \varepsilon_0} }{2} \log n \right) \le c \frac{\mathbb{E}[(\Delta N_q )^r]}{(n^\frac{ { \varepsilon_0} }{2} \log n)^r} \\
&\leq & \frac{ { c\left(t_{q_{n}}^{r}+t_{q_{n}}\right)}}{(n^\frac{{\varepsilon_0}}{2} \log n )^r }
{ \le
\frac{c ((\log n)^r + \log n)}{(n^\frac{\varepsilon_0}{2} \log n )^r} 
\le \frac{c}{ {  n^\frac{\varepsilon_0 r}{2}}}
},
\end{eqnarray*}
where again we have used Markov inequality and we got a quantity arbitrarily small { choosing $r \ge 1$ large enough}. We put all the pieces together and we observe we can choose in particular $r$ for which
$$\mathbb{P}\left(\frac{1}{q_n} \sum_{k = 1}^{q_n} |E^*_{t_k}| \ge \frac{\tilde{c}}{3} n^{ \varepsilon_0} \right) \le \frac{c}{n^{\modar 5}}.$$
In the same way it is possible to choose $r$ and $\tilde{c}$ such that
$$\mathbb{P}(\Omega_B^c) 
\le  \sum_{j=0,1;k\in\{1,\dots,p_n\}}
\P \left( \sup_{t\in \mathcal{B}_{m,m'} } |t(U^*_{k, j})| \ge \tilde{c} n^{ \varepsilon_0} D^{\frac{1}{2}} \right)
\le \frac{cp_n}{n^5} \le
\frac{c}{n^4}. $$
\end{proof}
{\modar
	\subsection{Proof of Lemma \ref{lemma: variance U00}}
	\begin{proof}
		We observe that for any $t \in \mathcal{B}_{m, m'}$, by \eqref{E:def_t_Ukl} and Proposition \ref{prop: size A B C E}, it is
		\begin{align*}\E[t(U^*_{0,0})^2]&=\text{Var}(U^*_{0,0}) \le \frac{4}{q_n^2} \sum_{l = 1}^{q_n} \mathbb{E}\left[t^2(X^*_{t_l}) \mathbb{E}_l[B^{*,2}_{t_l} + E^{*,2}_{t_l}]\right]
			+ 4\text{Var}\left( \frac{1}{q_n}\sum_{l = 1}^{q_n} t(X^*_{t_l})C^{*}_{t_l} \right)
			\\
			&=: V_1+V_2.
		\end{align*}
		By the second and the third points of Proposition \ref{prop: size A B C E}, we can upper bound $V_1$ as
		\begin{eqnarray*}
			V_1 &\leq &\frac{c}{q_n^2} \sum_{l = 1}^{q_n} \mathbb{E}\left[t^2(X^*_{t_l})  (\sigma_1^2 + \frac{a_1^4}{\Delta_n} \sum_{j = 1}^M \lambda^{(j)}_{t_l})\right] \nonumber\\
			&\le &\frac{c}{q_n^2} \sum_{l = 1}^{q_n} \mathbb{E}\left[t^{2p}(X^*_{t_l})\right]^{\frac{1}{p}} \mathbb{E}\left[\left(\sigma_1^2 + \frac{a_1^4}{\Delta_n} \sum_{j = 1}^M \lambda^{(j)}_{t_i}\right)^q\right]^\frac{1}{q},
		\end{eqnarray*}
		where we have used H\"older inequality with $q$ big and $p$ next to $1$. We can see $t^{2p}(X^*_{t_l})$ as 
		$$t^{2 +(2p-2)}(X^*_{t_l}) = t^{2}(X^*_{t_l}) t^{(2p-2)}(X^*_{t_l}) \le \left \| t \right \|^{2 p - 2}_\infty t^{2}(X^*_{t_l}). $$
		From Assumption \ref{ass: subspace}, $\displaystyle \left \| t \right \|^{\frac{2 p - 2}{p}}_\infty \le cD_m^{\frac{2 p - 2}{2p}} \le c D_m^{\delta}$, for any $\delta$ arbitrarily small, as $p$ has been chosen next to $1$.
		Using also the boundedness of the moments 
		of $\lambda$ it follows that 
		$$
		V_1 \le 
		\frac{D_m^\delta}{q_n^2} \frac{c}{\Delta_n} q_n = 
		\frac{c D_m^\delta}{q_n \Delta_n}.$$
		
		Using the same arguments as in the proof of Lemma \ref{lemma: variance Cti}, { remarking that the sum over $n$ is now replaced by the sum over $q_n$ and that $t$ now plays the same role as $\Tilde{\psi}$, being such that its $L^2$ norm is smaller than $1$ and it is bounded by $D_m$,} we can show  
		$$
		V_2 \le \frac{c}{q_n \Delta_n} \| t \|_\infty^{2\delta} \le \frac{c}{q_n \Delta_n} D_m^{2\delta},
		$$
		for any $\delta>0$. It concludes the proof of the lemma.
	\end{proof}
}

{ \section*{Acknowledgements}
The authors would like to thank the anonymous referees
for their helpful remarks that helped to improve the  first version of the paper.}

\end{document}